\documentclass[11pt, reqno]{amsart}

\usepackage{amsmath, amsthm, amsfonts, amssymb, color}
\usepackage{mathrsfs}
\usepackage{amsfonts, amsmath}
\usepackage{amsmath,amstext,amsthm,amssymb,amsxtra}
\usepackage{hyperref}
\hypersetup{colorlinks=true,linkcolor=black, urlcolor=blue, citecolor=blue}

\allowdisplaybreaks

\newcommand{\EQ}[1]{\begin{equation}\begin{split} #1 \end{split}\end{equation}}
\newcommand{\EQN}[1]{\begin{equation*}\begin{split} #1 \end{split}\end{equation*}}

\newcommand{\CAS}[1]{\begin{cases} #1 \end{cases}}

\usepackage[margin=3cm, a4paper]{geometry}

\newcommand {\brk}[1]{\left(#1\right)}

\newcommand {\les}{\lesssim}

\newcommand {\C}{{\mathbb C}}

\newcommand {\F}{\mathcal{F}}

\newcommand {\rb}{\rangle}
\newcommand {\lb}{{\langle}}

\newcommand {\norm}[1]{\left\|#1\right\|}

\newcommand {\R}{\mathbb R}

\newcommand {\supp}{\mathrm{supp}}

\newcommand {\Sw}{\mathcal{S}}

\newcommand {\Z}{\mathbb Z}

\newcommand {\vanish}[1]{\relax}

\newcommand {\TT}{{\mathcal T}}

\newcommand {\RR}{{\mathcal R}}

\newcommand{\wh}{\widehat}

\newcommand {\Wone}{W_{1}}
\newcommand {\Wtwo}{W_{2}}
\newcommand {\Wthr}{W_{3}}
\newcommand {\Uone}{U_{1}}
\newcommand {\Utwo}{U_{2}}
\newcommand {\Uthr}{U_{3}}
\newcommand {\aij}{a_{ij}}
\newcommand {\Akthr}{A_{k_{2}}^{(3)}}
\newcommand {\Akkthr}{A_{k,k_{2}}^{(4)}}
\newcommand {\Bkkthr}{\tilde{A}_{k,k_{2}}^{(4)}}
\newcommand {\Akone}{A_{k_{1}}^{(2)}}
\newcommand {\aktwo}{A_{k_{2}}^{(1)}}
\newcommand {\bkone}{\frac{\psi_{k_{1}}(\xi-\eta)}{\phi_{1}(\xi,\eta,\zeta)}}
\newcommand {\Phione}{\Phi_{1}}
\newcommand {\Phitwo}{\Phi_{2}}
\newcommand {\Phithr}{\Phi_{3}}
\newcommand {\Phifo}{\Phi_{4}}
\newcommand {\Phifi}{\Phi_{5}}
\newcommand {\Phisi}{\Phi_{6}}
\newcommand {\phione}{\phi_{1}}
\newcommand {\Gammaone}{\Gamma_{ij}^{(1)}}
\newcommand {\Gammatwo}{\Gamma_{ij}^{(2)}}
\newcommand {\Gammathree}{\Gamma_{ij}^{(3)}}
\newcommand {\MI}{\mathcal{I}}
\newcommand {\MJ}{\mathcal{J}}
\newcommand {\Hs}{H^{2}}
\newcommand {\PW}{F}
\newcommand {\PU}{G}

\newtheorem{theorem}{Theorem}[section]
\newtheorem{lemma}[theorem]{Lemma}
\newtheorem{proposition}[theorem]{Proposition}
\newtheorem{corollary}[theorem]{Corollary}

\theoremstyle{definition}
\newtheorem{definition}[theorem]{Definition}
\newtheorem{remark}[theorem]{Remark}

\numberwithin{equation}{section}
\protected\def\ignorethis#1\endignorethis{}
\let\endignorethis\relax

\begin{document}

\title[Global solutions for quadratic NLS]{Global solutions to 3D quadratic nonlinear Schr\"odinger-type equation}

\author[Z. Guo]{Zihua Guo}
\address{School of Mathematics, Monash University, VIC 3800, Australia }
\email{zihua.guo@monash.edu}
	
\author[N. Liu]{Naijia Liu}
\address{School of Mathematics, Sun Yat-sen University
Guangzhou, 510275, China}
\email{liunj@mail2.sysu.edu.cn}

\author[L. Song]{Liang Song}
\address{School of Mathematics, Sun Yat-sen University
Guangzhou, 510275, China}
\email{songl@mail.sysu.edu.cn}


	
\begin{abstract}
We consider the Cauchy problem to the 3D fractional Schr\"odinger equation with quadratic interaction of $u\bar u$ type.  We prove the global existence of solutions and scattering properties for small initial data.  For the proof, one novelty is that we combine the normal form methods and the space-time resonance methods.  Using the normal form transform enables us to have more {flexibility} in designing the resolution spaces so that we can control various interactions. It is also convenient for the final data problem.
\end{abstract}
	
\maketitle	

\tableofcontents

\section{Introduction}\label{sec1}
\setcounter{equation}{0}

In this paper, we study the Cauchy problem to the 3D quadratic nonlinear Schr\"odinger-type equation
\EQ{\label{e1.1}
\CAS{(\partial_{t}+iD^{\alpha})u = \rho u\bar{u}, \\
{u(0) = u_{0}},
}
}
where $u(t,x):\R\times \R^3 \to \C$ is the unknown function, $u_0$ is a given data, $\alpha\in (1,2)$, and $D^{\alpha}$ is defined via the Fourier multiplier:
$$
D^{\alpha}f(x):=(2\pi)^{-3/2}\int_{\R^3} e^{ix\xi}|\xi|^{\alpha}\hat{f}(\xi)d\xi.
$$
{Here $\rho\in \C$} and plays no role in this paper.  We may {assume $\rho=1$.}

When $\alpha=2$, we have $D^{\alpha}=-\Delta$. Then \eqref{e1.1} becomes the quadratic nonlinear Schr\"odinger equation (NLS). The global existence of the solutions for the following NLS
\EQ{\label{eq:NLS}
(\partial_{t}-i\Delta)u =& F(u), \quad (t,x)\in \R\times \R^d
}
where $|F(u)|\sim |u|^{p+1}$, in particular the Hamiltonian nonlinearity $F(u)=i\mu|u|^pu$, has been extensively studied.  We do not attempt to exhaust the list of literature of the studies, but refer to the nice introduction of \cite{NY}.
The NLS \eqref{eq:NLS} is invariant under the scaling transform: for $\tau>0$
\EQ{
u(t,x)\to \tau^{2/p} u(\tau^2 t, \tau x).
}
The critical Sobolev space is $\dot H^{s_c}$ where $p=\frac{4}{d-s_c}$ in the sense that the norm of $\dot H^{s_c}$ is invariant under the above scaling transform. There are three important indices for the study of global existence of solutions: mass-critical index $p_0=\frac{4}{d}$, Strauss index $p_1(d)=\frac{2-d+\sqrt{d^2+12d+4}}{2d}$ and scattering-critical index $p_2(d)=\frac{2}{d}$.  Note that $p_1(3)=1$.

For the quadratic nonlinear terms $F(u)=\alpha_1 u^2+\alpha_2 u\bar u+\alpha_3 \bar u \bar u$, which appear in many physical models, have the same scaling invariance as the nonlinear term $i|u|u$. However, the methods for $i\mu|u|u$ do not work in general for $F(u)$, as $i\mu|u|u$ has some special gauge-invariant and Hamiltonian structures. On the other hand, compared to the nonlinearity $i|u|u$, $F(u)$ has better algebraic structures and smoothness. Some new methods are developed to exploit the nonlinear interaction structures. When $\alpha_2=0$, small data global existence was proved by Germain-Masmoudi-Shatah \cite{GMS} via the space-time resonance method. When $\alpha_2\neq 0$, the space-time resonance structure of $u\bar u$ is worse (as explained below) and to our knowledge the small data global-existence is still an open question.  Ikeda-Inui \cite{IInui} showed the existence of blow-up solutions for a class of small $L^2$ initial data, which decays at rate $\frac{1}{|x|^{2-\epsilon}}$ as $|x|\to \infty$, $0<\epsilon<1/2$.  In \cite{GHay}, Ginibre and Hayashi proved the almost global existence using the vector fields method.  See \cite{HNau}, \cite{Kawa}, \cite{W}, and \cite{Su} for related results. In particular, in \cite{W}, Wang gave an alternative proof using the space-time resonance method. The difficulty for showing the global existence is some logarithmic divergence problem due to the $high\times high\to 0$ interactions. Indeed, it was shown in \cite{W} that small data global existence holds if $u\bar u$ is replaced by some similar nonlinearity $Q(u,\bar u)$ with some null-structure, e.g. $Q(u,\bar u)\sim D^{\epsilon}(u\bar u)$. See \cite{Su} for similar results for general 3D quadratic systems.

The purpose of this paper is to study the small data global existence for the equation \eqref{e1.1}.  We keep the nonlinearity $u\bar u$, but with a general dispersion. Our results show that the Schr\"odinger dispersion (in particular for low frequency) is really critical for $u\bar u$.
Our main results can be roughly described as follows (We refer to Theorem \ref{thm:main} for the precise version):

\begin{theorem}\label{thm1.1}
Let $\alpha\in (1,2)$. Assume that $u_0$ is sufficiently small in suitable space, then there exists a unique global solution $u$ to \eqref{e1.1}.  Moreover, {$\norm{u(t)}_{L^\infty}\les (1+|t|)^{-1-}$} and scattering holds.
\end{theorem}

In the rest of the introduction, we would like to describe our methods of proof. Consider the general 3D quadratic dispersive system
\EQ{\label{eq:qDS}
(\partial_t+L_1)u_1=& B_1(u_1,u_2)\\
(\partial_t+L_2)u_2=& B_2(u_1,u_2)\\
(u_1,u_2)|_{t=0}=&(f_1,f_2)
}
where $L_j f=\F^{-1}i\omega_j(\xi)\F f$, $j=1,2$, with dispersion $\omega_j(\xi): \R^3\to \R$, and $B_j(u_1,u_2)$ are the Coifman-Meyer bilinear Fourier multiplier operators
\EQ{
\F[B_j(u_1,u_2)](\xi)=\int_{\xi=\xi_1+\xi_2} m_j(\xi_1,\xi_2)\wh {u_1}(\xi_1)\wh{u_2}(\xi_2)d\mu, \quad j=1,2.
}
Here $u_j$ could be replaced by its conjugate $\bar u_j$.

It is now well-known that the nonlinear interaction structures of the equation \eqref{eq:qDS} play a decisive role on the behaviour of its solutions, e.g. well-posedness and large time behaviour. A  powerful approach is the perturbation method. That is, to view the nonlinearity as a perturbation to the linear equations under proper sense.  This requires very delicate work to design the suitable topology. Consider the equivalent integral equation of \eqref{eq:qDS}
\EQ{\label{eq:qDS-int}
u_j(t)=S_j(t)f_j+\int_{0}^{t} W_j(t-s)B_j(u_1,u_2)(s)ds, \quad j=1,2,
}
where $S_j(t)=e^{-tL_j}=\F^{-1} e^{-it\omega_j(\xi)}\F$. Define the sequence of iteration: $u^{(0)}=0$ and
\EQ{
u_j^{(n+1)}(t)=S_j(t)f_j+\int_{0}^{t} S_j(t-s)B_j(u_1^{(n)},u_2^{(n)})(s)ds, \quad n\geq 1.
}
We would like to derive some compactness properties of the sequence $\{u_j^{(n)}\}$. We can see some nonlinear interactions through the second iteration. The second iteration
\EQ{
u_j^{(2)}(t)=S_j(t)f_j+\int_{0}^{t} S_j(t-s)B_j(u_1^{(1)},u_2^{(1)})(s)ds
}
implies 
\EQ{
e^{it \omega_j(\xi)}\F [u_j^{(2)}](t,\xi)-\wh{f_j}(\xi)=\int_0^t \brk{\int_{\R^3} e^{is \phi(\xi,\eta)}\wh{f_1}(\eta)\wh{f_2}(\xi-\eta)\, d\eta} ds
}
where $\phi(\xi,\eta)=\omega_j(\xi)-\omega_1(\eta)-\omega_2(\xi-\eta)$ is the resonance function.
The right-hand side is a bilinear oscillatory integral operator with a phase function $s\phi(\xi,\eta)$.  By the theory of oscillatory integrals (e.g. see \cite{Stein}), the stationary sets play a crucial role:
\begin{itemize}
    \item $\RR=\{(\xi,\eta): \phi(\xi,\eta)=0\}$

    \item $\TT=\{(\xi,\eta): \nabla_{\eta} \phi(\xi,\eta)=0\}$
\end{itemize}
In general, when $\RR\cap \TT$ is very small, the equation has good nonlinear interactions.  When $\RR\cap \TT$ is very large, the equation has bad nonlinear interactions. For example, when $\omega_1(\xi)=\omega_2(\xi)=\xi$, $\RR=\TT=\R^3\times \R^3$, then we do not have any nonlinear oscillations.

Many tools have been developed to exploit the nonlinear oscillations. The first one is the Bourgain's $X^{s,b}$ method.  This method (and its relatives) is now a powerful and standard tool for low-regularity well-posedness and small data global existence.  See \cite{Tao-kz} for local-in-time analysis, and \cite{Tao-gkdv} for global-in-time analysis.  According to Tao \cite{Tao-kz}, $\RR$ is called the {\it resonance} and $\TT$ is called the {\it coherence}.  The non-resonant and non-coherent structures lead to extra smoothing effects and stronger decay.  To exploit that, some very delicate harmonic analysis tools (e.g. for transversality) enter into play, in particular when data only belongs to Sobolev space, that is $\wh f_1, \wh f_2$ has no smoothness.

Another approach to exploit the nonlinear oscillation is more elementary and straightforward. On $\R^3\times \R^3 \setminus (\RR\cap \TT)$, one can integrate by parts either in time variable or in spatial variables. This technique was explicitly used by Gustafson-Nakanishi-Tsai in \cite{GNT} (e.g. Section 10), where they proved global existence and scattering for the Gross-Pitaevskii equation in three dimensions with small data in weighted Sobolev spaces. This technique was also systematically developed by Germain-Masmoudi-Shatah \cite{GMS} and is now known as the space-time resonance method.  According to \cite{GNT} and \cite{GMS}, $\RR$ is referred as time-resonance and $\TT$ is referred as space-resonance. This method is powerful for obtaining global existence of solutions for many physical models for small and nice data.  In particular, the integration by parts in spatial variables will {inevitably} require the data in weighted Sobolev spaces (namely, require that $\wh{f_j}$ has some smoothness).
The integration by parts in time variables still works for Sobolev spaces, and is closely related to the method of normal form transform introduced by Shatah \cite{Shatah}.  In \cite{GN}, the first-named author and Nakanishi introduced the combination of the normal form transform and generalized Strichartz estimates to obtain small data scattering in Sobolev spaces for 3D quadratic dispersive systems.

In this paper, we combine the normal form transform and the space-time resonance method. We use the normal form transform to replace (some) integration-by-parts in time in the space-time resonance method.  More precisely, we write the equation \eqref{eq:qDS} as
\EQ{
(\partial_t + L_j)u_j=B_{j, R}(u_1,u_2)+B_{j,NR}(u_1,u_2)
}
where $B_{j,R}(u_1,u_2)$ denotes the resonant terms and  $B_{j,NR}(u_1,u_2)$ denotes the non-resonant terms (where $\phi(\xi,\eta)$ is large). Then we perform a normal form transform $u_j=w_j+\Omega_j(u_l,u_k)$ and get an equivalent system
\EQ{\label{eq:qDS2}
(\partial_t + L_j)w_j=& B_{j, R}(u_1,u_2)+{\mbox{Cubic terms}}(u_l,u_k, w_m)\\
u_j=&w_j+\Omega_j(u_l,u_k).
}
The advantage of doing so is:
\begin{itemize}
\item One has more flexibilities for choosing the function spaces for $w_j,u_j$.  The spaces for $u_j$ are usually weaker than that for $w_j$. For some problems, it is necessary as $w_j$ behaves (e.g. decay) better than $u_j$ due to the cancellation between $u_j$ and $\Omega_j(u_l,u_k)$

\item The map $u_j\to w_j$ is one-to-one for small data. One can rewrite the equation \eqref{eq:qDS2} further by plug-in $u_j$, so that in the first equation $u_j$ is only involved in higher order nonlinearity. In this way, one can use much weaker spaces for $u_j$ than $w_j$.  This is exactly what we use for this paper (see Section 2).

\item For \eqref{eq:qDS2}, it is convenient and easier to deal with the final data problem.  The final data problem (namely construction of wave operator) of \eqref{e1.1} is easier to handle. See \cite{GNT1} Section 2 for the case $\alpha=2$ and their methods may also work for some $\alpha<2$. We revisit the final data problem in Section \ref{sec:final data} using our approach.  Different from the initial data problem, the time interval is now $[0,\infty]$.
\end{itemize}

\section{Normal form transform and resolution spaces}

For $X,Y\geq0$, $X\les Y$ means that there exists a constant $C>0$ such that $X\leq CY$.  $X\sim Y$ means $X\les Y$ and $Y\les X$. {In particular, all constants in this paper are independent of $t,k,k_{1},k_{2}$ but may depend on $\alpha,\lambda$.}

We use $\hat{u}$ or $\F u$ to denote the standard Fourier transform
\EQN{
\F u(\xi):=(2\pi)^{-3/2}\int_{\R^3}e^{-ix\xi}u(x)dx.
}
We also use $\mathcal{F}_x u$, $\mathcal{F}_t u$ or $\mathcal{F}_{t, x} u$ to  denote the Fourier transform with specified variables.
Define
$$
e^{itD^{\alpha}}u(x):=(2\pi)^{-3/2}\int_{\R^3} e^{i(x\xi+t|\xi|^{\alpha})}\hat{u}(\xi)d\xi.
$$
Let $\varphi\in C_0^\infty(\R)$ be a real-valued, nonnegative, even, and radially decreasing function such that $\supp\  \varphi\subset [-5/4, 5/4]$ and $\varphi\equiv 1$ in $[-1, 1]$. Let $\psi(\xi):=\varphi(|\xi|)-\varphi(2|\xi|)$.
For $k\in \Z$,  define $\psi_k(\xi):=\psi(2^{-k}\xi)$, $\psi_{\leq k}(\xi):=\varphi(2^{-k}|\xi|)$ and the Littlewood-Paley projectors:
\EQN{\widehat{P_kf}(\xi):=\psi(2^{-k}|\xi|)\hat{f}(\xi), \quad \widehat{P_{\leq k}f}(\xi):=\varphi(2^{-k}|\xi|)\hat{f}(\xi).}

Define $a_{HH}(\xi,\eta):=\sum_{\substack{|k_{1}-k_{2}|<10\\
k_{1},k_{2}\in \Z}}\psi_{k_{1}}(\xi-\eta)\psi_{k_{2}}(\eta)$, $a_{HL}(\xi,\eta):=\sum_{k\in\Z}\psi_{k}(\xi-\eta)\psi_{\leq k-10}(\eta)$ and $a_{LH}(\xi,\eta):=\sum_{k\in\Z}\psi_{\leq k-10}(\xi-\eta)\psi_{k}(\eta)$.
Define
\begin{align*}
(u\bar v)_{X}(t,x):=(2\pi)^{-3/2}\int_{\R^3} \int_{\R^3} e^{ix\xi}a_{X}(\xi,\eta)\hat{u}(t,\xi-\eta)\hat{\bar{v}}(t,\eta) d\xi d\eta,
\end{align*}
where $X\in \{HH, HL, LH\}$. Then we can decompose $u\bar v$ as
\EQ{\label{e2.1}
u\bar v=(u\bar v)_{HH}+(u\bar v)_{HL}+(u\bar v)_{LH}.
}
Since $(u\bar v)_{LH}$ is non-resonant, we will use normal form transform to remove this term.  More precisely,  let
\begin{align}\label{e2.2}
B(u,v)(t,x):=(2\pi)^{-3/2}\int_{\R^3}\int_{\R^3} e^{ix\xi}\phi(\xi,\eta)^{-1}a_{LH}(\xi,\eta)\hat{u}(t,\xi-\eta)\hat{{v}}(t,\eta)d\xi d\eta,
\end{align}
where
\EQ{\label{e2.3}
\phi(\xi,\eta):=|\xi|^{\alpha}-|\xi-\eta|^{\alpha}+|\eta|^{\alpha}.
}
We define a normal form transform
\begin{align}\label{e2.4}
w=u+iB(u,\bar u).
\end{align}

\begin{lemma}\label{lem2.1}
Let $I\subseteq\R$ be an interval. Suppose that $u$ satisfies $(\partial_{t}+iD^{\alpha})u = u\bar{u}$ for all $t\in I$, then $(w,u)$ satisfies
\EQ{\label{e2.5}
(\partial_{t}+iD^\alpha)w&= (w\bar{w})_{HH+HL}+i[w \overline{B(u,\bar u)}]_{HH+HL}-i[B(u,\bar u)\bar u]_{HH+HL} \\
&\hskip 2.57cm +iB(|u|^2,\bar u)+iB(u,|u|^2), \\
u&=w-iB(u, \bar u)
}
for all $t\in I$.
\end{lemma}

\begin{proof}
{Let $t,t_{0}\in I$.}
By Duhamel's formula and \eqref{e2.1}, we have
\EQN{
e^{it|\xi|^{\alpha}}\hat{u}(t,\xi)
=e^{it_{0}|\xi|^{\alpha}}\hat{u}(t_{0},\xi)+\int_{t_{0}}^{t}e^{is|\xi|^{\alpha}}(u\bar{u})_{HH+HL+LH}^\wedge(s,\xi)ds.
}
The term $(u\bar u)_{LH}$ is non-resonant.  Indeed,
\EQN{
\int_{t_{0}}^{t}e^{is|\xi|^{\alpha}}(u\bar{u})_{LH}^\wedge(s,\xi)ds
=\int_{t_{0}}^{t}\int_{\R^3} e^{is\phi(\xi,\eta)}a_{LH}(\xi,\eta)e^{is|\xi-\eta|^{\alpha}}\hat{u}(s,\xi-\eta)e^{-is|\eta|^{\alpha}}\hat{\bar{u}}(s,\eta)dsd\eta
}
and we have $|\phi|\sim |\xi|^{\alpha}$ on $\supp(a_{LH})$. By integration by parts in $s$ and the first equation of \eqref{e1.1}, we obtain
\begin{align*}
&i\int_{t_{0}}^{t}e^{is|\xi|^{\alpha}}(u\bar{u})_{LH}^\wedge(s,\xi)ds\\
=&e^{it|\xi|^{\alpha}}\int_{\R^3} \frac{a_{LH}(\xi,\eta)}{\phi(\xi,\eta)}\hat{u}(t,\xi-\eta)\hat{\bar{u}}(t,\eta)d\eta
-e^{it_{0}|\xi|^{\alpha}}\int_{\R^3}\frac{a_{LH}(\xi,\eta)}{\phi(\xi,\eta)}\hat{u}(t_{0},\xi-\eta)\hat{\bar{u}}(t_{0},\eta)d\eta\\
&-\int_{t_{0}}^{t}\int_{\R^3} e^{is\phi(\xi,\eta)}\frac{a_{LH}(\xi,\eta)}{\phi(\xi,\eta)}e^{is|\xi-\eta|^{\alpha}}(|u|^{2})^\wedge(s,\xi-\eta)e^{-is|\eta|^{\alpha}}\hat{\bar{u}}(s,\eta)dsd\eta\\
&-\int_{t_{0}}^{t}\int_{\R^3} e^{is\phi(\xi,\eta)}\frac{a_{LH}(\xi,\eta)}{\phi(\xi,\eta)}e^{is|\xi-\eta|^{\alpha}}\hat{u}(s,\xi-\eta)e^{-is|\eta|^{\alpha}}(|u|^{2})^\wedge(s,\eta)dsd\eta\\
=&\  e^{it|\xi|^{\alpha}}B(u,\bar{u})^\wedge(t,\xi)
-e^{it_{0}|\xi|^{\alpha}}B(u,\bar{u})^\wedge(t_{0},\xi)-\int_{t_{0}}^{t}e^{is|\xi|^{\alpha}}B(|u|^{2},\bar{u})^\wedge(s,\xi)ds\\
&-\int_{t_{0}}^{t} e^{is|\xi|^{\alpha}}B(u,|u|^{2})^\wedge(s,\xi)ds,
\end{align*}
which, together with \eqref{e2.4}, implies
\EQN{
e^{it|\xi|^{\alpha}}\hat{w}(t,\xi)
&=e^{it_{0}|\xi|^{\alpha}}\hat{w}(t_{0},\xi)+i\int_{t_{0}}^{t}e^{is|\xi|^{\alpha}}B(|u|^{2},\bar u)^\wedge(s,\xi)ds\\
&\quad +i\int_{t_{0}}^{t} e^{is|\xi|^{\alpha}}B(u,|u|^{2})^\wedge(s,\xi)ds+\int_{t_{0}}^{t}e^{is|\xi|^{\alpha}}(u\bar{u})_{HH+HL}^\wedge(s,\xi)ds.
}
Therefore we get
\EQN{
(\partial_{t}w+iD^\alpha w)^\wedge(t,\xi)&=e^{-it|\xi|^{\alpha}}\partial_{t}(e^{it|\xi|^{\alpha}}\hat{w})(t,\xi)\\
&=iB(|u|^{2}, \bar u)^\wedge(t,\xi)+ iB(u,|u|^{2})^\wedge(t,\xi)+(u\bar{u})_{HH+HL}^\wedge(t,\xi).
}
Hence,  we have
\EQN{
\CAS{
(\partial_{t}+iD^\alpha)w =(u \bar u)_{HH+HL}+iB(|u|^2, \bar u)+iB(u,|u|^2),\\
u =w-iB(u, \bar u).
}
}
Using the second equation as above, we obtain
\begin{align*}
(u \bar u)_{HH+HL}=(w\bar{w})_{HH+HL}+i[w \overline{B(u,\bar u)}]_{HH+HL}-i[B(u,\bar u)\bar u]_{HH+HL}.
\end{align*}
This finishes the proof of \eqref{e2.5}. 
\end{proof}

\begin{remark}
Note that in system \eqref{e2.5}, in the first equation $u$ only appears in the cubic terms.
This enables us to have more flexibility to use different resolution spaces for $w$ and $u$.  We can plug-in the second equation again into the first equation, then we can make $u$ appear in higher order nonlinearity.  However, this is not necessary for our purposes.
\end{remark}

Inspired by \cite{W}, we define some function spaces. {For $\lambda>0$, we define} the space $\PW$ with the norm
\EQ{\label{F}
\|f\|_{\PW}:=\sup_{k\in \mathbb{Z}}2^{2k_{+}}\big(2^{-\lambda k}\|\psi_{k}(\xi)\hat{f}\|_{2}+2^{(1-\lambda)k}\|\psi_{k}(\xi)\nabla_{\xi}\hat{f}\|_{2}
+2^{(2-\lambda)k}\|\psi_{k}(\xi)\nabla_{\xi}^{2}\hat{f}\|_{2}\big)
}
where $k_+:=\max(k,0)$, and the space $\PU$ with the norm
\EQ{\label{G}
\|g\|_{\PU}:=\sup_{k\in \mathbb{Z}}2^{2k_{+}}\big(2^{-\lambda k}&\|\psi_{k}(\xi)\hat{g}\|_{2}+2^{(1-\lambda)k}\|\psi_{k}(\xi)\nabla_{\xi}\hat{g}\|_{2}\big).
}
Denote $C(\R;\Hs)^2:=C(\R;\Hs)\times C(\R;\Hs)$.  Now,  we present our main results. 
\begin{theorem}\label{thm:main}
Suppose that $\alpha\in (1,2)$ and $\lambda\in(\frac{\alpha-1}{2},\frac{1}{2})$. Assume that the initial data satisfies the following assumption
$$
\|u_{0}\|_{\Hs}+\|u_{0}\|_{\PW}\leq \varepsilon_{0},
$$
{where $\varepsilon_{0}>0$ is a sufficiently small constant that depends only on $\alpha,\lambda$.}
Then there exists a unique global solution $(w,u)$ to \eqref{e2.5} in $C(\R;\Hs)^2$ such that
\EQ{\label{e2.10}
\sup_{t>0}\left(\|w(t)\|_{\Hs}+\|e^{itD^{\alpha}}w(t)\|_{\PW}+(1+t)^{1+\delta}\norm{w(t)}_{L^\infty}\right)&\\
+\sup_{t>0}\left(\|u(t)\|_{\Hs}+\|e^{itD^{\alpha}}u(t)\|_{\PU}+(1+t)^{1+\delta}\norm{u(t)}_{L^\infty}\right)&\leq C_{\alpha,\lambda} \, \varepsilon_{0},
}
where $0<\delta<\min\{\frac{\lambda+\frac{3}{2}}{\alpha},\frac{3}{2}\}-1$. As a consequence, the solution possesses the scattering property. In particular, $u$ is a unique global solution to \eqref{e1.1} in $C(\R;\Hs)$ and scatters.
\end{theorem}

\begin{remark}
We can show $w_0=w(0)\in H^2\cap F$. See Lemma \ref{lem4.7a}. However, we couldn't close the arguments by choosing $G=F$. See Remark \ref{rem4.8}. The use of a weaker norm for $u$ seems necessary. 
\end{remark}

\begin{remark}\label{re2.4}
If $\lambda\neq \frac{3}{2}\alpha-\frac{3}{2}$, we can take $\delta=\min\{\frac{\lambda+\frac{3}{2}}{\alpha},\frac{3}{2}\}-1$ in Theorem \ref{thm:main}.
\end{remark}

\begin{remark}\label{re2.5}
If $\alpha=2$, we require $\lambda>1/2$ to get decay rate $(1+t)^{-1-}$ but we require $\lambda\leq 1/2$ to balance the $high\times high \to 0$ frequency interaction. This is exactly where
the argument for Theorem \ref{thm:main} breaks down. 
\end{remark}

\smallskip

\section{Linear and bilinear estimates}\label{sec2}
\setcounter{equation}{0}

For the sake of notational simplicity,  for $f\in \Sw'$, we denote $f_{k}=P_kf$ and $f_{\leq k}=P_{\leq k}f$. In this section, we collect some linear and bilinear estimates. 

\begin{lemma}\label{lem3.1}
Let $\alpha>1$. For all $t>0$, $k\in \mathbb{Z}$, we have
\begin{align*}
\|e^{-itD^{\alpha}}f\|_{2}&=\|f\|_{2},\\
\|e^{-itD^{\alpha}}f_{k}\|_{\infty}&\les t^{-\frac{3}{2}}2^{3(1-\frac{\alpha}{2})k}\|f\|_{1}.
\end{align*}
In particular, by interpolation we have
\begin{align*}
\|e^{-itD^{\alpha}}f_{k}\|_{4}\les t^{-\frac{3}{4}}2^{\frac{3}{2}(1-\frac{\alpha}{2})k}\|f\|_{\frac{4}{3}}.
\end{align*}
\end{lemma}

\begin{proof}
The first estimate follows from Plancherel's theorem.
The second estimate follows from \cite{GN}.
\end{proof}

\begin{lemma}\label{lemint}
We have
\begin{align*}
\norm{f}_{\frac{4}{3}}\les \|f\|_{2}^{\frac{1}{4}}\cdot \big\||x|f\big\|_{2}^{\frac{3}{4}}  \quad {\rm and} \quad \norm{f}_{1}\les\|f\|_{2}^{\frac{1}{4}}\cdot \big\||x|^{2}f\big\|_{2}^{\frac{3}{4}}.
\end{align*}
\end{lemma}
\begin{proof}
For the first one, we have
\EQN{
\norm{f}_{\frac{4}{3}}\leq& \norm{1_{|x|\leq R}f}_{\frac{4}{3}}+\norm{1_{|x|\geq R}f}_{\frac{4}{3}}\\
\les&  \norm{1_{|x|\leq R}}_{4}\norm{f}_2+\norm{|x|^{-1} 1_{|x|\geq R}}_{4}\norm{|x|f}_{2}\\
\les&\  R^{\frac{3}{4}}\norm{f}_2+R^{-\frac{1}{4}}\norm{|x|f}_{2}.
}
Optimising in $R$, we complete the proof. The proof of the second inequality follows a similar approach.
\end{proof}

\begin{lemma}\label{lem3.2}
{Let $\alpha>1$, $\lambda>0$ and $\PW$ be given by \eqref{F}.} For all $t>0$ and $k\in \mathbb{Z}$, we have
\EQN{
\|e^{-itD^{\alpha}}f_{k}\|_{\infty}\les& \min\{2^{(\lambda+\frac{3}{2})k-2k_{+}},t^{-\frac{3}{2}}2^{(\lambda-\frac{3}{2}\alpha+\frac{3}{2})k-2k_{+}}\}\|f\|_{\PW},\\
\|e^{-itD^{\alpha}}f_{k}\|_{4}\les& \min\{2^{(\lambda+\frac{3}{4})k-2k_{+}},t^{-\frac{3}{4}}2^{(\lambda-\frac{3}{4}\alpha+\frac{3}{4})k-2k_{+}}\}\|f\|_{\PW},\\
\|e^{-itD^{\alpha}}(\nabla\hat{f}_{k})^\vee\|_{4}\les& \min\{2^{(\lambda-\frac{1}{4})k-2k_{+}},t^{-\frac{3}{4}}2^{(\lambda-\frac{3}{4}\alpha-\frac{1}{4})k-2k_{+}}\}\|f\|_{\PW},\\
\|e^{-itD^{\alpha}}f_{k}\|_{6}\les& \min\{2^{(\lambda+1)k-2k_{+}},t^{-1}2^{(\lambda-\alpha+1)k-2k_{+}}\}\|f\|_{\PW}.
}
\end{lemma}

\begin{proof}
By Lemmas \ref{lem3.1} and \ref{lemint}, we have
\begin{align*}
\|e^{-itD^{\alpha}}f_{k}\|_{\infty}&\les t^{-\frac{3}{2}}2^{3(1-\frac{\alpha}{2})k}\|f_{k}\|_{1}\les t^{-\frac{3}{2}}2^{3(1-\frac{\alpha}{2})k}\|f_{k}\|_{2}^{\frac{1}{4}}\big\||x|^{2}f_{k}\big\|_{2}^{\frac{3}{4}}\\
&\les t^{-\frac{3}{2}}2^{(\lambda-\frac{3}{2}\alpha+\frac{3}{2})k-2k_{+}}\|f\|_{\PW}.
\end{align*}
We apply  Bernstein's inequality to obtain
\begin{align*}
\|e^{-itD^{\alpha}}f_{k}\|_{\infty}&\les 2^{\frac{3}{2}k}\|f_{k}\|_{2}\les2^{(\lambda+\frac{3}{2})k-2k_{+}}\|f\|_{\PW}.
\end{align*}
Using interpolation, we  then get 
\EQN{
\|e^{-itD^{\alpha}}f_{k}\|_{4}\leq\|f_{k}\|_{2}^{\frac{1}{2}}\|e^{-itD^{\alpha}}f_{k}\|_{\infty}^{\frac{1}{2}}
\les \min\{2^{(\lambda+\frac{3}{4})k-2k_{+}},t^{-\frac{3}{4}}2^{(\lambda-\frac{3}{4}\alpha+\frac{3}{4})k-2k_{+}}\}\|f\|_{\PW}.
}
Similarly,  we may use Lemma \ref{lem3.1} to obtain
\EQN{
\|e^{-itD^{\alpha}}(\nabla\hat{f}_{k})^\vee\|_{4}&\les t^{-\frac{3}{4}}2^{\frac{3}{2}(1-\frac{\alpha}{2})k}\big\||x|f_{k}\big\|_{\frac{4}{3}}
\les t^{-\frac{3}{4}}2^{\frac{3}{2}(1-\frac{\alpha}{2})k}\big\||x|f_{k}\big\|_{2}^{\frac{1}{4}}\big\||x|^{2}f_{k}\big\|_{2}^{\frac{3}{4}}\\
&\les t^{-\frac{3}{4}}2^{(\lambda-\frac{3}{4}\alpha-\frac{1}{4})k-2k_{+}}\|f\|_{\PW}.
}
By Bernstein's inequality, we have
\begin{align*}
\|e^{-itD^{\alpha}}(\nabla\hat{f}_{k})^\vee\|_{4}\les 2^{\frac{3}{4}k}\|\nabla\hat{f}_{k}\|_{2}\les2^{(\lambda-\frac{1}{4})k-2k_{+}}\|f\|_{\PW}.
\end{align*}
By interpolation, it yields
\begin{align*}
\|e^{-itD^{\alpha}}f_{k}\|_{6} \leq& \|f_{k}\|_{2}^{\frac{1}{3}}\|e^{-it|D|^{\alpha}}f_{k}\|_{\infty}^{\frac{2}{3}}\les \min\{2^{(\lambda+1)k-2k_{+}},t^{-1}2^{(\lambda-\alpha+1)k-2k_{+}}\}\|f\|_{\PW}.
\end{align*}
\end{proof}

Let $m\in L^\infty(\R^3\times \R^3)$.  For $j,k\in \Z$,  define
\EQN{
T_{j,k}(f,g)(x):=\int_{\R^3}\int_{\R^3}e^{ix\xi}\psi_{k}(\xi)m(\xi,\eta)\psi_{j}(\eta)\hat{f}(\xi-\eta)\hat{g}(\eta)d\xi d\eta.
}

\begin{lemma}\label{lem3.3}
{Let $m\in C^{\infty}(\R^3\times \R^3)$. For all $\beta_{1},\beta_{2}\in \mathbb{Z}^{3}_{+}$ with $|\beta_{1}|+|\beta_{2}|\leq 30$, suppose that there exists a constant $C_{\beta_{1},\beta_{2}}$ satisfying
\begin{align*}
\sup_{\xi,\eta\in \R^{3}}|\xi|^{|\beta_{1}|}|\eta|^{|\beta_{2}|}|\partial_{\xi}^{\beta_{1}}\partial_{\eta}^{\beta_{2}}m(\xi,\eta)|\leq C_{\beta_{1},\beta_{2}}<\infty.
\end{align*}
Then for all $1\leq p_{1},p_{2},p\leq \infty$,} $f\in L^{p_{1}}(\R^{3})$ and $g\in L^{p_{2}}(\R^{3})$, the following bilinear estimate holds,
\begin{align*}
\sup_{j,k\in \mathbb{Z}}\|T_{j,k}(f,g)\|_{p}\les\|f\|_{p_{1}}\|g\|_{p_{2}},
\end{align*}
where $\frac{1}{p}=\frac{1}{p_{1}}+\frac{1}{p_{2}}$. 
\end{lemma}

\begin{proof}
Choose a smooth function $\tilde{\psi}\in C^{\infty}_{0}(\R^{3}\setminus\{0\})$, which equals to 1 on $\supp\,\psi$.   By the Fourier series expansion, we have
\begin{align*}
\psi_{k}(\xi)m(\xi,\eta)\psi_{j}(\eta)=\sum_{\gamma\in \mathbb{Z}^{3}_{+}}m_{k,\gamma}(\xi)\tilde{\psi}_{j,\gamma}(\eta),
\end{align*}
where $m_{k,\gamma}(\xi):=\psi_{k}(\xi)\int_{\R^{3}} e^{i\gamma\zeta}m(\xi,2^{j}\zeta)\psi(\zeta)d\zeta$ and $\tilde{\psi}_{j,\gamma}(\eta):=e^{i\gamma 2^{-j}\eta}\tilde{\psi}_{j}(\eta)$. By integration by parts, for all $|\nu|\leq10$, we have
\begin{align*}
|\partial_{\xi}^{\nu}m_{k,\gamma}(\xi)|\leq C_{\nu}(1+|\gamma|)^{-20}2^{-|\nu|k}.
\end{align*}
Define
\begin{align*}
T_{j,k}^{\gamma}(f,g)(x):=\int_{\R^3}\int_{\R^3}e^{ix\xi}m_{k,\gamma}(\xi)\tilde{\psi}_{j,\gamma}(\eta)\hat{f}(\xi-\eta)\hat{g}(\eta)d\xi d\eta=m_{k,\gamma}(D)(f\tilde{\psi}_{j,\gamma}(D)g)(x).
\end{align*}
For all $j,k\in \mathbb{Z}$, we have
\begin{align*}
\|T_{j,k}(f,g)\|_{p}\leq& \sum_{\gamma\in \mathbb{Z}^{3}_{+}}\|T_{j,k}^{\gamma}(f,g)\|_{p}
=\sum_{\gamma\in \mathbb{Z}^{3}_{+}}\|m_{k,\gamma}(D)(f\tilde{\psi}_{j,\gamma}(D)g)\|_{p}\\
\les& \sum_{\gamma\in \mathbb{Z}^{3}_{+}}(1+|\gamma|)^{-20}\|f\tilde{\psi}_{j,\gamma}(D)g\|_{p}
\les \sum_{\gamma\in \mathbb{Z}^{3}_{+}}(1+|\gamma|)^{-20}\|f\|_{p_{1}}\|\tilde{\psi}_{j,\gamma}(D)g\|_{p_{2}}\\
\les& \sum_{\gamma\in \mathbb{Z}^{3}_{+}}(1+|\gamma|)^{-10}\|f\|_{p_{1}}\|g\|_{p_{2}}\les \|f\|_{p_{1}}\|g\|_{p_{2}}.
\end{align*}
\end{proof}

\begin{lemma}\label{lem3.4}
Let $m$ be as in Lemma \ref{lem3.3}. For $f\in L^{2}(\R^{3})$ and $g\in L^{2}(\R^{3})$, the following bilinear estimate holds,
\begin{align*}
\Big\|\mathcal{F}^{-1}\big(\psi_{k}(\cdot)\int_{\R^3} m(\cdot,\eta)\hat{f}_{k_{1}}(\cdot-\eta)\hat{g}_{k_{2}}(\eta)d\eta\big)\Big\|_{2}\les\min\{2^{\frac{3}{2}k},2^{\frac{3}{2}k_{1}},2^{\frac{3}{2}k_{2}}\}\|f_{k_{1}}\|_{2}\|g_{k_{2}}\|_{2}.
\end{align*}
\end{lemma}

\begin{proof}
It is a consequence of Lemma \ref{lem3.3} and Bernstein's inequality.
\end{proof}

Let $m\in L^\infty(\R^3\times \R^3\times\R^{3})$. For $j,k,l\in\Z$, define
\EQN{
T_{j,k,l}(f,g,h)(x):=\int_{\R^3}\int_{\R^3}\int_{\R^3} e^{ix\xi}\psi_{k}(\xi)m(\xi,\eta,\zeta)\psi_{j}(\eta)\psi_{l}(\zeta)\hat{f}(\xi-\eta-\zeta)\hat{g}(\eta)\hat{h}(\zeta)d\xi d\eta d\zeta.
}

Again,  by the Fourier series expansion as in Lemma \ref{lem3.3}, we have
\begin{lemma}\label{lem3.5}
{Let $m\in C^{\infty}(\R^3\times \R^3\times \R^3)$. For all $\beta_{1},\beta_{2},\beta_{3}\in \mathbb{Z}^{3}_{+}$ with $|\beta_{1}|+|\beta_{2}|+|\beta_{3}|\leq 60$, suppose that there exists a constant $C_{\beta_{1},\beta_{2},\beta_{3}}$ satisfying
\begin{align*}
\sup_{\xi,\eta,\zeta\in\R^{3}}|\xi|^{|\beta_{1}|}|\eta|^{|\beta_{2}|}|\zeta|^{|\beta_{3}|}|\partial_{\xi}^{\beta_{1}}\partial_{\eta}^{\beta_{2}}\partial_{\zeta}^{\beta_{3}}m(\xi,\eta,\zeta)|\leq C_{\beta_{1},\beta_{2},\beta_{3}}<\infty.
\end{align*}
Then for all $1\leq p_{1},p_{2},p_{3},p\leq \infty$,} $f\in L^{p_{1}}(\R^{3})$ , $g\in L^{p_{2}}(\R^{3})$ and $h\in L^{p_{3}}(\R^{3})$, the following trilinear estimate holds,
\begin{align*}
\sup_{j,k,l\in \mathbb{Z}}\|T_{j,k,l}(f,g,h)\|_{p}\les\|f\|_{p_{1}}\|g\|_{p_{2}}\|h\|_{p_{3}},
\end{align*}
where $\frac{1}{p}=\frac{1}{p_{1}}+\frac{1}{p_{2}}+\frac{1}{p_{3}}$. 
\end{lemma}

{
\begin{definition}\label{def3.7}
Let $N\geq1$, $\tau>0$ and $\Omega\subseteq (\R^{3})^{N}$ be an open set. We say a  function $a\in \tau S^{0}(\Omega)$ if $a\in C^{\infty}(\Omega)$ and for all $\beta\in (\mathbb{Z}^{3}_{+})^{N}$ , there exists a constant $C_{\beta}$ such that
\begin{align*}
\sup_{v\in \Omega}\bigg|\big(\prod_{i=1}^{N}|v_{i}|^{|\beta_{i}|}\partial_{v_{i}}^{\beta_{i}}\big)a(v)\bigg|\leq C_{\beta}\tau,
\end{align*}
where $v:=(v_{1},...,v_{N})$,  $\beta:=(\beta_{1},...,\beta_{N})$ and $v_{i}\in\R^{3}$, $\beta_{i}\in\Z^{3}_{+}$ for all $1\leq i\leq N$.
\end{definition}}

Recall that $\phi$ is given by \eqref{e2.3}. We define
$$
\Omega_{k,k_{1},k_{2}}:=\{(\xi,\eta)\in\R^{6}:  |\xi|\in (2^{k-1},2^{k+1}), |\xi-\eta|\in (2^{k_{1}-1},2^{k_{1}+1}), |\eta|\in (2^{k_{2}-1},2^{k_{2}+1})\}.
$$
\begin{lemma}\label{lem3.8}
Assume $\alpha\in(1,2)$. Let $|k_{1}-k_{2}|< 10$ and $k_{1}\geq k-12$. Then we have
$$
\partial_{\xi_{i}}\phi\in 2^{(\alpha-1)k_{1}}S^{0}(\Omega_{k,k_{1},k_{2}}),\, \partial_{\eta_{i}}\phi\in 2^{k}2^{(\alpha-2)k_{1}}S^{0}(\Omega_{k,k_{1},k_{2}})
$$
{for all $1\leq i\leq3$. Moreover, there exist constants $c_{1},c_{2}>0$ such that
\EQN{
\inf_{\Omega_{k,k_{1},k_{2}}}|\nabla_{\xi}\phi|\geq c_{1}2^{(\alpha-1)k_{1}},\, \inf_{\Omega_{k,k_{1},k_{2}}}|\nabla_{\eta}\phi|\geq c_{2}2^{k}2^{(\alpha-2)k_{1}}.
}
}
\end{lemma}

\begin{proof}
Suppose $(\xi,\eta)\in \Omega_{k,k_{1},k_{2}}$. By our assumption, we have $|\xi|\les|\xi-\eta|\sim|\eta|$. Let $\varepsilon>0$ be  a sufficiently small constant,  which will be determined later.

{\bf Case 1:}  $|\xi|\geq \varepsilon|\eta|$, which means $|\xi-\eta|\sim|\eta|\sim|\xi|$. {Let $\angle(\xi,\xi-\eta)$ denote the angle between the vectors $\xi$ and $\xi-\eta$. If $\angle(\xi,\xi-\eta)\leq \varepsilon'$ and $\varepsilon'=\varepsilon'(\varepsilon)$ is small enough, we have $\big||\xi|^{\alpha-1}-|\xi-\eta|^{\alpha-1}\big|\geq C_{\varepsilon}|\eta|^{\alpha-1}$ because $|\xi-\eta|\sim|\eta|$.} Then by the triangle inequality, we have
\EQN{
|\nabla_{\xi}\phi(\xi,\eta)|=\bigg|\alpha|\xi|^{\alpha-1}\frac{\xi}{|\xi|}-\alpha|\xi-\eta|^{\alpha-1}\frac{\xi-\eta}{|\xi-\eta|}\bigg|\geq C_{\varepsilon}2^{(\alpha-1)k_{1}}.
}
If $\angle(\xi,\xi-\eta)\geq \varepsilon'$,  then we have
\EQN{
|\nabla_{\xi}\phi(\xi,\eta)|
=\bigg|\alpha|\xi|^{\alpha-1}\frac{\xi}{|\xi|}-\alpha|\xi-\eta|^{\alpha-1}\frac{\xi-\eta}{|\xi-\eta|}\bigg|
\geq C_{\varepsilon}|\xi|^{\alpha-1}\geq C_{\varepsilon}2^{(\alpha-1)k_{1}}.
}

Similarly, we have
\EQN{
|\nabla_{\eta}\phi(\xi,\eta)|=\bigg|\alpha|\eta|^{\alpha-1}\frac{\eta}{|\eta|}-\alpha|\xi-\eta|^{\alpha-1}\frac{\eta-\xi}{|\eta-\xi|}\bigg|\geq C_{\varepsilon}2^{(\alpha-1)k_{1}}.
}
The proofs of the other estimates are standard, so  we omit them.

{\bf Case 2:}  $|\xi|\leq \varepsilon|\eta|$.
By the triangle inequality, we have
\begin{align*}
|\nabla_{\xi}\phi(\xi,\eta)|\sim|\xi-\eta|^{\alpha-1}\sim2^{(\alpha-1)k_{1}}.
\end{align*}
Define $\omega(\eta):=|\eta|^{\alpha}$. Then
\begin{align*}
\partial_{\eta_{i},\eta_{j}}^{2}\omega(\eta)&=\partial_{\eta_{j}}(\alpha|\eta|^{\alpha-2}\eta_{i})=\alpha|\eta|^{\alpha-2}\big((\alpha-2)|\eta|^{-2}\eta_{i}\eta_{j}+\delta_{ij}\big),
\end{align*}
which implies that
\begin{align*}
\partial_{\eta}^{2}\omega(\eta)\xi=\alpha|\eta|^{\alpha-2}\big((\alpha-2)
\lb \xi,\eta/|\eta|\rb\eta/|\eta|+\xi\big).
\end{align*}
Let $\xi=\lb \xi,\eta/|\eta|\rb\eta/|\eta|+\xi'$ be the orthogonal decomposition. Then we have
\begin{align}\label{e3.2}
|(\alpha-2)
\lb \xi,\eta/|\eta|\rb\eta/|\eta|+\xi|=|(\alpha-1)
\lb \xi,\eta/|\eta|\rb\eta/|\eta|+\xi'|\sim|\xi|,
\end{align}
where we use $\alpha>1$.
By the mean value theorem, we have
\begin{align*}
\nabla_{\eta}\phi(\xi,\eta)&=\nabla_{\eta}\omega(\eta)-\nabla_{\eta}\omega(\eta-\xi)
=\int_{0}^{1}\partial_{\eta}^{2}\omega(\eta-\theta\xi)\xi d\theta.
\end{align*}
Because $|\xi|\leq \varepsilon|\eta|$, we have $|\eta-\theta\xi|\sim|\eta|$ if $\varepsilon$ is small enough and
$$
\frac{\eta-\theta\xi}{|\eta-\theta\xi|}=\frac{\eta}{|\eta|}+O(|\xi|/|\eta|)=\frac{\eta}{|\eta|}+O(\varepsilon),
$$
which tells us that
\begin{align*}
\partial_{\eta}^{2}\omega(\eta-\theta\xi)\xi&=\alpha|\eta-\theta\xi|^{\alpha-2}\big((\alpha-2)
\lb \xi,\eta-\theta\xi/|\eta-\theta\xi|\rb(\eta-\theta\xi)/|\eta-\theta\xi|+\xi\big)\\
&=\alpha|\eta-\theta\xi|^{\alpha-2}\big((\alpha-2)
\lb \xi,\eta/|\eta|\rb\eta/|\eta|+\xi\big)+O(\varepsilon|\eta|^{\alpha-2}|\xi|).
\end{align*}
Then  we have
\begin{align*}
\nabla_{\eta}\phi(\xi,\eta)=\alpha\int_{0}^{1}|\eta-\theta\xi|^{\alpha-2}d\theta\big((\alpha-1)
\lb \xi,\eta/|\eta|\rb\eta/|\eta|+\xi'\big)+O(\varepsilon|\eta|^{\alpha-2}|\xi|),
\end{align*}
which, together with \eqref{e3.2}, yields
\begin{align*}
|\nabla_{\eta}\phi(\xi,\eta)|
\sim & |\eta|^{\alpha-2}|\xi|+O(\varepsilon|\eta|^{\alpha-2}|\xi|)\sim|\eta|^{\alpha-2}|\xi|
\end{align*}
if $\varepsilon>0$ is small enough.
The other estimates can be proven by standard calculus.
\end{proof}

\begin{lemma}\label{lem3.9}
Assume $\alpha\in(1,2)$.
Let $k_{1}\geq k_{2}+10$ and $|k_{1}-k|\leq2$.
Then we have
$$
\partial_{\xi_{i}}\phi\in2^{(\alpha-2)k}2^{k_{2}}S^{0}(\Omega_{k,k_{1},k_{2}}),\,\partial_{\eta_{i}}\phi\in 2^{(\alpha-1)k}S^{0}(\Omega_{k,k_{1},k_{2}})
$$
{for all $1\leq i\leq  3$. Moreover, there exist constants $c_{1},c_{2}>0$ such that
\EQN{
\inf_{\Omega_{k,k_{1},k_{2}}}|\nabla_{\xi}\phi|\geq c_{1}2^{(\alpha-2)k}2^{k_{2}},\, \inf_{\Omega_{k,k_{1},k_{2}}}|\nabla_{\eta}\phi|\geq c_{2}2^{(\alpha-1)k}.
}}
\end{lemma}

\begin{proof}
Suppose $(\xi,\eta)\in \Omega_{k,k_{1},k_{2}}$. By the assumptions on $(k,k_{1},k_{2})$, we have $|\eta|\leq C|\xi|$.
Let $\varepsilon>0$ be small enough. If $|\eta|\leq \varepsilon|\xi|$, the above estimates can be proven in a similar manner to Case 2 of Lemma \ref{lem3.8}.
If $|\eta|\in [\varepsilon|\xi|,C|\xi|]$, then $|\xi|\sim|\eta|\sim|\xi-\eta|$ by the assumptions on $(k,k_{1},k_{2})$ and the above estimates can be proven in a similar way to Case 1 of Lemma \ref{lem3.8}.
\end{proof}

For $k\in\Z$, define
\EQ{\label{chi}
&\chi_{k}^{1}:=\{(k_{1},k_{2})\in \mathbb{Z}^{2}: |k_{1}-k_{2}|< 10,k\leq k_{1}+12\},\\
&\chi_{k}^{2}:=\{(k_{1},k_{2})\in \mathbb{Z}^{2}: k_{1}-k_{2}\geq10,|k_{1}-k|\leq2\},\\
&\chi_{k}^{3}:=\{(k_{1},k_{2})\in \mathbb{Z}^{2}: k_{2}-k_{1}\geq 10,|k_{2}-k|\leq2\},
}
and $\chi_{k}:=\cup_{\ell=1}^{3}\chi_{k}^{\ell}$.

\begin{lemma}\label{lem3.12aa}
Assume $\alpha\in(1,2)$.
Let $(k_{1},k_{2})\in \chi_{k}^{\ell}$ and $\ell\in\{1,2\}$. There exist $\{\rho_{j}^{(\ell)}\}_{j=1}^{3} \subseteq S^{0}(\Omega_{k,k_{1},k_{2}})$, depending on $k$ and $k_{1}$, that satisfy
$\sum_{j=1}^{3}\rho_{j}^{(\ell)}=1$ on $\Omega_{k,k_{1},k_{2}}$ and
$$
\inf_{j=1,2,3}\inf_{\supp \rho_{j}^{(\ell)}}\lambda^{-1}_{k,k_{1}}|\partial_{\eta_{j}} \phi|\geq c>0,
$$
where  $\lambda_{k,k_{1}}:=2^{k}2^{(\alpha-2)k_{1}}$ if $k_{1}-k>12$ and $\lambda_{k,k_{1}}:=2^{(\alpha-1)k}$ if $k_{1}-k\leq12$. Moreover, there exist $\{\rho_{j}^{(3)}\}_{j=1}^{3}\subseteq S^{0}(\R^{3}\setminus\{0\})$ satisfying $\sum_{j=1}^{3}\rho_{j}^{(3)}=1$ on $\R^{3}\setminus\{0\}$ and
$$
\inf_{j=1,2,3}\inf_{\supp \rho_{j}^{(3)}}\big|\eta_{j}/|\eta|\big|\geq c>0.
$$
Here $c$ is independent of $k,k_{1},k_{2}$ and the symbol class is defined in Definition \ref{def3.7}.
\end{lemma}

\begin{proof}
Let $\ell\in\{1,2\}$.
By Lemmas \ref{lem3.8}, \ref{lem3.9} and the triangle inequality, there exists a constant $C>0$ such that for all $(k_{1},k_{2})\in \chi_{k}^{\ell}$ and $(\xi,\eta)\in \Omega_{k,k_{1},k_{2}}$, we have
\EQN{
{\sqrt{3}C^{-1}\leq \lambda^{-1}_{k,k_{1}}|\nabla_{\eta} \phi(\xi,\eta)|\leq C.}
}
Choose a nonnegative function $\rho\in C^{\infty}_{c}(\R)$ such that $\rho(s)=1$ whenever $C^{-1}\leq |s|\leq C$ and $\rho(s)=0$ when $|s|>2C$ or $|s|<(2C)^{-1}$.
For $(k_{1},k_{2})\in \chi_{k}^{\ell}$ and  $j\in\{1,2,3\}$, define
\EQN{
\rho_{j}^{(\ell)}(\xi,\eta):=\rho(\lambda^{-1}_{k,k_{1}}\partial_{\eta_{j}} \phi(\xi,\eta))\bigg(\sum_{i=1}^{3}\rho(\lambda^{-1}_{k,k_{1}}\partial_{\eta_{i}} \phi(\xi,\eta))\bigg)^{-1}\in S^{0}(\Omega_{k,k_{1},k_{2}}),
}
where we use Lemmas \ref{lem3.8} and \ref{lem3.9} again.
Choose nonnegative $\tilde{\rho}\in C^{\infty}_{c}(\R)$ such that $\tilde{\rho}(s)=1$ if $(\sqrt{3})^{-1}\leq |s|\leq 1$ and $\tilde{\rho}(s)=0$ if $|s|>2$ or $|s|<(2\sqrt{3})^{-1}$. For  $j\in\{1,2,3\}$, define
\EQN{
\rho_{j}^{(3)}(\eta):=\tilde{\rho}(\eta_{j}/|\eta|)\bigg(\sum_{i=1}^{3}\tilde{\rho}(\eta_{i}/|\eta|)\bigg)^{-1}\in S^{0}(\R^{3}\setminus\{0\}).
}
The remaining results follow readily.
\end{proof}

{Although $\phi(\xi,\eta)^{-1}$, $\{(\partial_{\eta_{l}}\phi(\xi,\eta))^{-1}\}_{l=1}^{3}$ fail to satisfy the  conditions of Lemma \ref{lem3.3} when $|\xi-\eta|\cdot|\eta|^{-1}$ is small, the desired bilinear (or trilinear) estimates can still be established if we remove the factor $|\xi-\eta|$ by employing the following key lemma.}

{
\begin{lemma}\label{lem5.1}
Let $1\leq l\leq 3$. For all $N\geq1$, one has
$$
\phi(\xi,\eta)^{-1}=\phi(\xi,\eta)^{-1}\rho(\xi,\eta)^{-N}|\xi-\eta|^{N\alpha}+\sum_{j=1}^{N}\rho(\xi,\eta)^{-j}|\xi-\eta|^{(j-1)\alpha},
$$
and 
\begin{align*}
(\partial_{\eta_{l}}\phi(\xi,\eta))^{-1}=\frac{(\partial_{\eta_{l}}\omega(\eta-\xi))^{N}}{\partial_{\eta_{l}}\phi(\xi,\eta)(\partial_{\eta_{l}}\omega(\eta))^{N}}
+\sum_{j=1}^{N}(\partial_{\eta_{l}}\omega(\eta))^{-j}(\partial_{\eta_{l}}\omega(\eta-\xi))^{j-1},
\end{align*}
for all $(\xi,\eta)\in\R^{6}$ such  that $\eta,\phi(\xi,\eta),\partial_{\eta_{l}}\phi(\xi,\eta)\neq0$.
Here $\rho(\xi,\eta):=|\xi|^{\alpha}+|\eta|^{\alpha}$ and $\omega(\eta)=|\eta|^{\alpha}$.
\end{lemma}
}
\begin{proof}
We have
\begin{align*}
\phi(\xi,\eta)^{-1}=\phi(\xi,\eta)^{-1}-\rho(\xi,\eta)^{-1}+\rho(\xi,\eta)^{-1}=\phi(\xi,\eta)^{-1}\rho(\xi,\eta)^{-1}|\xi-\eta|^{\alpha}+\rho(\xi,\eta)^{-1}.
\end{align*}
By this relation again, we have
\begin{align*}
\phi(\xi,\eta)^{-1}=\phi(\xi,\eta)^{-1}\rho(\xi,\eta)^{-2}|\xi-\eta|^{2\alpha}+\rho(\xi,\eta)^{-2}|\xi-\eta|^{\alpha}+\rho(\xi,\eta)^{-1}.
\end{align*}
The other cases can be proven by induction.  This completes the proof of the first equality, and the second one can be derived in a similar manner.
\end{proof}

\section{Estimate for $u$}

From here we always assume $\alpha\in(1,2)$ and $\lambda\in(\frac{\alpha-1}{2},\frac{1}{2})$, which imply
\EQ{\label{e4.1s}
0<\lambda+\frac{3}{2}-\alpha<1.
}
Let $\PW,\PU$ be defined as in \eqref{F} and \eqref{G}. Suppose that $\delta$ is given by Theorem \ref{thm:main}. The norms of $w$ and $u$ are each composed of three components. For simplicity, we define the norms as follows:

For $w$:
\EQ{\label{e4.2s}
\norm{w}_{\Wone}:=& \sup_{t\geq0}\norm{w(t)}_{\Hs},\\
\norm{w}_{\Wtwo}:=& \sup_{t\geq0}\, (1+t)^{1+\delta}\norm{w(t)}_{L^\infty},\\
\norm{w}_{\Wthr}:=& \sup_{t\geq0}\norm{e^{itD^\alpha}w(t)}_{\PW}.
}

For $u$:
\EQ{\label{e4.3s}
\norm{u}_{\Uone}:=& \sup_{t\geq0}\norm{u(t)}_{\Hs},\\
\norm{u}_{\Utwo}:=& \sup_{t\geq0}\, (1+t)^{1+\delta}\norm{u(t)}_{L^\infty},\\
\norm{u}_{\Uthr}:=& \sup_{t\geq0}\norm{e^{itD^\alpha}u(t)}_{\PU}.
}
Let 
\EQ{\label{e4.4A}
W:=W_1\cap W_2\cap W_3 \quad and \quad U:=U_1\cap U_2\cap U_3.
}
From here we denote $f(t):=e^{itD^{\alpha}}w(t)$ and $g(t):=e^{itD^{\alpha}}u(t)$, which are the {\it profiles} of $w$ and $u$. 

In this section, we prove the following proposition, using the second equation in \eqref{e2.5}.

\begin{proposition}\label{prop5.1}
Suppose that $\eqref{e2.4}$ holds on $[0,\infty)$. Let $(W,U)$ be defined as in \eqref{e4.4A}. Then we have
\begin{align*}
\|u\|_{U}\les \|w\|_{W}+\|u\|_{U}^{2}.
\end{align*}
\end{proposition}

\begin{lemma}\label{lem4.2}
Let $t\geq0$ and $\{U_{i}\}_{i=2}^{3}$ be defined as in  \eqref{e4.3s}. Suppose that $g$ is the profile of $u$. For all $k\in\Z$, we have
\EQ{\label{e4.4s}
\|g_{k}(t)\|_{2}+2^{k}\|\nabla\hat{g}_{k}(t)\|_{2}\les2^{\lambda k-2k_{+}}\|u\|_{U_{3}},
}
\EQ{\label{e4.5s}
\|u_{k}(t)\|_{\infty}\les \min\{2^{\frac{3}{2} k-2k_{+}}\|u\|_{U_{1}},(1+t)^{-1-\delta}\|u\|_{U_{2}},2^{(\lambda+\frac{3}{2}) k-2k_{+}}\|u\|_{U_{3}}\},
}
\EQ{\label{e4.6s}
2^{-\alpha k}\|u_{\leq k}(t)\|_{\infty}\les\min\{2^{(\lambda-\alpha+\frac{3}{2})k},2^{-\alpha k}(1+t)^{-1-\delta},1\}\|u\|_{U_2\cap U_3},
}
\EQ{\label{e4.7a}
\sum_{\ell\in \{0,1\}}2^{-\alpha k}\sum_{k_{1}\leq k}2^{-(1-\ell) k_{1}}\|e^{-itD^{\alpha}}(\nabla^{\ell}\hat{g}_{k_{1}})^\vee(t)\|_{\infty}\les2^{-k}\min
\{2^{(\lambda-\alpha+\frac{3}{2})k},1\}\|u\|_{U_3}.
}
\end{lemma}

\begin{proof}
Inequalities \eqref{e4.4s} and \eqref{e4.5s} are obtained by applying \eqref{G}, \eqref{e4.3s} together with Bernstein's inequality.
By \eqref{e4.5s}, we have
\begin{align*}
\|u_{\leq k}(t)\|_{\infty}\leq \sum_{k_{1}\leq k}\|u_{k_{1}}(t)\|_{\infty}
\les \sum_{k_{1}\leq k}2^{(\lambda+\frac{3}{2})k_{1}-2k_{1,+}}\|u\|_{\Uthr}
\les 2^{(\lambda+\frac{3}{2})k}\|u\|_{\Uthr}.
\end{align*}
By \eqref{e4.3s}, we have
\begin{align*}
\|u_{\leq k}(t)\|_{\infty}\les \|u(t)\|_{\infty}\leq (1+t)^{-1-\delta}\|u\|_{\Utwo}.
\end{align*}
These, together with \eqref{e4.1s}, finish the proof of \eqref{e4.6s}.
By Bernstein's inequality, \eqref{e4.4s} and \eqref{e4.1s}, we have
\begin{align*}
\sum_{\ell\in \{0,1\}}2^{-\alpha k}\sum_{k_{1}\leq k}2^{-(1-\ell) k_{1}}\|e^{-itD^{\alpha}}(\nabla^{\ell}\hat{g}_{k_{1}})^\vee(t)\|_{\infty}&\les 2^{-\alpha k}\sum_{k_{1}\leq k}2^{(\lambda+\frac{1}{2})k_{1}-2k_{1,+}}\|u\|_{\Uthr}\\
&\les \min
\{2^{(\lambda-\alpha+\frac{1}{2})k},2^{-k}\}\|u\|_{\Uthr}.
\end{align*}
This finishes the proof of \eqref{e4.7a}.
\end{proof}

Proposition \ref{prop5.1} follows from the next Lemma \ref{lem4.3} and the second equation in \eqref{e2.5}.

\begin{lemma}\label{lem4.3}
Let $\{U_{i}\}_{i=1}^{3}$ and $U$ be defined as in \eqref{e4.3s} and \eqref{e4.4A}, respectively, and $B(u,\bar u)$ be given by \eqref{e2.2}. Then
\EQ{\label{e4.8a}
\norm{B(u,\bar u)}_{U}
\les\norm{u}_{U}^{2}.
}
Moreover, for all $t\geq0$ and $k\in\Z$, we have
\begin{align}\label{e4.8s}
\|B(u,\bar u)_{k}(t)\|_{2}\les 2^{\lambda k-2k_{+}}
\min\{1,2^{-\alpha k}(1+t)^{-1-\delta}\}\norm{u}_{U_{2}\cap U_{3}}^{2},
\end{align}
\begin{align}\label{e4.9s}
\|B(u,\bar u)_{k}(t)\|_{\infty}\les 2^{-\alpha k}(1+t)^{-2-2\delta}\norm{u}_{U_{2}\cap U_{3}}^{2},
\end{align}
\begin{align}\label{e4.10s}
\|e^{-itD^{\alpha}}(x e^{itD^{\alpha}}B(u,\bar u)_{k})(t)\|_{\infty}\les 2^{-k}(1+t)^{-1-\delta}\norm{u}_{U_{2}\cap U_{3}}^{2},
\end{align}
\begin{align}\label{e4.11s}
\|\nabla^{2}\big(e^{itD^{\alpha}}B(u,\bar{u})_{k}\big)^\wedge(t)\|_{2}\les 2^{(\lambda-2)k-2k_{+}}\max\{t2^{\alpha k},1\}\norm{u}_{U_{2}\cap U_{3}}^{2}.
\end{align}
In particular, for all $v\in \PU$, we have
\EQ{\label{e4.12s}
\|B(v,\bar{v})\|_{\PW}\les \|v\|_{\PU}^{2},
}
where $\PW$ and $\PU$ are defined as in \eqref{F} and \eqref{G}.
\end{lemma}

\begin{proof}
Suppose that $g$ is the profile of $u$.
By Lemma \ref{lem5.1}, we have 
\begin{align*}
&e^{it|\xi|^{\alpha}}B(u,\bar u)^\wedge(t,\xi)\\
=&\sum_{j=1}^{N+1}\sum_{\substack{k_{2}-k_{1}\geq 10\\
k_{1},k_{2}\in \Z}}2^{\frac{(j-1)\alpha}{2}(k_{1}-k_{2})}\int_{\R^3}e^{it\phi(\xi,\eta)}2^{-\alpha k_{2}}a_{j}(\xi,\eta) \big(2^{-k_{1}}|\xi-\eta|\big)^{(j-1)\alpha}\\
&\qquad\cdot\hat{g}_{k_{1}}(t,\xi-\eta)\hat{\bar{g}}_{k_{2}}(t,\eta)d\eta,
\end{align*}
where $a_{j}(\xi,\eta):=2^{\frac{(j-1)\alpha}{2}(k_{1}-k_{2})}2^{j\alpha k_{2}}\rho(\xi,\eta)^{-j}$ for $1\leq j\leq N$ and
$$
a_{N+1}(\xi,\eta):=2^{\frac{N\alpha}{2}(k_{1}-k_{2})}2^{(N+1)\alpha k_{2}}\phi(\xi,\eta)^{-1}\rho(\xi,\eta)^{-N}.
$$
By standard calculus, $\{a_{j}\}_{j=1}^{N+1}$ satisfy the condition in Lemma \ref{lem3.3} if $N$ is large enough and the integrands are not zero.  It suffices to show that $T(u,\bar u)$ satisfies the desired estimates, where 
\begin{align*}
e^{it|\xi|^{\alpha}}T(u,\bar u)^\wedge(t,\xi):=\sum_{k_{2}\in\Z}\int_{\R^3} e^{it\phi(\xi,\eta)}2^{-\alpha k_{2}}a(\xi,\eta)\hat{g}_{\leq k_{2}-10}(t,\xi-\eta)\hat{\bar{g}}_{k_{2}}(t,\eta)d\eta
\end{align*}
and $a$ satisfies the condition in Lemma \ref{lem3.3}. 
For $k\geq0$, we apply  Lemma  \ref{lem3.3} to obtain
\begin{align*}
\|T(u,\bar u)_{k}(t)\|_{\Hs}\les 2^{-\alpha k}\|u(t)\|_{\infty}\|u(t)\|_{\Hs}\leq 2^{-\alpha k}(1+t)^{-1-\delta}\|u\|_{\Uone}\|u\|_{\Utwo}.
\end{align*}
For $k\leq0$, by Lemma  \ref{lem3.3},  \eqref{e4.6s} and \eqref{e4.3s}, we obtain
\begin{align*}
\|T(u,\bar u)_{k}(t)\|_{\Hs}&\les
2^{-\alpha k}\|u(t)\|_{\Hs}\sup_{|k_{2}-k|\leq 2}\|u_{\leq k_{2}-10}(t)\|_{\infty}\\
&\les 2^{-\alpha k}\min\{(1+t)^{-1-\delta},2^{(\lambda+\frac{3}{2})k}\}\|u\|_{U_{1}}\|u\|_{U_{2}\cap U_{3}}.
\end{align*}
Then by the triangle inequality and \eqref{e4.1s}, we have
\begin{align*}
\|T(u,\bar u)\|_{\Uone}\les\sum_{k\in\Z}\min\{2^{(\lambda-\alpha+\frac{3}{2})k},2^{-\alpha k}\}\|u\|_{U}^{2}\les \|u\|_{U}^{2}.
\end{align*}
By Lemmas \ref{lem3.3} and \ref{lem4.2}, we have
\begin{align}\label{e5.3}
\|T(u,\bar u)_{k}(t)\|_{2}&\les \sup_{|k_{2}-k|\leq 2}2^{-\alpha k}\|u_{\leq k_{2}-10}(t)\|_{\infty}\|g_{k_{2}}(t)\|_{2}\nonumber\\
&\les \min\{1,2^{-\alpha k}(1+t)^{-1-\delta}\}2^{\lambda k-2k_{+}}\|u\|_{U_{2}\cap U_{3}}^{2},
\end{align}
\begin{align}\label{e4.15a}
\|T(u,\bar u)_{k}(t)\|_{\infty}&\les \sup_{|k_{2}-k|\leq 2}2^{-\alpha k}\|u_{\leq k_{2}-10}(t)\|_{\infty}\|u_{k_{2}}(t)\|_{\infty}\nonumber\\
&\les\min\{2^{(\lambda-\alpha+\frac{3}{2})k},2^{-\alpha k}\}(1+t)^{-1-\delta}\|u\|_{U_{2}\cap U_{3}}^{2}.
\end{align}
Then by \eqref{e4.1s} again, we have
\begin{align*}
\|T(u,\bar u)\|_{\Utwo}\les\sum_{k\in\Z}\min\{2^{(\lambda-\alpha+\frac{3}{2})k},2^{-\alpha k}\}\|u\|_{U_{2}\cap U_{3}}^{2}\les \|u\|_{U_{2}\cap U_{3}}^{2}.
\end{align*}

By integration by parts, we have
\begin{align*}
&\psi_{k}(\xi)\nabla_{\xi}\big(e^{it|\xi|^{\alpha}}T(u,\bar{u})^\wedge(t,\xi)\big)\\
=&\sum_{|k_{2}-k|\leq 2}i\psi_{k}(\xi)\int_{\R^3} e^{it\phi}2^{-\alpha k_{2}}a(\xi,\eta)t(\nabla_{\xi}\phi+\nabla_{\eta}\phi)\hat{g}_{\leq k_{2}-10}(t,\xi-\eta)\hat{\bar{g}}_{k_{2}}(t,\eta)d\eta\\
&+\psi_{k}(\xi)\int_{\R^3} e^{it\phi}2^{-\alpha k_{2}}(\nabla_{\xi}+\nabla_{\eta})a(\xi,\eta)\hat{g}_{\leq k_{2}-10}(t,\xi-\eta)\hat{\bar{g}}_{k_{2}}(t,\eta)d\eta\\
&+\psi_{k}(\xi)\int_{\R^3} e^{it\phi}2^{-\alpha k_{2}}a(\xi,\eta)\hat{g}_{\leq k_{2}-10}(t,\xi-\eta)\nabla_{\eta}\hat{\bar{g}}_{k_{2}}(t,\eta)d\eta,
\end{align*}
which, together with Lemmas \ref{lem3.3} and  \ref{lem4.2}, implies  that 
\begin{align}\label{e5.5a}
&\|\psi_{k}\nabla \big(e^{itD^{\alpha}}T(u,\bar{u})\big)^\wedge(t)\|_{2}\nonumber\\
&\les \sup_{|k_{2}-k|\leq 2}(t+2^{-\alpha k})\|u_{\leq k_{2}-10}(t)\|_{\infty}\big(2^{-k}\|g_{k_{2}}(t)\|_{2}+\|\nabla\hat{g}_{k_{2}}(t)\|_{2}\big)\nonumber\\
&\les 2^{(\lambda-1)k-2k_{+}}\|u\|_{U_{2}\cap U_{3}}^{2},
\end{align}
and
\begin{align}\label{e4.17a}
&\|e^{-itD^{\alpha}}(x e^{itD^{\alpha}}T(u,\bar u))_{k}(t)\|_{\infty}\nonumber\\
&\les \sup_{|k_{2}-k|\leq 2}\big[(t+2^{-\alpha k})\|u_{\leq k_{2}-10}(t)\|_{\infty}\cdot2^{-k}\|u_{k_{2}}(t)\|_{\infty}
\nonumber\\
&\quad +\|u_{\leq k_{2}-10}(t)\|_{\infty}\cdot2^{-\alpha k}\|e^{-itD^{\alpha}}(\nabla\hat{g}_{k_{2}})^\vee(t)\|_{\infty}\big]\nonumber\\
&\les 2^{-k}(1+t)^{-1-\delta}\|u\|_{U_{2}\cap U_{3}}^{2}.
\end{align}
By \eqref{e5.3} and \eqref{e5.5a}, we have
\begin{align*}
\|T(u,\bar u)\|_{\Uthr}=\sup_{t\geq0}\|e^{itD^{\alpha}}T(u,\bar u)(t)\|_{\PU}\les \|u\|_{U_{2}\cap U_{3}}^{2}.
\end{align*}
Based on the estimates for $\{\|T(u,\bar u)\|_{U_{i}}\}_{i=1}^{3}$, \eqref{e4.8a} follows. \eqref{e4.8s}, \eqref{e4.9s} and \eqref{e4.10s} are derived from \eqref{e5.3}, \eqref{e4.15a} and \eqref{e4.17a}, respectively.
Let $\{\rho_{l}^{(3)}\}_{l=1}^{3}$ be given by Lemma \ref{lem3.12aa}. {By integration by parts, we can write $\psi_{k}(\xi)\partial_{\xi_{i},\xi_{j}}^{2}\big(e^{it|\xi|^{\alpha}}T(u,\bar{u})^\wedge(t,\xi)\big)$ as the sum of following terms,
\EQN{
i\psi_{k}(\xi)\sum_{(k_{1},k_{2})\in \chi_{k}^{3}}2^{-\alpha k_{2}}\sum_{l=1}^{3}\int_{\R^3} e^{it\phi(\xi,\eta)}t\partial_{\eta_{l}}\bigg(\big(\frac{\partial_{\xi_{i}}\phi\partial_{\xi_{j}}\phi}{\partial_{\eta_{l}}\phi} a\big)(\xi,\eta)\hat{g}_{k_{1}}(t,\xi-\eta)\bigg)(\rho_{l}^{(3)}\hat{\bar{g}}_{k_{2}})(t,\eta)d\eta,
}
\EQN{
i\psi_{k}(\xi)\sum_{(k_{1},k_{2})\in \chi_{k}^{3}}2^{-\alpha k_{2}}\sum_{l=1}^{3}\int_{\R^3} e^{it\phi(\xi,\eta)}t\big(\rho_{l}^{(3)}\frac{\partial_{\xi_{i}}\phi\partial_{\xi_{j}}\phi}{\partial_{\eta_{l}}\phi} a\big)(\xi,\eta)\hat{g}_{ k_{1}}(t,\xi-\eta)\partial_{\eta_{l}}(\rho_{l}^{(3)}\hat{\bar{g}}_{k_{2}})(t,\eta)d\eta,
}
\EQN{
i\psi_{k}(\xi)\sum_{(k_{1},k_{2})\in \chi_{k}^{3}}2^{-\alpha k_{2}}\int_{\R^3} e^{it\phi(\xi,\eta)}t\big(\partial_{\xi_{i}}\phi(\xi,\eta)\partial_{\xi_{j}}+\partial_{\xi_{j}}\phi(\xi,\eta)\partial_{\xi_{i}}\big)a(\xi,\eta)\hat{g}_{ k_{1}}(t,\xi-\eta)\hat{\bar{g}}_{k_{2}}(t,\eta)d\eta,
}
\EQN{
i\psi_{k}(\xi)\sum_{(k_{1},k_{2})\in \chi_{k}^{3}}2^{-\alpha k_{2}}\int_{\R^3} e^{it\phi(\xi,\eta)}t\partial_{\xi_{i},\xi_{j}}^{2}\phi(\xi,\eta)a(\xi,\eta)\hat{g}_{ k_{1}}(t,\xi-\eta)\hat{\bar{g}}_{k_{2}}(t,\eta)d\eta,
}
\EQN{
i\psi_{k}(\xi)\sum_{(k_{1},k_{2})\in \chi_{k}^{3}}2^{-\alpha k_{2}}\int_{\R^3} e^{it\phi(\xi,\eta)}ta(\xi,\eta)(\partial_{\xi_{i}}\phi(\xi,\eta)\partial_{\xi_{j}}+\partial_{\xi_{j}}\phi(\xi,\eta)\partial_{\xi_{i}}
)\hat{g}_{ k_{1}}(t,\xi-\eta)\hat{\bar{g}}_{k_{2}}(t,\eta)d\eta,
}
\EQN{
i\psi_{k}(\xi)\sum_{(k_{1},k_{2})\in \chi_{k}^{3}}2^{-\alpha k_{2}}\int_{\R^3} e^{it\phi(\xi,\eta)}ta(\xi,\eta)\partial_{\eta_{j}}\phi(\xi,\eta)\partial_{\xi_{i}}\hat{g}_{ k_{1}}(t,\xi-\eta)\hat{\bar{g}}_{k_{2}}(t,\eta)d\eta,
}
\EQN{
\psi_{k}(\xi)\sum_{(k_{1},k_{2})\in \chi_{k}^{3}}2^{-\alpha k_{2}}\int_{\R^3} e^{it\phi(\xi,\eta)}\big(\partial_{\xi_{i}}a(\xi,\eta)\partial_{\xi_{j}}+\partial_{\xi_{j}}a(\xi,\eta)\partial_{\xi_{i}}\big)\hat{g}_{ k_{1}}(t,\xi-\eta)\hat{\bar{g}}_{k_{2}}(t,\eta)d\eta,
}
\EQN{
\psi_{k}(\xi)\sum_{(k_{1},k_{2})\in \chi_{k}^{3}}2^{-\alpha k_{2}}\int_{\R^3} e^{it\phi(\xi,\eta)}\partial_{\xi_{i},\xi_{j}}^{2}a(\xi,\eta)\hat{g}_{ k_{1}}(t,\xi-\eta)\hat{\bar{g}}_{k_{2}}(t,\eta)d\eta,
}
\EQN{
\psi_{k}(\xi)\sum_{(k_{1},k_{2})\in \chi_{k}^{3}}2^{-\alpha k_{2}}\int_{\R^3} e^{it\phi(\xi,\eta)}\partial_{\xi_{i}}\hat{g}_{ k_{1}}(t,\xi-\eta)\big(a(\xi,\eta)\partial_{\eta_{j}}+\partial_{\eta_{j}}a(\xi,\eta)\big)\hat{\bar{g}}_{k_{2}}(t,\eta)d\eta.
}}
We can use Lemma \ref{lem5.1} to treat $(\partial_{\eta_{l}}\phi)^{-1}$ and then apply Lemma \ref{lem4.2} to obtain
\begin{align}\label{e4.15s}
&\|\psi_{k}\nabla^{2} \big(e^{itD^{\alpha}}T(u,\bar{u})\big)^\wedge(t)\|_{2}\nonumber\\
&\les \sup_{|k_{2}-k|\leq 2}\bigg[\big(2^{-\alpha k}\|u_{\leq k_{2}-10}(t)\|_{\infty}+2^{(1-\alpha)k}\sum_{k_{1}\leq k_{2}}\|e^{-itD^{\alpha}}(\nabla\hat{g}_{k_{1}})^\vee(t)\|_{\infty}\big)\cdot2^{-2k}\|g_{k_{2}}(t)\|_{2}\nonumber\\
&\quad+t2^{k}\sum_{k_{1}\leq k_{2}}\big(2^{-k_{1}}\|u_{k_{1}}(t)\|_{\infty}+\|e^{-itD^{\alpha}}(\nabla\hat{g}_{k_{1}})^\vee(t)\|_{\infty}\big)
\big(2^{-2k}\|g_{k_{2}}(t)\|_{2}+2^{-k}\|\nabla\hat{g}_{k_{2}}(t)\|_{2}\big)\bigg]\nonumber\\
&\les \big(t2^{(\lambda+\alpha-2)k-2k_{+}}+2^{(\lambda-2)k-2k_{+}}\big)\|u\|_{U_{2}\cap U_{3}}^{2}.
\end{align}
Thus, we have established  \eqref{e4.11s}. 

By the same way as in the proof of \eqref{e5.3}, \eqref{e5.5a} and \eqref{e4.15s}, we have 
\EQN{
\|B(v,\bar{v})\|_{\PW}\les\sup_{|\beta|\leq2}\sup_{k\in\Z}2^{(|\beta|-\lambda)k+2k_{+}}\|\psi_{k}(\xi)\partial_{\xi}^{\beta} (B(v,\bar{v}))^\wedge(\xi)\|_{2}\les \|v\|_{\PU}^{2}.
}
This finishes the proof of  \eqref{e4.12s}.
\end{proof}

\begin{remark}\label{rem4.8}
The estimates \eqref{e4.8s}-\eqref{e4.11s} will be used in the next section. \eqref{e4.11s} is an estimate for the profile of $u$ in $F$-norm. {We couldn't obtain an estimate that is uniform in $t$}, even if using $W$-norm for u on the right-hand side. This is the main reason why we use $G$-norm for the profile of $u$.
\end{remark}

\begin{lemma}\label{lem4.7a}
Define $w_{0}:=u_{0}+iB(u_{0},\bar u_{0})$. Then we have
\EQN{
\|w_{0}\|_{\Hs}+\|w_{0}\|_{F}\les \|u_{0}\|_{\Hs}+\|u_{0}\|_{F}+\|u_{0}\|_{\Hs}^{2}+\|u_{0}\|_{G}^{2}.
}
\end{lemma}

\begin{proof}
By the way similar to the proof of \eqref{e4.8a}, we have 
\EQN{
\|w_{0}\|_{\Hs}\les \|u_{0}\|_{\Hs}+\|u_{0}\|_{\Hs}^{2}+\|u_{0}\|_{G}^{2}.
}
By \eqref{e4.12s}, we have
\EQN{
\|w_{0}\|_{\PW}\les \|u_{0}\|_{\PW}+\|u_{0}\|_{\PU}^{2}.
}
These finish the proof of the lemma.
\end{proof}

\section{Estimate for $w$}

In this section, we prove the following proposition, using the first equations in \eqref{e2.5}.

\begin{proposition}\label{prop4.1}
Suppose that $\eqref{e1.1}$ and $\eqref{e2.5}$ hold on $[0,\infty)$. Let $(W,U)$ be defined as in \eqref{e4.4A}. Then we have
\begin{align*}
\|w\|_{W}\les \|w_{0}\|_{\Hs}+\|w_{0}\|_{\PW}+\sum_{l=2}^{4}\|(w,u)\|_{W\times U}^{l}.
\end{align*}
\end{proposition}

First we prove some basic estimates. 

\begin{lemma}\label{lem4.4}
For all $t>0$ and $k\in \Z$, we have
\EQN{
\|u_{k}(t)\|_{6}\les t^{-1}2^{(\lambda-\alpha+1)k-2k_{+}}\|u\|_{\Uthr}.
}
\end{lemma}

\begin{proof}
Let $g$ be the profile of $u$. By Lemma \ref{lem3.1}, \cite[Theorem 1.4.19]{G14}, H\"older's inequality, and \eqref{e4.4s}, we have
\EQN{
\|e^{-itD^{\alpha}}g_{k}(t)\|_{L^{6,2}}&\les t^{-1}2^{(2-\alpha)k}\|g_{k}(t)\|_{L^{\frac{6}{5},2}}
\les t^{-1}2^{(2-\alpha)k}\|xg_{k}(t)\|_{L^{2}}\big\||x|^{-1}\big\|_{L^{3,\infty}}\\
&\les t^{-1}2^{(\lambda-\alpha+1)k-2k_{+}}\|u\|_{\Uthr}.
}
Since $L^{6,2}\subseteq L^{6}$, the proof is finished.
\end{proof}

\begin{lemma}\label{lem4.5}
Suppose that $u$ satisfies \eqref{e1.1} on $[0,\infty)$. Let $g$ be the profile of $u$ and $\{U_{i}\}_{i=1}^{3}$ be defined as in \eqref{e4.3s}. For all $t>0$ and $k\in \Z$, we have
\EQ{\label{e5.13A}
\|\psi_{k}\partial_{t}\hat{g}(t)\|_{2}=\|(|u|^{2})_{k}(t)\|_{2}\les 2^{-2k_{+}}\min\{t^{-1}2^{\lambda k}\norm{u}_{U_{2}\cap U_{3}}^{2},(1+t)^{-1-\delta}\norm{u}_{U_{1}\cap U_{2}}^{2}\},
}
\EQ{\label{e5.14A}
\|\psi_{k}\nabla\partial_{t}\hat{g}(t)\|_{2}\les t^{-1}2^{(\lambda-1)k-2k_{+}}
\max\{t2^{\alpha k},1\}\norm{u}_{U_{2}\cap U_{3}}^{2},
}
\EQ{\label{e5.15A}
\|e^{-itD^{\alpha}}\partial_{t}g(t)\|_{\infty}= \|u(t)\|_{\infty}^{2}\leq (1+t)^{-2-2\delta}\norm{u}_{U_{2}}^{2}.
}
\end{lemma}

\begin{proof}
By symmetry, we have
\begin{align*}
\|(|u|^{2})_{k}(t)\|_{2}\les \sum_{k_{1}\geq k-12}\sum_{|k_{2}-k_{1}|\leq10}\|(u_{k_{1}}\bar{u}_{k_{2}})_{k}(t)\|_{2}+
\sum_{|k_{1}-k|\leq2}\big\|\sum_{k_{2}\leq k_{1}-10}(u_{k_{1}}\bar{u}_{k_{2}})_{k}(t)\big\|_{2}.
\end{align*}
By Bernstein's inequality, H\"older's inequality, Lemmas \ref{lem4.4} and \eqref{e4.1s}, we have
\begin{align*}
\sum_{k_{1}\geq k-12}\sum_{|k_{2}-k_{1}|\leq10}\|(u_{k_{1}}\bar{u}_{k_{2}})_{k}(t)\|_{2}&\leq \sum_{k_{1}\geq k-12}\sum_{|k_{2}-k_{1}|\leq10}2^{\frac{1}{2}k}\|u_{k_{1}}(t)\|_{2}\|u_{k_{2}}(t)\|_{6}\\
&\les t^{-1}2^{\frac{1}{2}k-2k_{+}}\sum_{k_{1}\geq k-12}2^{(2\lambda-\alpha+1)k_{1}}2^{-2k_{1,+}}\norm{u}_{U_{3}}^{2}\\
&\les t^{-1}2^{\lambda k-2k_{+}}\sum_{k_{1}\geq k-12}2^{(\lambda-\alpha+\frac{3}{2})k_{1}}2^{-2k_{1,+}}
\norm{u}_{U_{3}}^{2}\\
&\les t^{-1}2^{\lambda k-2k_{+}}\norm{u}_{U_{3}}^{2}.
\end{align*}
By H\"older's inequality, we have
\begin{align*}
\sum_{|k_{1}-k|\leq2}\big\|\sum_{k_{2}\leq k_{1}-10}(u_{k_{1}}\bar{u}_{k_{2}})_{k}(t)\big\|_{2}&\les
\sup_{|k_{1}-k|\leq 2}\|u_{k_{1}}(t)\|_{2}\|u_{\leq k_{1}-10}(t)\|_{\infty}\\
&\les 2^{\lambda k-2k_{+}}\|g(t)\|_{\PU}\|u(t)\|_{\infty}\\
&\les (1+t)^{-1-\delta}2^{\lambda k-2k_{+}}\norm{u}_{U_{2}\cap U_{3}}^{2},
\end{align*}
which means
\begin{align}\label{e5.5}
\|(|u|^{2})_{k}(t)\|_{2}\les
t^{-1}2^{\lambda k-2k_{+}}\norm{u}_{U_{2}\cap U_{3}}^{2}.
\end{align}
If $k\leq0$, we have
\begin{align*}
\|(u\bar{u})_{k}(t)\|_{2}\les \|(u\bar{u})(t)\|_{2}\leq \|u(t)\|_{2}\|u(t)\|_{\infty}\les\|u(t)\|_{\Hs}\|u(t)\|_{\infty}\les(1+t)^{-1-\delta}\norm{u}_{U_{1}\cap U_{2}}^{2}.
\end{align*}
If $k\geq0$, we have
\begin{align*}
\|((u\bar{u})_{HH})_{k}(t)\|_{2}
\leq& \sum_{k_{1}\geq k-12}\sum_{|k_{2}-k_{1}|\leq10}\|u_{k_{1}}(t)\bar{u}_{k_{2}}(t)\|_{2}\les \sum_{k_{1}\geq k-12}\|u_{k_{1}}(t)\|_{2}\|u(t)\|_{\infty}\\
\leq& (1+t)^{-1-\delta}\sum_{k_{1}\geq k-12}2^{-2 k_{1,+}}\norm{u}_{U_{1}\cap U_{2}}^{2}\les (1+t)^{-1-\delta}2^{-2 k_{+}}\norm{u}_{U_{1}\cap U_{2}}^{2},
\end{align*}
and
\begin{align*}
\|((u\bar{u})_{HL})_{k}(t)\|_{2}
\leq& \sup_{|k_{1}-k|\leq 2}\|u_{k_{1}}(t)\bar{u}_{\leq k_{1}-10}(t)\|_{2}\\
\les& \sup_{|k_{1}-k|\leq 2}\|u_{k_{1}}(t)\|_{2}\|u(t)\|_{\infty}\les (1+t)^{-1-\delta}2^{-2 k_{+}}\norm{u}_{U_{1}\cap U_{2}}^{2}.
\end{align*}
By symmetry, there holds
\begin{align*}
\|(u\bar{u})_{k}(t)\|_{2}\les\|((u\bar{u})_{HH})_{k}(t)\|_{2}
+\|((u\bar{u})_{HL})_{k}(t)\|_{2}
\les (1+t)^{-1-\delta}2^{-2 k_{+}}\norm{u}_{U_{1}\cap U_{2}}^{2}.
\end{align*}
These, together with \eqref{e5.5} and \eqref{e1.1}, complete the proof of \eqref{e5.13A}.

By \eqref{e1.1}, we have
\begin{align*}
&\partial_{t}\hat{g}(t,\xi)=e^{it|\xi|^{\alpha}}(u\bar{u})^\wedge(t,\xi)
=\int_{\R^3}e^{it\phi(\xi,\eta)}\hat{g}(t,\xi-\eta)\hat{\bar{g}}(t,\eta)d\eta.
\end{align*}
We can write $\psi_{k}(\xi)\partial_{\xi_{j}}\partial_{t}\hat{g}(t,\xi)$ as
\begin{align*}
&
\psi_{k}(\xi)\int_{\R^3} e^{it\phi(\xi,\eta)}\partial_{\xi_{j}}\hat{g}(t,\xi-\eta)\hat{\bar{g}}(t,\eta)d\eta+
it\psi_{k}(\xi)\int_{\R^3} e^{it\phi(\xi,\eta)}\partial_{\xi_{j}}\phi(\xi,\eta)\hat{g}(t,\xi-\eta)\hat{\bar{g}}(t,\eta)d\eta\\
&=\psi_{k}(\xi)\bigg(\sum_{(k_{1},k_{2})\in \chi_{k}^{1}}+\sum_{(k_{1},k_{2})\in \chi_{k}^{2}}+\sum_{(k_{1},k_{2})\in \chi_{k}^{3}}\bigg)\int_{\R^3} e^{it\phi(\xi,\eta)}\partial_{\xi_{j}}\hat{g}_{k_{1}}(t,\xi-\eta)\hat{\bar{g}}_{k_{2}}(t,\eta)d\eta\\
&+it\psi_{k}(\xi)\bigg(\sum_{(k_{1},k_{2})\in \chi_{k}^{1}}+\sum_{(k_{1},k_{2})\in \chi_{k}^{2}}+\sum_{(k_{1},k_{2})\in \chi_{k}^{3}}\bigg)\int_{\R^3} e^{it\phi(\xi,\eta)}\partial_{\xi_{j}}\phi(\xi,\eta)\hat{g}_{k_{1}}(t,\xi-\eta)\hat{\bar{g}}_{k_{2}}(t,\eta)d\eta\\
&=:\sum_{\ell=1}^{6}I_{\ell}(t,\xi).
\end{align*}
By Lemma \ref{lem3.3}, we have {
\begin{align*}
\|I_{1}(t)\|_{2}&\les \sup_{|l|\leq 10}\sum_{k_{1}\geq k-12}\|\nabla\hat{g}_{k_{1}}(t)\|_{2}\|u_{k_{1}+l}(t)\|_{\infty}\les \|u(t)\|_{\infty}\sum_{k_{1}\geq k-12}2^{(\lambda-1)k_{1}-2k_{1,+}}\norm{u}_{U_{3}}\\
&\les (1+t)^{-1-\delta}2^{(\lambda-1)k-2k_{+}}\norm{u}_{U_{2}\cap U_{3}}^{2},
\end{align*}
and
\begin{align*}
\|I_{2}(t)\|_{2}\les \sup_{|k_{1}-k|\leq 2}\|\nabla\hat{g}_{k_{1}}(t)\|_{2}\|u_{\leq k_{1}-10}(t)\|_{\infty}
\les (1+t)^{-1-\delta}2^{(\lambda-1)k-2k_{+}}\norm{u}_{U_{2}\cap U_{3}}^{2}.
\end{align*}}
By integration by parts, we have
\begin{align*}
I_{3}(t,\xi)=&\psi_{k}(\xi)\sum_{|k_{2}-k|\leq 2}\bigg(it\int_{\R^3} e^{it\phi(\xi,\eta)}\partial_{\eta_{j}}\phi(\xi,\eta)\hat{g}_{\leq k_{2}-10}(t,\xi-\eta)\hat{\bar{g}}_{k_{2}}(t,\eta)d\eta\\
&\hskip 3.5cm+\int_{\R^3} e^{it\phi(\xi,\eta)}\hat{g}_{\leq k_{2}-10}(t,\xi-\eta)\partial_{\eta_{j}}\hat{\bar{g}}_{k_{2}}(t,\eta)d\eta\bigg),
\end{align*}
which, together with Lemma \ref{lem3.3}, implies 
\begin{align*}
\|I_{3}(t)\|_{2}&\les \sup_{|k_{2}-k|\leq 2}\|u_{\leq k_{2}-10}(t)\|_{\infty}
\big(t2^{(\alpha-1)k}\|g_{k_{2}}(t)\|_{2}+\|\nabla\hat{g}_{k_{2}}(t)\|_{2}\big)\\
&\les \big((1+t)^{-1-\delta}2^{(\lambda-1)k-2k_{+}}+2^{(\alpha-1)k}2^{\lambda k-2k_{+}}\big)\norm{u}_{U_{2}\cap U_{3}}^{2}.
\end{align*}
Let $\{\rho_{l}^{(1)}\}_{l=1}^{3}$ be given by Lemma \ref{lem3.12aa}. For $I_{4}(t)$, we  use integration by parts to get
\begin{align*}
I_{4}(t,\xi)=&\psi_{k}(\xi)\sum_{l=1}^{3}\sum_{(k_{1},k_{2})\in \chi_{k}^{1}}\bigg[-\int_{\R^3} e^{it\phi(\xi,\eta)}\hat{g}_{k_{1}}(t,\xi-\eta)\partial_{\eta_{l}}\bigg(\big(\rho_{l}^{(1)}\frac{\partial_{\xi_{j}}\phi}{\partial_{\eta_{l}}\phi}\big)(\xi,\eta)\hat{\bar{g}}_{k_{2}}(t,\eta)\bigg)d\eta\\
&\hskip 3cm-\int_{\R^3} e^{it\phi(\xi,\eta)}\partial_{\xi_{l}}\hat{g}_{k_{1}}(t,\xi-\eta)\big(\rho_{l}^{(1)}\frac{\partial_{\xi_{j}}\phi}{\partial_{\eta_{l}}\phi}\big)(\xi,\eta)\hat{\bar{g}}_{k_{2}}(t,\eta)d\eta\bigg].
\end{align*}
Then we follow the same way as in the proof of \eqref{e5.5} to obtain
\begin{align*}
\|I_{4}(t)\|_{2}\les t^{-1}2^{(\lambda-1)k-2k_{+}}\norm{u}_{U_{2}\cap U_{3}}^{2}.
\end{align*}
For $I_{5}(t)$ and $I_{6}(t)$, we apply Lemma \ref{lem3.3} to get 
\begin{align*}
\sum_{\ell=5}^{6}\|I_{\ell}(t)\|_{2}\les t2^{(\alpha-1)k}\sup_{|l|\leq 2}\|g_{k+l}(t)\|_{2}\|u_{\leq k+l-10}(t)\|_{\infty}
\les 2^{(\alpha-1)k}2^{\lambda k-2k_{+}}\norm{u}_{U_{2}\cap U_{3}}^{2}.
\end{align*}
These imply \eqref{e5.14A}. \eqref{e5.15A} can be proven by \eqref{e1.1} and the definition of $\Utwo$ in \eqref{e4.3s}.
\end{proof}

\begin{corollary}
For all $t>0$, we have 
\EQ{\label{e5.27A}
\sum_{j\in\Z}\big\|u_{j}(t)\big\|_{\infty}\les  (1+t)^{-1-\frac{\delta}{2}}\|u\|_{U_{2}\cap U_{3}}^{2},
}
\EQ{\label{e5.28A}
\sum_{j\in\Z}\big\|B(u,\bar{u})_{j}(t)\big\|_{\infty}\les (1+t)^{-1-\delta}\|u\|_{U_{2}\cap U_{3}}^{2}.
}
If $u$ satisfies \eqref{e1.1} on $[0,\infty)$, then for all $t>0$, we have
\EQ{\label{e5.27AA}
\sum_{j\in\Z}\big\|e^{-itD^{\alpha}}\partial_{t}g_{j}(t)\big\|_{\infty}\les t^{-1}(1+t)^{-1-\frac{\delta}{2}}\|u\|_{U_{2}\cap U_{3}}^{2},
}
where $g$ is the profile of $u$ and $\delta$ is given by Theorem \ref{thm:main}.
\end{corollary}

\begin{proof}
Without loss of generality, assume that $\|u\|_{U_{2}\cap U_{3}}=1$.
By Bernstein's inequality, \eqref{e4.5s}, 
\eqref{e5.13A},
\eqref{e5.15A},
\eqref{e4.8s} and
\eqref{e4.1s}, for all $t>0$,
we have
\EQN{
\sum_{j\in\Z}\big\|u_{j}(t)\big\|_{\infty}\les \sum_{j\in\Z}\min\{2^{(\lambda+\frac{3}{2})j-2j_{+}},(1+t)^{-1-\delta}\}\les (1+t)^{-1-\frac{\delta}{2}},
}
\EQN{
\sum_{j\in\Z}\big\|B(u,\bar{u})_{j}(t)\big\|_{\infty}\les (1+t)^{-1-\delta}\sum_{j\in\Z}2^{(\lambda+\frac{3}{2}-\alpha)j-2j_{+}}\les (1+t)^{-1-\delta},
}
\EQN{
\sum_{j\in\Z}\big\|(|u|^{2})_{j}(t)\big\|_{\infty}\les \sum_{j\in\Z}\min\{t^{-1}2^{(\lambda+\frac{3}{2})j-2j_{+}},(1+t)^{-2-2\delta}\}\les t^{-1}(1+t)^{-1-\frac{\delta}{2}}.
}    
We remark that if \eqref{e1.1} holds on $[0,\infty)$, then $e^{-itD^{\alpha}}\partial_{t}g(t)=|u(t)|^{2}$ for all $t\geq0$. These complete the proof.
\end{proof}

\begin{lemma}\label{lem4.6}
For all $t>0$ and $k\in\Z$, we have 
\begin{align}\label{e4.5}
C_{k,1}(t):=2^{-\lambda k}\sum_{ k_{1}\geq k-12}2^{-2k_{1,+}}\min\{2^{\frac{3}{2}k+2\lambda k_{1}},t^{-\frac{3}{2}}2^{(2\lambda-\frac{3}{2}\alpha+\frac{3}{2})k_{1}}\}\les (1+t)^{-1-\varepsilon},
\end{align}
\begin{align}\label{e4.6}
C_{k,2}(t):=\sum_{k_{2}\in\Z}2^{-2k_{2,+}}\min\{2^{(\lambda+\frac{3}{2}) k_{2}},t^{-\frac{3}{2}}2^{(\lambda-\frac{3}{2}\alpha+\frac{3}{2})k_{2}}\}\les (1+t)^{-1-\delta},
\end{align}
where $\varepsilon:=\min\{\frac{2\lambda}{\alpha-1},\frac{3}{2}-\lambda\}-1>0$ and $\delta$ is given by Theorem \ref{thm:main}.
\end{lemma}

\begin{proof}
If $\frac{\alpha-1}{2}<\lambda\leq \frac{3}{2}\frac{\alpha-1}{\alpha+1}$, we have
\begin{align*}
&C_{k,1}(t)\les 2^{-2k_{+}}2^{(\frac{3}{2}-\frac{\alpha+1}{\alpha-1}\lambda)k}t^{-\frac{2\lambda}{\alpha-1}}\les t^{-\frac{2\lambda}{\alpha-1}}.
\end{align*}
If $\frac{3}{2}\frac{\alpha-1}{\alpha+1}<\lambda< \frac{1}{2}$, we can choose $\tau_{1}\in(0,2\lambda)$ satisfying $\frac{3}{2}-\lambda=\frac{2\lambda-\tau_{1}}{\alpha-1}$ and $2\lambda-\tau_{1}<\frac{3(\alpha-1)}{2}$ such that
\begin{align*}
C_{k,1}(t)&\les 2^{-\lambda k}\sum_{ k_{1}\geq k-12}\min\{2^{\frac{3}{2}k+(2\lambda-\tau_{1}) k_{1}},t^{-\frac{3}{2}}2^{(2\lambda-\frac{3}{2}\alpha+\frac{3}{2}-\tau_{1})k_{1}}\}\\
&\les 2^{(\frac{3}{2}-\lambda-\frac{2\lambda-\tau_{1}}{\alpha-1})k}t^{-\frac{2\lambda-\tau_{1}}{\alpha-1}}= t^{\lambda-\frac{3}{2}}.
\end{align*}
These, together with the estimate
\begin{align*}
C_{k,1}(t)\leq\sum_{ k_{1}\geq k-12}2^{-2k_{1,+}}2^{(\frac{3}{2}-\lambda)k+2\lambda k_{1}}
\leq\sum_{ k_{1}\geq k-12}2^{-2k_{1,+}}2^{(\lambda+\frac{3}{2})k_{1}}\les 1,
\end{align*}
imply that \eqref{e4.5} holds {for all $t>0$ and $k\in\Z$.} 

Next we show \eqref{e4.6}. If $\frac{\alpha-1}{2}<\lambda\leq\frac{3(\alpha-1)}{2}$, we can choose a small $\tau_{2}\in[0,2)$ such that $\delta=\min\{\frac{\lambda+\frac{3}{2}-\tau_{2}}{\alpha},\frac{3}{2}\}-1>0$ and
\begin{align*}
C_{k,2}(t)\leq\sum_{ k_{2}\in\Z}\min\{2^{(\lambda+\frac{3}{2}-\tau_{2}) k_{2}},t^{-\frac{3}{2}}2^{(\lambda-\frac{3}{2}\alpha+\frac{3}{2}-\tau_{2})k_{2}}\}\les t^{-\frac{\lambda+\frac{3}{2}-\tau_{2}}{\alpha}}.
\end{align*}
If $\frac{3(\alpha-1)}{2}<\lambda<\frac{3\alpha+1}{2}$, we have
\begin{align*}
C_{k,2}(t)\les\sum_{ k_{2}\in\Z}2^{-2k_{2,+}}t^{-\frac{3}{2}}2^{(\lambda-\frac{3}{2}\alpha+\frac{3}{2})k_{2}}\les t^{-\frac{3}{2}}.
\end{align*}
{Then for all $\frac{\alpha-1}{2}<\lambda<\frac{1}{2}$ and $t\geq1$, we have
\begin{align*}
C_{k,2}(t)\les t^{-1-\delta},
\end{align*}
which, together with the estimate
\begin{align*}
\sup_{t>0}C_{k,2}(t)\leq \sum_{k_{2}\in\Z}2^{-2k_{2,+}}2^{(\lambda+\frac{3}{2})k_{2}}\les 1
\end{align*}
implies that for all $t>0$ and $k\in\Z$, \eqref{e4.6} holds.
}
\end{proof}

\subsection{$W_1$-estimate}

\begin{proposition}\label{prop4.7}
Suppose that $\eqref{e2.5}$ holds on $[0,\infty)$. Let $W_{1}$ be defined as in \eqref{e4.2s}. Then we have
\begin{align*}
\|w\|_{W_{1}}\les \|w_{0}\|_{\Hs}+\sum_{l=2}^{3}\|(w,u)\|_{W\times U}^{l}.
\end{align*}
\end{proposition}

\begin{lemma}\label{lem4.8}
Let $\{(W_{i},U_{i})\}_{i=1}^{3}$ be defined as in \eqref{e4.2s} and \eqref{e4.3s}. For all $w$ and $v$, we have
\EQN{
\|(w\bar{v})_{HH+HL}(t)\|_{\Hs}\les (1+t)^{-1-\delta}\left(\|w\|_{W_{1}}\|v\|_{U_{2}}+\|w\|_{W_{2}}\|v\|_{U_{1}}\right)
}
holds for arbitrary $t\geq0$.
\end{lemma}

\begin{proof}
By Lemma \ref{lem3.3}, we have 
\EQN{
\|(w\bar{v})_{LH}(t)\|_{\Hs}&\les \big(\sum_{k\in\Z}\|((w\bar{v})_{LH})_{k}(t)\|_{\Hs}^{2}\big)^{1/2}\les \big(\sum_{k\in\Z}\sum_{|k_{1}-k|\leq2}\|w_{\leq k_{1}-10}(t)\bar{v}_{k_{1}}(t)\|_{\Hs}^{2}\big)^{1/2}\\
&\les\|w(t)\|_{\infty}\big(\sum_{k\in\Z}\sum_{|k_{1}-k|\leq2}\|v_{k_{1}}(t)\|_{\Hs}^{2}\big)^{1/2}\leq \|w(t)\|_{\infty}\|v(t)\|_{\Hs},
}
which, together with \eqref{e2.1} and the estimate
\EQN{
\|(w\bar{v})(t)\|_{\Hs}\les \|w(t)\|_{\Hs}\|v(t)\|_{\infty}+\|v(t)\|_{\Hs}\|w(t)\|_{\infty},
}
finishes the proof of the lemma.
\end{proof}

\begin{proof}[Proof of Proposition \ref{prop4.7}]
Let $t>0$ and $s\in[0,t]$. By Lemma \ref{lem4.8} and \eqref{e4.8a}, we have
\EQN{
&\|(w\bar{w})_{HH+HL}(s)\|_{\Hs}+\|[w \overline{B(u,\bar u)}]_{HH+HL}(s)\|_{\Hs}+\|[B(u,\bar u)\bar u]_{HH+HL}(s)\|_{\Hs}\\
&\les(1+s)^{-1-\delta}\sum_{l=2}^{3}\|(w,u)\|_{W\times U}^{l}.
}
Following an approach analogous to that used in the proof of \eqref{e4.8s}, we have 
\EQN{
\|B(|u|^2,\bar u)_{k}(s)\|_{2}\les 2^{-\alpha k}\|u(s)\|_{\infty}^{2}\sup_{|k_{1}-k|\leq 2}\|u_{k_{1}}(s)\|_{2},
}
\EQN{
\|B(|u|^2,\bar u)_{k}(s)\|_{2}&\les 2^{(\frac{3}{2}-\alpha) k}\|(u\bar{u})(s)\|_{2}\sup_{|k_{1}-k|\leq 2}\|u_{k_{1}}(s)\|_{2}\\
&\les 2^{(\frac{3}{2}-\alpha) k}\|u(s)\|_{2}\|u(s)\|_{\infty}\sup_{|k_{1}-k|\leq 2}\|u_{k_{1}}(s)\|_{2},
}
which, mean
\EQN{
\|B(|u|^2,\bar u)_{k}(s)\|_{2}\les \min\{2^{-\alpha k-2k_{+}}
,2^{(\lambda+\frac{3}{2}-\alpha) k-2k_{+}}\}(1+s)^{-1-\delta}\|u\|_{U}^{3}.
}
By \eqref{e4.6s} and \eqref{e5.13A}, we have
\EQN{
\|B(u,|u|^2)_{k}(s)\|_{2}&\les 2^{-\alpha k}\sup_{|k_{1}-k|\leq 2}\|u_{\leq k_{1}-10}(s)\|_{\infty}\|(|u|^{2})_{k_{1}}(s)\|_{2}\\
&\les \min\{2^{(\lambda-\alpha+\frac{3}{2})k},2^{-\alpha k}\}(1+s)^{-1-\delta}2^{-2k_{+}}\|u\|_{U}^{3}.
}
Then by \eqref{e4.1s}, we have
\EQN{
\|B(|u|^2,\bar u)(s)\|_{\Hs}\les \sum_{k\in\Z}2^{2k_{+}}\|B(|u|^2,\bar u)_{k}(s)\|_{2}\les(1+s)^{-1-\delta}\|u\|_{U}^{3},
}
\EQN{
\|B(u,|u|^2)(s)\|_{\Hs}\les \sum_{k\in\Z}2^{2k_{+}}\|B(u,|u|^2)_{k}(s)\|_{2}\les(1+s)^{-1-\delta}\|u\|_{U}^{3}.
}
With the above estimates at hand, we may invoke  \eqref{e2.5} combined with Duhamel's formula to derive
\EQN{
\|w(t)\|_{\Hs}&\leq \|w_{0}\|_{\Hs}+ \int_{0}^{t}\bigg(\|(w\bar{w})_{HH+HL}(s)\|_{\Hs}+\|[w \overline{B(u,\bar u)}]_{HH+HL}(s)\|_{\Hs}\\
&+\|[B(u,\bar u)\bar u]_{HH+HL}(s)\|_{\Hs}+\|B(|u|^2,\bar u)(s)\|_{\Hs}+\|B(u,|u|^2)(s)\|_{\Hs}\bigg)ds\\
&\les \|w_{0}\|_{\Hs}+\sum_{l=2}^{3}\|(w,u)\|_{W\times U}^{l}.
}
\end{proof}

\subsection{$W_2$-estimate}

\begin{proposition}\label{prop4.9}
Let $\{W_{i}\}_{i=2}^{3}$ be defined as in \eqref{e4.2s}. Then for all $w\in \Wthr$, we have
\begin{align*}
\|w\|_{\Wtwo}\les
\|w\|_{\Wthr}.
\end{align*}
\end{proposition}

\begin{proof}
Let $f$ be the profile of $w$. By Lemmas \ref{lem3.2} and \ref{lem4.6}, for all $t>0$, we have
\begin{align*}
\|w(t)\|_{\infty} \leq &\sum_{k\in \Z}\|e^{-itD^{\alpha}}f_{k}(t)\|_{\infty}\\
\les&\sum_{k\in \Z}\min\{2^{(\lambda
+\frac{3}{2})k-2k_{+}},t^{-\frac{3}{2}}2^{(\lambda-\frac{3}{2}\alpha+\frac{3}{2})k-2k_{+}}\}\|f(t)\|_{\PW}\\
\les& (1+t)^{-1-\delta}\|f(t)\|_{\PW},
\end{align*}
and for $t=0$, we have
\begin{align*}
\|w_{0}\|_{\infty} 
\les\sum_{k\in \Z}2^{(\lambda
+\frac{3}{2})k-2k_{+}}\|f(0)\|_{\PW}
\les \|f(0)\|_{\PW}.
\end{align*}
These imply
\begin{align*}
\|w\|_{\Wtwo}
\les\sup_{t\geq0}\|f(t)\|_{\PW}=\|w\|_{\Wthr}.
\end{align*}
\end{proof}

\subsection{$W_3$-estimate}

\begin{proposition}\label{prop4.10s}
Under the assumption of Proposition \ref{prop4.1}, we have 
\begin{align*}
\|w\|_{W_{3}}\les \|w_{0}\|_{\PW}+\sum_{l=2}^{4}\|(w,u)\|_{W\times U}^{l},
\end{align*}
where  $W_{3}$ is defined as in \eqref{e4.2s}.
\end{proposition}

Recall that $f(t)=e^{itD^{\alpha}}w(t)$ and $\{\chi_{k}^{\ell}\}_{\ell=1}^{2}$ is defined by \eqref{chi}.
Then we have
\EQ{\label{e4.7}
e^{it|\xi|^{\alpha}}(w\bar{w})_{HH+HL}^\wedge(\xi)=&\sum_{\substack{k_{2}-k_{1}<10\\
k_{1},k_{2}\in \Z}}\int_{\R^3} e^{it\phi(\xi,\eta)}\hat{f}_{k_{1}}(t,\xi-\eta)\hat{\bar{f}}_{k_{2}}(t,\eta)d\eta,
}
where $\phi(\xi,\eta)=|\xi|^{\alpha}-|\xi-\eta|^{\alpha}+|\eta|^{\alpha}$.
Consequently,
\EQ{\label{e4.23}
\psi_{k}(\xi)e^{it|\xi|^{\alpha}}(w\bar{w})_{HH+HL}^\wedge(\xi)=
\sum_{(k_{1},k_{2})\in \chi_{k}^{1}\cup\chi_{k}^{2}}H_{k,k_{1},k_{2}}(t,\xi),
}
where
\begin{align*}
H_{k,k_{1},k_{2}}(t,\xi):=\psi_{k}(\xi)\int_{\R^3} e^{it\phi(\xi,\eta)}\hat{f}_{k_{1}}(t,\xi-\eta)\hat{\bar{f}}_{k_{2}}(t,\eta)d\eta.
\end{align*}

Taking the partial derivative $\partial_{\xi_{i}}$ of both sides of  \eqref{e4.7}, we have
\EQ{\label{e4.9}
\psi_{k}(\xi)\partial_{\xi_{i}}\big(e^{it|\xi|^{\alpha}}(w\bar{w})_{HH+HL}^\wedge(\xi)\big)=\sum_{\ell=1}^{2}\sum_{(k_{1},k_{2})\in \chi_{k}^{1}\cup\chi_{k}^{2}} I_{k,k_{1},k_{2}}^{(\ell)}(t,\xi),
}
where 
\begin{align*}
I_{k,k_{1},k_{2}}^{(1)}(t,\xi):=&\psi_{k}(\xi)\int_{\R^3}e^{it\phi(\xi,\eta)}\partial_{\xi_{i}}\hat{f}_{k_{1}}(t,\xi-\eta)\hat{\bar{f}}_{k_{2}}(t,\eta)d\eta,\\
I_{k,k_{1},k_{2}}^{(2)}(t,\xi):=&i\psi_{k}(\xi)\int_{\R^3} e^{it\phi(\xi,\eta)}t\partial_{\xi_{i}}\phi(\xi,\eta)\hat{f}_{k_{1}}(t,\xi-\eta)\hat{\bar{f}}_{k_{2}}(t,\eta)d\eta.
\end{align*}

Taking   $\partial_{\xi_{i},\xi_{j}}^{2}$ of both sides of \eqref{e4.7}, we have
\EQ{\label{e4.25}
\psi_{k}(\xi)\partial_{\xi_{i},\xi_{j}}^{2}\big(e^{it|\xi|^{\alpha}}(w\bar{w})_{HH+HL}^\wedge(\xi)\big)=\sum_{\ell=1}^{4}\sum_{(k_{1},k_{2})\in \chi_{k}^{1}\cup\chi_{k}^{2}}J_{k,k_{1},k_{2}}^{(\ell)}(t,\xi),
}
where {
\EQN{
J_{k,k_{1},k_{2}}^{(1)}(t,\xi):=&\psi_{k}(\xi)\int_{\R^3} e^{it\phi(\xi,\eta)}\partial_{\xi_{i},\xi_{j}}^{2}\hat{f}_{k_{1}}(t,\xi-\eta)\hat{\bar{f}}_{k_{2}}(t,\eta)d\eta,\\
J_{k,k_{1},k_{2}}^{(2)}(t,\xi):=&i\psi_{k}(\xi)\int_{\R^3} e^{it\phi(\xi,\eta)}t\big(\partial_{\xi_{i}}\phi(\xi,\eta)\partial_{\xi_{j}}+\partial_{\xi_{j}}\phi(\xi,\eta)\partial_{\xi_{i}}\big)\hat{f}_{k_{1}}(t,\xi-\eta)\hat{\bar{f}}_{k_{2}}(t,\eta)d\eta,\\
J_{k,k_{1},k_{2}}^{(3)}(t,\xi):=&-\psi_{k}(\xi)\int_{\R^3}e^{it\phi(\xi,\eta)}t^{2}\partial_{\xi_{i}}\phi(\xi,\eta)\partial_{\xi_{j}}\phi(\xi,\eta)\hat{f}_{k_{1}}(t,\xi-\eta)\hat{\bar{f}}_{k_{2}}(t,\eta)d\eta,\\
J_{k,k_{1},k_{2}}^{(4)}(t,\xi):=&i\psi_{k}(\xi)\int_{\R^3} e^{it\phi(\xi,\eta)}t\partial_{\xi_{i},\xi_{j}}^{2}\phi(\xi,\eta)\hat{f}_{k_{1}}(t,\xi-\eta)\hat{\bar{f}}_{k_{2}}(t,\eta)d\eta.
}
}

Define
\begin{align}\label{e4.26}
\hat{Q}_{1}(t,\xi):=\sum_{\substack{k_{2}-k_{1}<10\\
k_{1},k_{2}\in \Z}}\int_{0}^{t}\int_{\R^3}e^{is|\xi|^{\alpha}}\hat{w}_{k_{1}}(s,\xi-\eta)\hat{\bar{w}}_{k_{2}}(s,\eta)d\eta ds,
\end{align}

\begin{align}\label{e4.27}
\hat{Q}_{2}(t,\xi)
&:=\bigg(\sum_{\substack{|k_{1}-k_{2}|< 10\\
k_{1},k_{2}\in \Z}}+\sum_{\substack{k_{1}-k_{2}\geq10\\
k_{1},k_{2}\in \Z}}\bigg)\int_{0}^{t}\int_{\R^3}e^{is|\xi|^{\alpha}}\hat{w}_{k_{1}}(s,\xi-\eta)(\overline{B(u,\bar{u})})^\wedge_{k_{2}}(s,\eta)d\eta ds\nonumber\\ &=:{\sum_{\ell=1}^{2}\hat{Q}_{2,\ell}(t,\xi),}
\end{align}

\begin{align}\label{e5.8}
\hat{Q}_{3}(t,\xi)&:=\bigg(\sum_{\substack{|k_{1}-k_{2}|< 10\\
k_{1},k_{2}\in \Z}}+\sum_{\substack{k_{1}-k_{2}\geq10\\
k_{1},k_{2}\in \Z}}\bigg)\int_{0}^{t}\int_{\R^3}e^{is|\xi|^{\alpha}}B(u,\bar{u})^\wedge_{k_{1}}(s,\xi-\eta)\hat{\bar{u}}_{k_{2}}(s,\eta)d\eta ds\nonumber\\
&=:{\sum_{\ell=1}^{2}\hat{Q}_{3,\ell}(t,\xi),}
\end{align}

\begin{align}\label{e4.29}
\hat{Q}_{4}(t,\xi):=\int_{0}^{t}e^{is|\xi|^{\alpha}}B(|u|^{2},\bar{u})^\wedge(s,\xi)ds,
\end{align}

\begin{align}\label{e4.30}
\hat{Q}_{5}(t,\xi)&:=\int_{0}^{t}e^{is|\xi|^{\alpha}}B(u,|u|^{2})^\wedge(s,\xi)ds\nonumber=\sum_{X=HH,HL,LH}\int_{0}^{t}e^{is|\xi|^{\alpha}}B(u,(\bar{u}u)_{X})^\wedge(s,\xi)\,ds\\
&=:\sum_{\ell=1}^{3}\hat{Q}_{5,\ell}(t,\xi),
\end{align}
where $B(u,\bar{u})$ is defined by \eqref{e2.2}. 

The proofs of Lemmas \ref{lem4.11}, \ref{lem4.12} and \ref{lem4.13} follow a similar approach to that in \cite{W}. Here, we just provide a brief outline of the proof. Let
$\{\chi_{k}^{\ell}\}_{\ell=1}^{2}$ be defined by \eqref{chi}.

\begin{lemma}\label{lem4.11}
Let $t>0$ and $\Wthr$ be defined as in \eqref{e4.2s}. Define $\kappa:=\min\{k,k_{2}\}$. Then for all $(k_{1},k_{2})\in \cup_{\ell=1}^{2}\chi_{k}^{\ell}$ and $w\in \Wthr$, we have
\begin{align*}
\|H_{k,k_{1},k_{2}}(t)\|_{2}\les\min\{2^{\frac{3}{2}\kappa}2^{\lambda k_{1}-2k_{1,+}}2^{\lambda k_{2}-2k_{2,+}},t^{-\frac{3}{2}}2^{(\lambda-\frac{3}{2}\alpha+\frac{3}{2})k_{1}-2k_{1,+}}2^{\lambda k_{2}-2k_{2,+}}\}\|w\|_{\Wthr}^{2}.
\end{align*}
\end{lemma}

\begin{proof}
It is a consequence of Lemmas \ref{lem3.2}, \ref{lem3.3} and \ref{lem3.4}.
\end{proof}

\begin{lemma}\label{lem4.12}
{Let $t>0$ and $\Wthr$ be defined as in \eqref{e4.2s}.}

(i) If $(k_{1},k_{2})\in \chi_{k}^{1}$, { for all $w\in \Wthr$}, we have
\begin{align}\label{e4.31}
\sum_{\ell=1}^{2}\|I_{k,k_{1},k_{2}}^{(\ell)}(t)\|_{2}\les2^{-k-4k_{1,+}}\min\{2^{\frac{3}{2}k} 2^{2\lambda k_{1}},t^{-\frac{3}{2}}2^{(2\lambda-\frac{3}{2}\alpha+\frac{3}{2})k_{1}} \}\|w\|_{\Wthr}^{2};
\end{align}

(ii) If $(k_{1},k_{2})\in \chi_{k}^{2}$, for all $w\in \Wthr$, we have
\begin{align}\label{e4.32}
\sum_{\ell=1}^{2}\|I_{k,k_{1},k_{2}}^{(\ell)}(t)\|_{2}\les2^{(\lambda-1) k-2k_{+}-2k_{2,+}}\min\{2^{(\lambda+\frac{3}{2}) k_{2}},t^{-\frac{3}{2}}2^{(\lambda-\frac{3}{2}\alpha+\frac{3}{2})k_{2}}\}\|w\|_{\Wthr}^{2}.
\end{align}
\end{lemma}

\begin{proof}
Let $\{\rho_{l}^{(\ell)}\}_{l=1}^{3}$ be given by Lemma \ref{lem3.12aa}, where $\ell\in\{1,2\}$.

First, we consider $(k_{1},k_{2})\in \chi_{k}^{1}$.
By integration by parts, we have
\begin{align*}
I_{k,k_{1},k_{2}}^{(2)}(t,\xi)&=\psi_{k}(\xi)\sum_{l=1}^{3}\int_{\R^3}e^{it\phi(\xi,\eta)}\partial_{\eta_{l}}\bigg(\rho_{l}^{(1)}(\xi,\eta)\frac{\partial_{\xi_{i}}\phi(\xi,\eta)}{\partial_{\eta_{l}}\phi(\xi,\eta)}\hat{f}_{k_{1}}(t,\xi-\eta)\hat{\bar{f}}_{k_{2}}(t,\eta)\bigg)d\eta.
\end{align*}
By Lemmas \ref{lem3.2}, \ref{lem3.3}, \ref{lem3.4} and \ref{lem3.8}, we can derive  \eqref{e4.31},
where we use $(2,2)$ and $(2,\infty)$ estimates.

Next, we consider $(k_{1},k_{2})\in \chi_{k}^{2}$. Again, by integration by parts, we have
\begin{align*}
I_{k,k_{1},k_{2}}^{(2)}(t,\xi)&=\psi_{k}(\xi)\sum_{l=1}^{3}\bigg[\int_{\R^3} e^{it\phi(\xi,\eta)}\rho_{l}^{(2)}(\xi,\eta)\frac{\partial_{\xi_{i}}\phi(\xi,\eta)}{\partial_{\eta_{l}}\phi(\xi,\eta)}\hat{f}_{k_{1}}(t,\xi-\eta)\partial_{\eta_{l}}\hat{\bar{f}}_{k_{2}}(t,\eta)d\eta\\
&+\int_{\R^3} e^{it\phi(\xi,\eta)}\partial_{\eta_{l}}\bigg(\rho_{l}^{(2)}(\xi,\eta)\frac{\partial_{\xi_{i}}\phi(\xi,\eta)}{\partial_{\eta_{l}}\phi(\xi,\eta)}\hat{f}_{k_{1}}(t,\xi-\eta)\bigg)\hat{\bar{f}}_{k_{2}}(t,\eta)d\eta\bigg].
\end{align*}
By Lemmas \ref{lem3.2}, \ref{lem3.3}, \ref{lem3.4}, and \ref{lem3.9}, we obtain
\begin{align*}
\sum_{\ell=1}^{2}\|I_{k,k_{1},k_{2}}^{(\ell)}(t)\|_{2}&\les2^{-k}2^{\lambda k_{2}-2k_{2,+}}\min\{t^{-\frac{3}{2}}2^{(\lambda-\frac{3}{2}\alpha+\frac{3}{2})k_{1}-2k_{1,+}},2^{\frac{3}{2}k_{2}}2^{\lambda k_{1}-2k_{1,+}}\}\|w\|_{\Wthr}^{2}\\
&+2^{(\lambda-1) k_{1}-2k_{1,+}}\min\{t^{-\frac{3}{2}}2^{(\lambda-\frac{3}{2}\alpha+\frac{3}{2})k_{2}-2k_{2,+}},2^{\frac{3}{2}k_{2}}2^{\lambda k_{2}-2k_{2,+}}\}\|w\|_{\Wthr}^{2}\\
&\les2^{(\lambda-1) k-2k_{+}}\min\{t^{-\frac{3}{2}}2^{(\lambda-\frac{3}{2}\alpha+\frac{3}{2})k_{2}-2k_{2,+}},2^{(\lambda+\frac{3}{2}) k_{2}-2k_{2,+}}\}\|w\|_{\Wthr}^{2},
\end{align*}
where we apply the $(\infty,2)$ estimate to the first term of $I_{k,k_{1},k_{2}}^{(2)}(t)$ and $(2,\infty)$ estimates to the remaining terms. This finishes the proof of \eqref{e4.32}.
\end{proof}

\begin{lemma}\label{lem4.13}
Let $t>0$ and $\Wthr$ be given by \eqref{e4.2s}.

(i) If $(k_{1},k_{2})\in \chi_{k}^{1}$, for all $w\in \Wthr$, we have
\begin{align}\label{e4.33}
\sum_{\ell=1}^{4}\|J_{k,k_{1},k_{2}}^{(\ell)}(t)\|_{2}\les2^{-2k-4k_{1,+}}\min\{2^{\frac{3}{2}k}2^{2\lambda k_{1}},t^{-\frac{3}{2}}2^{(2\lambda-\frac{3}{2}\alpha+\frac{3}{2})k_{1}}\}\|w\|_{\Wthr}^{2};
\end{align}

(ii) If $(k_{1},k_{2})\in \chi_{k}^{2}$, for all $w\in \Wthr$, we have
\begin{align}\label{e4.34}
\sum_{\ell=1}^{4}\|J_{k,k_{1},k_{2}}^{(\ell)}(t)\|_{2}\les 2^{(\lambda-2)k-2k_{+}-2k_{2,+}}\min\{2^{(\lambda+\frac{3}{2}) k_{2}},t^{-\frac{3}{2}}2^{(\lambda-\frac{3}{2}\alpha+\frac{3}{2})k_{2}}\}\|w\|_{\Wthr}^{2}.
\end{align}
\end{lemma}

\begin{proof}
The proof is similar to that of  Lemma \ref{lem4.12}. {Let  $(k_{1},k_{2})\in \chi_{k}^{\ell}$ and $\ell\in\{1,2\}$. By integration by parts, we can write $J_{k,k_{1},k_{2}}^{(2)}(t,\xi)$ as the sum of following terms,
\EQN{
\psi_{k}(\xi)\sum_{l=1}^{3}\int_{\R^3} e^{it\phi(\xi,\eta)}\rho_{l}^{(\ell)}(\xi,\eta)\frac{\partial_{\xi_{i}}\phi(\xi,\eta)\partial_{\xi_{j}}+\partial_{\xi_{j}}\phi(\xi,\eta)\partial_{\xi_{i}}}{\partial_{\eta_{l}}\phi(\xi,\eta)}\hat{f}_{k_{1}}(t,\xi-\eta)\partial_{\eta_{l}}\hat{\bar{f}}_{k_{2}}(t,\eta)d\eta,
}
\begin{align*}
\psi_{k}(\xi)\sum_{l=1}^{3}\int_{\R^3}e^{it\phi(\xi,\eta)}\partial_{\eta_{l}}\bigg(\rho_{l}^{(\ell)}(\xi,\eta)\frac{\partial_{\xi_{i}}\phi(\xi,\eta)\partial_{\xi_{j}}+\partial_{\xi_{j}}\phi(\xi,\eta)\partial_{\xi_{i}}}{\partial_{\eta_{l}}\phi(\xi,\eta)}\hat{f}_{k_{1}}(t,\xi-\eta)\bigg)\hat{\bar{f}}_{k_{2}}(t,\eta)d\eta,
\end{align*}
and we can write $J_{k,k_{1},k_{2}}^{(3)}(t,\xi)$ as the sum of
\EQN{
\psi_{k}(\xi)\sum_{l=1}^{3}\int_{\R^3}e^{it\phi(\xi,\eta)}\partial_{\eta_{l}}\bigg(\frac{1}{\partial_{\eta_{l}}\phi(\xi,\eta)}\partial_{\eta_{l}}\bigg(\big(\rho_{l}^{(\ell)}\frac{\partial_{\xi_{i}}\phi\partial_{\xi_{j}}\phi}{\partial_{\eta_{l}}\phi}\big)(\xi,\eta)\hat{f}_{k_{1}}(t,\xi-\eta)\bigg)\bigg)\hat{\bar{f}}_{k_{2}}(t,\eta)d\eta,
}
\EQN{
\psi_{k}(\xi)\sum_{l=1}^{3}\int_{\R^3}e^{it\phi(\xi,\eta)}\frac{1}{\partial_{\eta_{l}}\phi(\xi,\eta)}\partial_{\eta_{l}}\bigg(\big(\rho_{l}^{(\ell)}\frac{\partial_{\xi_{i}}\phi\partial_{\xi_{j}}\phi}{\partial_{\eta_{l}}\phi}\big)(\xi,\eta)\hat{f}_{k_{1}}(t,\xi-\eta)\bigg)\partial_{\eta_{l}}\hat{\bar{f}}_{k_{2}}(t,\eta)d\eta,
}
\EQN{
\psi_{k}(\xi)\sum_{l=1}^{3}\int_{\R^3}e^{it\phi(\xi,\eta)}\partial_{\eta_{l}}\bigg(\big(\rho_{l}^{(\ell)}\frac{\partial_{\xi_{i}}\phi\partial_{\xi_{j}}\phi}{(\partial_{\eta_{l}}\phi)^{2}}\big)(\xi,\eta)\hat{f}_{k_{1}}(t,\xi-\eta)\bigg)\partial_{\eta_{l}}\hat{\bar{f}}_{k_{2}}(t,\eta)d\eta,
}
\begin{align*}
\psi_{k}(\xi)\sum_{l=1}^{3}\int_{\R^3} e^{it\phi(\xi,\eta)}\big(\rho_{l}^{(\ell)}\frac{\partial_{\xi_{i}}\phi\partial_{\xi_{j}}\phi}{(\partial_{\eta_{l}}\phi)^{2}}\big)(\xi,\eta)\hat{f}_{k_{1}}(t,\xi-\eta)\partial_{\eta_{l}}^{2}\hat{\bar{f}}_{k_{2}}(t,\eta)d\eta.
\end{align*}
Moreover, we have
\begin{align*}
J_{k,k_{1},k_{2}}^{(4)}(t,\xi)
&=\psi_{k}(\xi)\sum_{l=1}^{3}\int_{\R^3} e^{it\phi(\xi,\eta)}\partial_{\eta_{l}}\bigg(\big(\rho_{l}^{(\ell)}\frac{\partial_{\xi_{i},\xi_{j}}^{2}\phi}{\partial_{\eta_{l}}\phi}\big)(\xi,\eta)\hat{f}_{k_{1}}(t,\xi-\eta)\hat{\bar{f}}_{k_{2}}(t,\eta)\bigg)d\eta.
\end{align*}
The proof of} \eqref{e4.33} and \eqref{e4.34} relies on Lemmas \ref{lem3.2}, \ref{lem3.3}, \ref{lem3.4}, \ref{lem3.8}, and \ref{lem3.9},
where we use $(4,4)$ estimates to the terms with $\nabla \hat{f}_{k_{1}}$ and $\nabla \hat{\bar{f}}_{k_{2}}$.
\end{proof}

\begin{lemma}\label{lem4.14}
Let $t>0$ and $\Wthr$ be given by \eqref{e4.2s}. Then there exists $\varepsilon=\varepsilon(\alpha,\lambda)>0$ such that for all $w\in W_{3}$, we have 
\begin{align*}
&2^{-\lambda k+2k_{+}}\sum_{(k_{1},k_{2})\in \chi_{k}^{1}\cup\chi_{k}^{2}}\int_{\frac{t}{2}}^{t}\|H_{k,k_{1},k_{2}}(s)\|_{2}ds\\
&+\sum_{\ell=1}^{2}2^{(1-\lambda) k+2k_{+}}\sum_{(k_{1},k_{2})\in \chi_{k}^{1}\cup\chi_{k}^{2}}\int_{\frac{t}{2}}^{t}\|I_{k,k_{1},k_{2}}^{(\ell)}(s)\|_{2}ds\\
&+\sum_{\ell=1}^{4}2^{(2-\lambda) k+2k_{+}}\sum_{(k_{1},k_{2})\in \chi_{k}^{1}\cup\chi_{k}^{2}}\int_{\frac{t}{2}}^{t}\|J_{k,k_{1},k_{2}}^{(\ell)}(s)\|_{2}ds\\
&\les (1+t)^{-\varepsilon}\|w\|_{\Wthr}^{2}.
\end{align*}
\end{lemma}

\begin{proof}
Let $\delta$ be given by Theorem \ref{thm:main}. By Lemmas \ref{lem4.11}, \ref{lem4.12}, \ref{lem4.13}, and \eqref{e4.5}, {there exists 
$\varepsilon\in(0,\delta)$ such that} the sum for $(k_{1},k_{2})\in \chi_{k}^{1}$ can be bounded by
\begin{align*}
t\sup_{s\in[\frac{t}{2},t]}C_{k,1}(s)\|w\|_{\Wthr}^{2}\les (1+t)^{-\varepsilon}\|w\|_{\Wthr}^{2}.
\end{align*}
By Lemmas \ref{lem4.11}, \ref{lem4.12}, \ref{lem4.13}, and \eqref{e4.6}, the sum for $(k_{1},k_{2})\in \chi_{k}^{2}$ can be bounded by
\begin{align*}
t\sup_{s\in[\frac{t}{2},t]}C_{k,2}(s)\|w\|_{\Wthr}^{2}
\les (1+t)^{-\delta}\|w\|_{\Wthr}^{2}.
\end{align*}
We have completed the proof of the lemma.
\end{proof}

\begin{lemma}\label{lem4.15}
Let $Q_{1}$ be defined as in \eqref{e4.26} and $W_{3}$ be defined as in \eqref{e4.2s}. Then we have
\begin{align*}
\sup_{t>0}\|Q_{1}(t)\|_{\PW}\les \|w\|_{W_{3}}^{2}.
\end{align*}
\end{lemma}

\begin{proof}
By \eqref{e4.7}, \eqref{e4.23}, \eqref{e4.9} and \eqref{e4.25}, the proof of Lemma \ref{lem4.15} follows from Lemma \ref{lem4.14}.
\end{proof}

\begin{lemma}\label{lem4.16}
Assume that \eqref{e1.1} and \eqref{e2.4} hold on $[0,\infty)$. Let $Q_{2}$ be defined as in \eqref{e4.27} and $\{(W_{i},U_{i})\}_{i=2}^{3}$ be defined by \eqref{e4.2s}, \eqref{e4.3s}. Then we have
\begin{align*}
\sup_{t>0}\|Q_{2}(t)\|_{\PW}\les \sum_{i=2}^{3}\|(w,u)\|_{W_{i}\times U_{i}}^{3}+\|(w,u)\|_{W_{i}\times U_{i}}^{4}.
\end{align*}
\end{lemma}

\begin{proof}
Without loss of generality, assume that $\sum_{i=2}^{3}\|(w,u)\|_{W_{i}\times U_{i}}^{3}+\|(w,u)\|_{W_{i}\times U_{i}}^{4}=1$. By \eqref{e4.27}, it suffices to estimate $\{\|Q_{2,\ell}(t)\|_{\PW}\}_{\ell=1}^{2}$. Let $(f,g)$ be the profiles of $(w,u)$. 

{\bf Case 1:} Estimates for $Q_{2,1}$.  By \eqref{e4.27}, we have
\begin{align}\label{e4.35}
\hat{Q}_{2,1}(t,\xi)
&=\sum_{\substack{|k_{1}-k_{2}|<10\\
k_{1},k_{2}\in \Z}}\int_{0}^{t}\int_{\R^3}e^{is\phi(\xi,\eta)}\hat{f}_{k_{1}}(s,\xi-\eta)(\overline{e^{isD^{\alpha}}B(u,\bar{u})})^\wedge_{k_{2}}(s,\eta)d\eta ds,
\end{align}
where $\phi$ is given by \eqref{e2.3}.
By \eqref{e4.35}, we can write $\psi_{k}(\xi)\partial_{\xi_{j}}\hat{Q}_{2,1}(t,\xi)$ as
\begin{align*}
&\psi_{k}(\xi)\sum_{(k_{1},k_{2})\in \chi_{k}^{1}}\bigg(i\int_{0}^{t}\int_{\R^3} e^{is\phi(\xi,\eta)}s\partial_{\xi_{j}}\phi(\xi,\eta)\hat{f}_{k_{1}}(s,\xi-\eta)(\overline{e^{isD^{\alpha}}B(u,\bar{u})})^\wedge_{k_{2}}(s,\eta)d\eta ds\\
&\hskip 2cm +\int_{0}^{t}\int_{\R^3} e^{is\phi(\xi,\eta)}\partial_{\xi_{j}}\hat{f}_{k_{1}}(s,\xi-\eta)(\overline{e^{isD^{\alpha}}B(u,\bar{u})})^\wedge_{k_{2}}(s,\eta)d\eta ds\bigg).
\end{align*}
Applying $\partial_{\xi_{i},\xi_{j}}^{2}$ to \eqref{e4.35}, we can write $\psi_{k}(\xi)\partial_{\xi_{i},\xi_{j}}^{2}\hat{Q}_{2,1}(t,\xi)$ as
\begin{align*}
&-\psi_{k}(\xi)\sum_{\substack{(k_{1},k_{2})\in\chi_{k}^{1}\\k_{1}> k+12}}\int_{0}^{t}\int_{\R^3}e^{is\phi(\xi,\eta)}s^{2}\aij(\xi,\eta)\hat{g}_{k_{1}}(s,\xi-\eta)(\overline{e^{isD^{\alpha}}B(u,\bar{u})})^\wedge_{k_{2}}(s,\eta)d\eta ds\\
&-\psi_{k}(\xi)\sum_{\substack{(k_{1},k_{2})\in\chi_{k}^{1}\\|k_{1}-k|\leq 12}}\int_{0}^{t}\int_{\R^3}e^{is\phi(\xi,\eta)}s^{2}\aij(\xi,\eta)\hat{g}_{k_{1}}(s,\xi-\eta)(\overline{e^{isD^{\alpha}}B(u,\bar{u})})^\wedge_{k_{2}}(s,\eta)d\eta ds\\
&-\psi_{k}(\xi)\sum_{(k_{1},k_{2})\in\chi_{k}^{1}}\int_{0}^{t}\int_{\R^3} e^{is\phi(\xi,\eta)}\bigg[is^{2}\aij(\xi,\eta) B(u,\bar{u})^\wedge_{k_{1}}(s,\xi-\eta)(\overline{e^{isD^{\alpha}}B(u,\bar{u})})^\wedge_{k_{2}}(s,\eta)\\
&\qquad \qquad \qquad \qquad -\partial_{\xi_{i},\xi_{j}}^{2}\hat{f}_{k_{1}}(s,\xi-\eta)(\overline{e^{isD^{\alpha}}B(u,\bar{u})})^\wedge_{k_{2}}(s,\eta)\\
&\qquad \qquad \qquad \qquad -is\Gammaone\big(\hat{f}_{k_{1}}(s,\xi-\eta)\big)(\overline{e^{isD^{\alpha}}B(u,\bar{u})})^\wedge_{k_{2}}(s,\eta)\bigg]d\eta ds\\
&=:\sum_{\ell=1}^{3}\psi_{k}(\xi)\MJ_{2,1}^{(\ell)}(t,\xi),
\end{align*}
where
\EQ{\label{e4.36}
\aij(\xi,\eta):=\partial_{\xi_{i}}\phi(\xi,\eta)\partial_{\xi_{j}}\phi(\xi,\eta)
}
and $\Gammaone $ is an operator  defined by
\EQ{\label{e4.37}
\Gammaone:=\partial_{\xi_{i}}\phi(\xi,\eta)\partial_{\xi_{j}}+\partial_{\xi_{j}}\phi(\xi,\eta)\partial_{\xi_{i}}+\partial_{\xi_{i},\xi_{j}}^{2}\phi(\xi,\eta).
}
By \eqref{e2.2} and integration by parts in $s$, we can write $\psi_{k}(\xi)\MJ_{2,1}^{(1)}(t,\xi)$ as
\begin{align*}
&-\psi_{k}(\xi)\sum_{\substack{|k_{2}-k_{1}|< 10\\k_{1}> k+12\\
(k_{3},k_{4})\in\chi_{k_{2}}^{3}}}\int_{0}^{t}\int_{\R^3}e^{is\Phione(\xi,\eta,\zeta)}s^{2} \aktwo(\xi,\eta,\zeta)\hat{g}_{k_{1}}(s,\xi-\eta-\zeta)\hat{\bar{g}}_{k_{3}}(s,\eta)\hat{g}_{k_{4}}(s,\zeta)d\eta d\zeta ds\\
&=i\psi_{k}(\xi)\sum_{\substack{|k_{2}-k_{1}|< 10\\k_{1}> k+12\\
(k_{3},k_{4})\in\chi_{k_{2}}^{3}}}\bigg[\int_{\R^3}e^{it\Phione(\xi,\eta,\zeta)}t^{2}\frac{\aktwo(\xi,\eta,\zeta)}{\Phione(\xi,\eta,\zeta)}\hat{g}_{k_{1}}(t,\xi-\eta-\zeta)\hat{\bar{g}}_{k_{3}}(t,\eta)\hat{g}_{k_{4}}(t,\zeta)d\eta d\zeta\\
&-\int_{0}^{t}\int_{\R^3} e^{is\Phione(\xi,\eta,\zeta)}\partial_{s}\bigg(s^{2}\frac{\aktwo(\xi,\eta,\zeta)}{\Phione(\xi,\eta,\zeta)}\hat{g}_{k_{1}}(s,\xi-\eta-\zeta)\hat{\bar{g}}_{k_{3}}(s,\eta)\hat{g}_{k_{4}}(s,\zeta)\bigg)d\eta d\zeta ds\bigg],
\end{align*}
where
\EQ{\label{e4.38}
\Phione(\xi,\eta,\zeta):=|\xi|^{\alpha}-|\xi-\eta-\zeta|^{\alpha}+|\eta|^{\alpha}-|\zeta|^{\alpha},
}
\EQ{\label{e4.39}
\aktwo(\xi,\eta,\zeta):=\frac{\partial_{\xi_{i}}\Phione(\xi,\eta,\zeta)\partial_{\xi_{j}}\Phione(\xi,\eta,\zeta)}{\phi(\eta+\zeta,\zeta)}\psi_{k_{2}}(\eta+\zeta).
}
By Lemma \ref{lem3.8} and some calculations, we have $$\partial_{\xi_{i}}\Phione\in 2^{(\alpha-1)k_{1}}S^{0}(\Omega^{(1)}_{k,k_{1},k_{2},k_{3},k_{4}})
$$ 
for all $i$ and
$\Phione\sim 2^{\alpha k_{1}}$ if the integrands do not vanish. Here $\Omega^{(1)}_{k,k_{1},k_{2},k_{3},k_{4}}$ denotes
\EQ{\label{omega1}
\{(\xi,\eta,\zeta)\in\R^{9}: 2^{-k}|\xi|, 2^{-k_{1}}|\xi-\eta-\zeta|, 2^{-k_{2}}|\eta+\zeta|, 2^{-k_{3}}|\eta|, 2^{-k_{4}}|\zeta|\in (1/2,2)\},
}
and  $ \tau S^{0}(\Omega) $ has been defined in  Definition \ref{def3.7}.
By \eqref{e2.2}, we can write $\psi_{k}(\xi)\MJ_{2,1}^{(2)}(t,\xi)$ as
\begin{align*}
&-\psi_{k}(\xi)\sum_{\substack{|k_{2}-k_{1}|<10\\|k_{1}-k|\leq 12\\(k_{3},k_{4})\in \chi_{k_{2}}^{(3)}}}\int_{0}^{t}\int_{\R^3}e^{is\Phione(\xi,\eta,\zeta)}s^{2} \aktwo(\xi,\eta,\zeta)\hat{g}_{k_{1}}(s,\xi-\eta-\zeta)\hat{\bar{g}}_{k_{3}}(s,\eta)\hat{g}_{k_{4}}(s,\zeta)d\eta d\zeta ds.
\end{align*}
We now have that $\partial_{\xi_{i}}\Phione,\partial_{\zeta_{i}}\Phione\in 2^{(\alpha-1)k}S^{0}(\Omega^{(1)}_{k,k_{1},k_{2},k_{3},k_{4}})$ for all $i$ and
$|\nabla_{\zeta}\Phione|\sim 2^{(\alpha-1)k}$
on $\Omega^{(1)}_{k,k_{1},k_{2},k_{3},k_{4}}$ by similar calculations in Case 1 of Lemma \ref{lem3.8}.
By Lemmas \ref{lem3.4}, \ref{lem3.8}, \eqref{e4.8s}, \eqref{e4.9s} and \eqref{e4.1s}, we have
\begin{align*}
\|\psi_{k}\hat{Q}_{2,1}(t)\|_{2}&\les \sup_{|l|\leq 10}\int_{0}^{t}\sum_{k_{1}\geq k-12}2^{\frac{3}{2}k}\|f_{k_{1}}(s)\|_{2}\|B(u,\bar{u})_{k_{1}+l}(s)\|_{2}ds\\
&\les \int_{0}^{t}(1+s)^{-1-\delta}ds\sum_{k_{1}\geq k-12}2^{\frac{3}{2}k}2^{(2\lambda-\alpha)k_{1}}2^{-4k_{1,+}}\\
&\les 2^{\lambda k-2k_{+}}2^{(\lambda-\alpha+\frac{3}{2})k}2^{-2k_{+}}\les 2^{\lambda k-2k_{+}},
\end{align*}
\begin{align*}
\|\psi_{k}\partial_{\xi_{j}}\hat{Q}_{2,1}(t)\|_{2}&\les \sup_{|l|\leq10}\sum_{k_{1}\geq k-12}\int_{0}^{t}\big(2^{-k_{1}}\|\hat{f}_{k_{1}}(s)\|_{2}+\|\nabla\hat{f}_{k_{1}}(s)\|_{2}\big)\cdot\\
&\quad\big(s2^{\alpha k_{1}}\|B(u,\bar{u})_{k_{1}+l}(s)\|_{\infty}+2^{\frac{3}{2}k}\|B(u,\bar{u})_{k_{1}+l}(s)\|_{2}\big)ds\\
&\les 2^{(\lambda-1)k-2k_{+}}.
\end{align*}
By Lemma \ref{lem3.5}, \eqref{e5.27A} and \eqref{e5.27AA}, we have
\begin{align*}
\|\psi_{k}\MJ_{2,1}^{(1)}(t)\|_{2}&\les \sum_{k_{1}\geq k-12}\bigg(t^{2}2^{-2k_{1}}\|g_{k_{1}}(t)\|_{2}\sum_{j\in\Z}\|u_{j}(t)\|_{\infty}\|u(t)\|_{\infty}\\
&\qquad\qquad+\int_{0}^{t}s2^{-2k_{1}}\|g_{k_{1}}(s)\|_{2}\sum_{j\in\Z}\|u_{j}(s)\|_{\infty}\|u(s)\|_{\infty}ds\\
&\quad\quad+\int_{0}^{t}s^{2}2^{-2k_{1}}\sup_{|l|\leq 22}\|g_{k_{1}+l}(s)\|_{2}{\sum_{j\in\Z}\|e^{-isD^{\alpha}}\partial_{s}g_{j}(s)\|_{\infty}}\sum_{j\in\Z}\|u_{j}(s)\|_{\infty}ds\bigg)\\
&\les\sum_{k_{1}\geq k-12}2^{(\lambda-2)k_{1}-2k_{1,+}}\les2^{(\lambda-2)k-2k_{+}}.
\end{align*}
We can apply Lemma \ref{lem4.2} together with integration by parts in $\zeta$ to estimate $\psi_{k}\MJ_{2,1}^{(2)}(t)$.
By $(2,\infty)$ estimates, Lemmas  \ref{lem3.8} and \ref{lem4.3}, we can obtain the estimate of $\psi_{k}\MJ_{2,1}^{(3)}(t)$ and then we have
\begin{align*}
\|\psi_{k}\partial^{2}_{\xi_{i},\xi_{j}}\hat{Q}_{2,1}(t)\|_{2}\leq\sum_{\ell=1}^{3}\|\psi_{k}\MJ_{2,1}^{(\ell)}(t)\|_{2}\les 2^{(\lambda-2)k-2k_{+}}.
\end{align*}

\medskip

{\bf Case 2:} Estimates for $Q_{2,2}$. By \eqref{e4.27}, we have
\begin{align}\label{e4.41}
\hat{Q}_{2,2}(t,\xi)=\sum_{\substack{k_{1}-k_{2}\geq10\\
k_{1},k_{2}\in \Z}}\psi_{k}(\xi)\int_{0}^{t}\int_{\R^3} e^{is\phi(\xi,\eta)}\hat{f}_{k_{1}}(s,\xi-\eta)(\overline{e^{isD^{\alpha}}B(u,\bar{u})})^\wedge_{k_{2}}(s,\eta)d\eta ds,
\end{align}
where $\phi$ is given by \eqref{e2.3}.
Let $\{\rho_{l}^{(2)}\}_{l=1}^{3}$ be given by Lemma \ref{lem3.12aa}.
By integration by parts, we can write $\psi_{k}(\xi)\partial_{\xi_{j}}\hat{Q}_{2,2}(t,\xi)$ as
\begin{align*}
&\psi_{k}(\xi)\sum_{(k_{1},k_{2})\in \chi_{k}^{2}}\bigg[\int_{0}^{t}\int_{\R^3} e^{is\phi(\xi,\eta)}\partial_{\xi_{j}}\hat{f}_{k_{1}}(s,\xi-\eta)(\overline{e^{isD^{\alpha}}B(u,\bar{u})})^\wedge_{k_{2}}(s,\eta)d\eta ds\\
&-\sum_{l=1}^{3}\int_{0}^{t}\int_{\R^3} e^{is\phi(\xi,\eta)}\partial_{\eta_{l}}\bigg(\big(\rho_{l}^{(2)}\frac{\partial_{\xi_{j}}\phi}{\partial_{\eta_{l}}\phi}\big)(\xi,\eta)\hat{f}_{k_{1}}(s,\xi-\eta)\bigg)(\overline{e^{isD^{\alpha}}B(u,\bar{u})})^\wedge_{k_{2}}(s,\eta)d\eta ds\\
&-\sum_{l=1}^{3}\int_{0}^{t}\int_{\R^3} e^{is\phi(\xi,\eta)}\big(\rho_{l}^{(2)}\frac{\partial_{\xi_{j}}\phi}{\partial_{\eta_{l}}\phi}\big)(\xi,\eta)\hat{f}_{k_{1}}(s,\xi-\eta)\partial_{\eta_{l}}(\overline{e^{isD^{\alpha}}B(u,\bar{u})})^\wedge_{k_{2}}(s,\eta)d\eta ds\bigg]\\
&=:\sum_{\ell=1}^{3}\psi_{k}(\xi)\MI_{2,2}^{(\ell)}(t,\xi).
\end{align*}
Applying $\partial_{\xi_{i},\xi_{j}}^{2}$ to \eqref{e4.41}, we can write $\psi_{k}(\xi)\partial_{\xi_{i},\xi_{j}}^{2}\hat{Q}_{2,2}(t,\xi)$ as
\begin{align*}
&\psi_{k}(\xi)\sum_{(k_{1},k_{2})\in \chi_{k}^{2}}\bigg(\int_{0}^{t}\int_{\R^3} e^{is\phi(\xi,\eta)}\partial_{\xi_{i},\xi_{j}}^{2}\hat{f}_{k_{1}}(s,\xi-\eta)(\overline{e^{isD^{\alpha}}B(u,\bar{u})})^\wedge_{k_{2}}(s,\eta)d\eta ds\\
&\qquad+i\int_{0}^{t}\int_{\R^3}e^{is\phi(\xi,\eta)}s\Gammaone\big(\hat{f}_{k_{1}}(s,\xi-\eta)\big)(\overline{e^{isD^{\alpha}}B(u,\bar{u})})^\wedge_{k_{2}}(s,\eta)d\eta ds\\
&\qquad-\int_{0}^{t}\int_{\R^3} e^{is\phi(\xi,\eta)}s^{2}\aij(\xi,\eta)\hat{f}_{k_{1}}(s,\xi-\eta)(\overline{e^{isD^{\alpha}}B(u,\bar{u})})^\wedge_{k_{2}}(s,\eta)d\eta ds\bigg)\\
&=:\sum_{\ell=1}^{3}\psi_{k}(\xi)\MJ_{2,2}^{(\ell)}(t,\xi),
\end{align*}
where $\aij$ and $\Gammaone$ are defined as in \eqref{e4.36} and \eqref{e4.37}.
Let $\{\rho_{l}^{(3)}\}_{l=1}^{3}$ be given by Lemma \ref{lem3.12aa}. By a change of variable and integration by parts, we can write $\psi_{k}(\xi)\MJ_{2,2}^{(3)}(t,\xi)$ as
\begin{align*}
&\sum_{l=1}^{3}\sum_{\substack{(k_{1},k_{2})\in \chi_{k}^{2}\\
(k_{3},k_{4})\in \chi_{k_{2}}^{3}}}
\psi_{k}(\xi)\int_{0}^{t}\int_{\R^6} e^{is\Phione(\xi,\eta)}s^{2}\aktwo(\xi,\eta,\zeta)
(\rho_{l}^{(3)}\hat{f}_{k_{1}})(s,\xi-\eta-\zeta)\hat{\bar{g}}_{k_{3}}(s,\eta)\hat{g}_{k_{4}}(s,\zeta)d\eta d\zeta ds\\
&=\psi_{k}(\xi)\sum_{l=1}^{3}\sum_{\substack{(k_{1},k_{2})\in \chi_{k}^{2}\\
(k_{3},k_{4})\in \chi_{k_{2}}^{3}}}\int_{0}^{t}\!\!\int_{\R^6} e^{is\Phione(\xi,\eta)}\bigg[\frac{\aktwo(\xi,\eta,\zeta)}{\big(\partial_{\eta_{l}}\Phione\partial_{\zeta_{l}}\Phione\big)(\xi,\eta,\zeta)}
\partial_{\xi_{l}}^{2}(\rho_{l}^{(3)}\hat{f}_{k_{1}})(s,\xi-\eta-\zeta)\\
&\hskip 10cm \times\hat{\bar{g}}_{k_{3}}(s,\eta)\hat{g}_{k_{4}}(s,\zeta)\\
&\quad-\partial_{\xi_{l}}(\rho_{l}^{(3)}\hat{f}_{k_{1}})(s,\xi-\eta-\zeta)\hat{\bar{g}}_{k_{3}}(s,\eta)\partial_{\zeta_{l}}\bigg(\frac{\aktwo(\xi,\eta,\zeta)}{\big(\partial_{\eta_{l}}\Phione\partial_{\zeta_{l}}\Phione\big)(\xi,\eta,\zeta)}
\hat{g}_{k_{4}}(s,\zeta)\bigg)\\
&\quad-\partial_{\xi_{l}}(\rho_{l}^{(3)}\hat{f}_{k_{1}})(s,\xi-\eta-\zeta)(\partial_{\zeta_{l}}\Phione)^{-1}(\xi,\eta,\zeta)\partial_{\eta_{l}}\bigg(\frac{\aktwo(\xi,\eta,\zeta)}{\partial_{\eta_{l}}\Phione(\xi,\eta,\zeta)}
\hat{\bar{g}}_{k_{3}}(s,\eta)\bigg)\hat{g}_{k_{4}}(s,\zeta)\\
&\quad+(\rho_{l}^{(3)}\hat{f}_{k_{1}})(s,\xi-\eta-\zeta)\hat{\bar{g}}_{k_{3}}(s,\eta)\partial_{\zeta_{l}}\bigg((\partial_{\zeta_{l}}\Phione)^{-1}(\xi,\eta,\zeta)\partial_{\eta_{l}}\big(\frac{\aktwo}{\partial_{\eta_{l}}\Phione}\big)(\xi,\eta,\zeta)\hat{g}_{k_{4}}(s,\zeta)\bigg)\\
&\quad+(\rho_{l}^{(3)}\hat{f}_{k_{1}})(s,\xi-\eta-\zeta)
\partial_{\eta_{l}}\hat{\bar{g}}_{k_{3}}(s,\eta)\partial_{\zeta_{l}}\bigg(\frac{\aktwo(\xi,\eta,\zeta)}{\big(\partial_{\eta_{l}}\Phione\partial_{\zeta_{l}}\Phione\big)(\xi,\eta,\zeta)}
\hat{g}_{k_{4}}(s,\zeta)\bigg)\bigg]d\eta d\zeta ds\\
&=:\sum_{\ell=1}^{5}\psi_{k}(\xi)\MJ_{2,2}^{(3,\ell)}(t,\xi),
\end{align*}
where $\Phione$ and $\aktwo$ are defined by \eqref{e4.38} and \eqref{e4.39}. We observe that for all $i$ and $l$, $\partial_{\xi_{i}}\Phione\in2^{(\alpha-2)k}2^{k_{2}}S^{0}(\Omega^{(1)}_{k,k_{1},k_{2},k_{3},k_{4}})$, while $\partial_{\eta_{l}}\Phione,\partial_{\zeta_{l}}\Phione\in2^{(\alpha-1)k}S^{0}(\Omega^{(1)}_{k,k_{1},k_{2},k_{3},k_{4}})$. Moreover, $|\partial_{\eta_{l}}\Phione|\sim |\partial_{\zeta_{l}}\Phione|\sim2^{(\alpha-1)k}$ if the integrands are not zero.
Definition \ref{def3.7} provides the precise formulation of the symbol class and $\Omega^{(1)}_{k,k_{1},k_{2},k_{3},k_{4}}$ is defined by \eqref{omega1}. 
By Lemma \ref{lem3.3} and \eqref{e5.28A}, we have
\begin{align*}
&\|\psi_{k}\hat{Q}_{2,2}(t)\|_{2}\les \sum_{|k_{1}-k|\leq 2}\int_{0}^{t}\|f_{k_{1}}(s)\|_{2}\sum_{k_{2}\in\Z}\big\|B(u,\bar{u})_{k_{2}}(s)\big\|_{\infty}ds\les 2^{\lambda k-2k_{+}}.
\end{align*}
$\{\psi_{k}\MI_{2,2}^{(\ell)}(t)\}_{\ell=1}^{2}$ can be estimated similarly to $\psi_{k}\hat{Q}_{2,2}(t)$.
By Lemmas \ref{lem3.3} and \ref{lem4.3}, we have
\begin{align*}
\|\psi_{k}\MI_{2,2}^{(3)}(t)\|_{2}&\les \sup_{|k_{1}-k|\leq 2}\int_{0}^{t}\sum_{k_{2}\leq k_{1}-10}2^{k_{2}-k}\|f_{k_{1}}(s)\|_{2}\|e^{-isD^{\alpha}}(x e^{isD^{\alpha}}B(u,\bar u)_{k_{2}})(s)\|_{\infty}ds\\
&\les 2^{(\lambda-1) k-2k_{+}}\int_{0}^{t}\sum_{k_{2}\in\Z}\min\{(1+s)^{-1-\delta},2^{(\lambda+\frac{3}{2})k_{2}-2k_{2,+}}\}ds\\
&\les 2^{(\lambda-1)k-2k_{+}}.
\end{align*}
These tell us that
\begin{align*}
\|\psi_{k}\partial_{\xi_{j}}\hat{Q}_{2,2}(t)\|_{2}\leq \sum_{\ell=1}^{3}\|\psi_{k}\MI_{2,2}^{(\ell)}(t)\|_{2}\les2^{(\lambda-1)k-2k_{+}}.
\end{align*}
$\{\psi_{k}\MJ_{2,2}^{(\ell)}(t)\}_{\ell=1}^{2}$ can be estimated similarly to $\hat{Q}_{2,2}(t)$ and $\partial_{\xi_{j}}\hat{Q}_{2,2}(t)$.
By Bernstein's inequality, combining Lemmas \ref{lem3.5} and \ref{lem4.2} with \eqref{e5.27A}, we obtain
\begin{align*}
&\sum_{\ell=1}^{3}\|\psi_{k}\MJ_{2,2}^{(3,\ell)}(t)\|_{2}\\
&\les \sup_{\substack{|k_{1}-k|\leq2\\
1\leq l\leq3}}\sum_{k_{2}\leq k }2^{k_{2}-k}2^{(\frac{3}{2}-\alpha)k_{2}}\sup_{|k_{4}-k_{2}|\leq2}\bigg(\int_{0}^{t}\|\nabla^{2}(\rho_{l}\hat{f}_{k_{1}})(s)\|_{2}\sum_{k_{3}\in\Z}\|u_{k_{3}}(s)\|_{\infty}\|g_{k_{4}}(s)\|_{2}ds\\
&\quad+ \int_{0}^{t}2^{-k}\|\nabla(\rho_{l}\hat{f}_{k_{1}})(s)\|_{2}\sum_{k_{3}\in\Z}\|u_{k_{3}}(s)\|_{\infty}\big(\|g_{k_{4}}(s)\|_{2}+2^{k_{2}}\|\nabla\hat{g}_{k_{4}}(s)\|_{2}\big)ds\\
&\quad+2^{-\frac{1}{2} k_{2}}\int_{0}^{t}2^{-k}\|\nabla(\rho_{l}\hat{f}_{k_{1}})(s)\|_{2}\sum_{k_{3}\leq k_{2} }\big(2^{\frac{1}{2}k_{3}}\|g_{ k_{3}}(s)\|_{2}+2^{\frac{3}{2}k_{3}}\|\nabla\hat{g}_{ k_{3}}(s)\|_{2}\big)\|u_{k_{4}}(s)\|_{\infty}ds\bigg)\\
&\les 2^{(\lambda-2)k-2k_{+}}\sum_{k_{2}\leq k}2^{k_{2}-k}\bigg(2^{(\lambda+\frac{3}{2}-\alpha)k_{2}-2k_{2,+}}
+ \sum_{k_{3}\in\Z}2^{(\lambda+\frac{3}{2}-\alpha)k_{3}-2k_{3,+}}\bigg)\\
&\les2^{(\lambda-2)k-2k_{+}}.
\end{align*}
By $(\infty,\infty,2)$, $(2,\infty,\infty)$ estimates, Lemma \ref{lem4.2} and \eqref{e4.6}, we have
\begin{align*}
&\sum_{\ell=4}^{5}\|\psi_{k}\MJ_{2,2}^{(3,\ell)}(t)\|_{2}\\
&\les 2^{-2k}\sum_{k_{2}\leq k}\int_{0}^{t}\min\{s^{-\frac{3}{2}}2^{(\lambda-\frac{3}{2}\alpha+\frac{3}{2})k-2k_{+}}2^{\lambda k_{2}-2k_{2,+}},2^{\lambda k-2k_{+}}2^{(\lambda+\frac{3}{2})k_{2}-2k_{2,+}}\}ds\\
&\les 2^{(\lambda-2) k-2k_{+}}\int_{0}^{t}\sum_{k_{2}\in\Z}2^{-2k_{2,+}}\min\{2^{(\lambda+\frac{3}{2}) k_{2}},s^{-\frac{3}{2}}2^{(\lambda-\frac{3}{2}\alpha+\frac{3}{2})k_{2}}\}ds\\
&\les 2^{(\lambda-2) k-2k_{+}}.
\end{align*}
These, together with the estimates for $\{\psi_{k}\MJ_{2,2}^{(\ell)}(t)\}_{\ell=1}^{2}$, imply that 
\begin{align*}
\|\psi_{k}\partial^{2}_{\xi_{i},\xi_{j}}\hat{Q}_{2,2}(t)\|_{2}\les 2^{(\lambda-2)k-2k_{+}}.
\end{align*}
\end{proof}

\begin{lemma}\label{lem4.17}
Suppose that \eqref{e1.1} and \eqref{e2.4} hold on $[0,\infty)$. Let $Q_{3}$ be defined as in \eqref{e5.8} and $\{(W_{i},U_{i})\}_{i=2}^{3}$ be defined as in \eqref{e4.2s}, \eqref{e4.3s}. Then we have
\begin{align*}
\sup_{t>0}\|Q_{3}(t)\|_{\PW}\les\sum_{i=2}^{3}\left(\|(w,u)\|_{W_{i}\times U_{i}}^{3}+\|(w,u)\|_{W_{i}\times U_{i}}^{4}\right).
\end{align*}
\end{lemma}

\begin{proof}
Without loss of generality, assume that $\sum_{i=2}^{3}\left(\|(w,u)\|_{W_{i}\times U_{i}}^{3}+\|(w,u)\|_{W_{i}\times U_{i}}^{4}\right)=1$. By \eqref{e5.8}, it suffices to estimate $\{\|Q_{3,\ell}(t)\|_{\PW}\}_{\ell=1}^{2}$.

\smallskip

{\bf Case 1:} Estimates for $Q_{3,1}(t)$. By Lemma \ref{lem4.16}, it suffices to show
\begin{align*}
\hat{R}_{3,1}(t,\xi):=\sum_{\substack{|k_{1}-k_{2}|<10\\
k_{1},k_{2}\in \Z}}\int_{0}^{t}\int_{\R^3} e^{is|\xi|^{\alpha}}B(u,\bar u)^\wedge_{k_{1}}(s,\xi-\eta)\overline{B(u,\bar u)}^\wedge_{k_{2}}(s,\eta)d\eta ds
\end{align*}
satisfies the desired estimate.
By applying $\partial_{\xi_{j}}$ to $\hat{R}_{3,1}(t,\xi)$, we can write $\psi_{k}(\xi)\partial_{\xi_{j}}\hat{R}_{3,1}(t,\xi)$ as
\begin{align*}
&\psi_{k}(\xi)\sum_{(k_{1},k_{2})\in \chi_{k}^{1}}\bigg(\int_{0}^{t}\int e^{is\phi(\xi,\eta)} \partial_{\xi_{j}}\big(e^{is|\xi-\eta|^{\alpha}}B(u,\bar u)^\wedge_{k_{1}}(s,\xi-\eta)\big)e^{-is|\eta|^{\alpha}}\overline{B(u,\bar u)}^\wedge_{k_{2}}(s,\eta)d\eta ds\\
&\qquad+i\int_{0}^{t}\int e^{is\phi(\xi,\eta)}s\partial_{\xi_{j}}\phi(\xi,\eta) e^{is|\xi-\eta|^{\alpha}}B(u,\bar u)^\wedge_{k_{1}}(s,\xi-\eta)e^{-is|\eta|^{\alpha}}\overline{B(u,\bar u)}^\wedge_{k_{2}}(s,\eta)d\eta ds\bigg).
\end{align*}
By integration by parts, we can write $\psi_{k}(\xi)\partial_{\xi_{i},\xi_{j}}^{2}\hat{R}_{3,1}(t,\xi)$ as
\begin{align*}
&\psi_{k}(\xi)\sum_{(k_{1},k_{2})\in \chi_{k}^{1}}\bigg(-\int_{0}^{t}\int_{\R^3} e^{is\phi(\xi,\eta)} s^{2}\aij(\xi,\eta) e^{is|\xi-\eta|^{\alpha}}B(u,\bar u)^\wedge_{k_{1}}(s,\xi-\eta)\\
&\hskip 8cm \times e^{-is|\eta|^{\alpha}}\overline{B(u,\bar u)}^\wedge_{k_{2}}(s,\eta)d\eta ds\\
&+i\int_{0}^{t}\int_{\R^3} e^{is\phi(\xi,\eta)}s \Gammatwo\big(e^{is|\xi-\eta|^{\alpha}}B(u,\bar u)^\wedge_{k_{1}}(s,\xi-\eta)\big)e^{-is|\eta|^{\alpha}}\overline{B(u,\bar u)}^\wedge_{k_{2}}(s,\eta)d\eta ds\\
&+\int_{0}^{t}\int_{\R^3}e^{is\phi(\xi,\eta)} \partial_{\xi_{i}}\big(e^{is|\xi-\eta|^{\alpha}}B(u,\bar u)^\wedge_{k_{1}}(s,\xi-\eta)\big)\partial_{\eta_{j}}\big(e^{-is|\eta|^{\alpha}}\overline{B(u,\bar u)}^\wedge_{k_{2}}(s,\eta)\big)d\eta ds\bigg),
\end{align*}
where $\aij$ is given by \eqref{e4.36} and $\Gammatwo$ is an operator defined by
\EQN{
\Gammatwo:=\partial_{\xi_{i}}\phi(\xi,\eta)\partial_{\xi_{j}}+\partial_{\xi_{j}}\phi(\xi,\eta)\partial_{\xi_{i}}+\partial_{\xi_{i},\xi_{j}}^{2}\phi(\xi,\eta) +\partial_{\eta_{j}}\phi(\xi,\eta) \partial_{\xi_{i}}.
}
By Lemmas \ref{lem3.3}, \ref{lem3.8} and \ref{lem4.3}, we have
\begin{align*}
\|\psi_{k}\hat{R}_{3,1}(t)\|_{2}&\les2^{-2k_{+}}\sum_{k_{1}\geq k-12}\int_{0}^{t}\min\{(1+s)^{-1-\delta}2^{\frac{3}{2}k}2^{(2\lambda-\alpha) k_{1}},(1+s)^{-2-2\delta}2^{(\lambda-\alpha) k_{1}}\} ds\\
&\les2^{\lambda k-2k_{+}}\int_{0}^{t}(1+s)^{-1-\delta}\min\{2^{(\lambda-\alpha+\frac{3}{2}) k},2^{-\alpha k}\} ds
\les2^{\lambda k-2k_{+}},
\end{align*}
\begin{align*}
&\|\psi_{k}\nabla^{\ell}\hat{R}_{3,1}(t)\|_{2}\les 2^{-(\ell-1)k-2k_{+}}\sum_{k_{1}\geq k-12} 2^{(\lambda-1)k_{1}}\int_{0}^{t}(1+s)^{-1-\delta}ds\les 2^{(\lambda-\ell)k-2k_{+}},
\end{align*}
where we use $(2,\infty)$ estimates for $\{\psi_{k}\nabla^{\ell}\hat{R}_{3,1}(t)\}_{\ell=1}^{2}$.

\smallskip

{\bf Case 2:} Estimates for $Q_{3,2}$. By \eqref{e5.8} and \eqref{e2.2}, we have
\begin{align*}
&\psi_{k}(\xi)\hat{Q}_{3,2}(t,\xi)=\psi_{k}(\xi)\sum_{(k_{1},k_{2})\in\chi_{k}^{2}}\int_{0}^{t}\int_{\R^3} e^{is\phi(\xi,\eta)}e^{is|\xi-\eta|^{\alpha}}B(u,\bar u)^\wedge_{k_{1}}(s,\xi-\eta)\hat{\bar{g}}_{k_{2}}(s,\eta)d\eta ds\\
&=\psi_{k}(\xi)\sum_{\substack{(k_{1},k_{2})\in\chi_{k}^{2}\\
(k_{3},k_{4})\in\chi_{k_{1}}^{3}}}\int_{0}^{t}\int_{\R^6} e^{is\Phitwo(\xi,\eta,\zeta)}\bkone\hat{g}_{k_{3}}(s,\zeta)\hat{\bar{g}}_{k_{4}}(s,\xi-\eta-\zeta)\hat{\bar{g}}_{k_{2}}(s,\eta)d\eta d\zeta ds,
\end{align*}
where
\EQ{\label{e4.42}
\phione(\xi,\eta,\zeta):=|\xi-\eta|^{\alpha}-|\zeta|^{\alpha}+|\xi-\eta-\zeta|^{\alpha},
}
\EQ{\label{e4.43}
\Phitwo(\xi,\eta,\zeta):=|\xi|^{\alpha}-|\zeta|^{\alpha}
+|\xi-\eta-\zeta|^{\alpha}+|\eta|^{\alpha}.
}
Note that $|\Phitwo|\sim|\phi_{1}|\sim2^{\alpha k}$ on $\Omega^{(2)}_{k,k_{1},k_{2},k_{3},k_{4}}$ and $\partial_{\xi_{i}}\Phitwo,\partial_{\zeta_{i}}\Phitwo\in2^{(\alpha-1)k}S^{0}(\Omega^{(2)}_{k,k_{1},k_{2},k_{3},k_{4}})$ for all $i$, where the symbol class under consideration is defined in Definition \ref{def3.7} and 
\EQN{
\Omega^{(2)}_{k,k_{1},k_{2},k_{3},k_{4}}:=\{(\xi,\eta,\zeta)\in\R^{9}: 2^{-k}|\xi|,2^{-k_{1}}|\xi-\eta|, 2^{-k_{2}}|\eta|,2^{-k_{3}}|\zeta|,2^{-k_{4}}|\xi-\eta-\zeta|\in(1/2,2)\}.
}
By \eqref{e5.8}, we can write $\psi_{k}(\xi)\partial_{\xi_{j}}\hat{Q}_{3,2}(t,\xi)$ as
\begin{align*}
&\psi_{k}(\xi)\sum_{(k_{1},k_{2})\in\chi_{k}^{2}}\bigg(\int_{0}^{t}\int_{\R^3} e^{is\phi(\xi,\eta)} \partial_{\xi_{j}}\big(e^{is|\xi-\eta|^{\alpha}}B(u,\bar u)^\wedge_{k_{1}}(s,\xi-\eta)\big)\hat{\bar{g}}_{k_{2}}(s,\eta)d\eta ds\\
&\qquad+i\int_{0}^{t}\int_{\R^3} e^{is\phi(\xi,\eta)}s\partial_{\xi_{j}}\phi(\xi,\eta) e^{is|\xi-\eta|^{\alpha}}B(u,\bar u)^\wedge_{k_{1}}(s,\xi-\eta)\hat{\bar{g}}_{k_{2}}(s,\eta)d\eta ds\bigg).
\end{align*}
By integration by parts, we can write $\psi_{k}(\xi)\partial_{\xi_{i},\xi_{j}}^{2}\hat{Q}_{3,2}(t,\xi)$ as the sum of following terms,
\begin{align*}
&\psi_{k}(\xi)\sum_{\substack{(k_{1},k_{2})\in\chi_{k}^{2}\\
(k_{3},k_{4})\in\chi_{k_{1}}^{3}}}\bigg[-\int_{0}^{t}\int_{\R^6} e^{is\Phitwo(\xi,\eta,\zeta)} s^{2}\big(\partial_{\xi_{i}}\Phitwo\partial_{\xi_{j}}\Phitwo\big)(\xi,\eta,\zeta)\hat{g}_{k_{3}}(s,\zeta)\bkone\\
&\quad \quad \quad \quad \quad \quad \quad \quad \quad \quad \times\hat{\bar{g}}_{k_{4}}(s,\xi-\eta-\zeta)\hat{\bar{g}}_{k_{2}}(s,\eta)d\eta d\zeta ds\\
&+i\int_{0}^{t}\int_{\R^6} e^{is\Phitwo(\xi,\eta,\zeta)}s\hat{g}_{k_{3}}(s,\zeta)\Gammathree
\bigg(\bkone\hat{\bar{g}}_{k_{4}}(s,\xi-\eta-\zeta)\bigg)\hat{\bar{g}}_{k_{2}}(s,\eta)d\eta d\zeta ds\\
&+i\int_{0}^{t}\int_{\R^6} e^{is\Phitwo(\xi,\eta,\zeta)}s\partial_{\zeta_{j}}\Phitwo\hat{g}_{k_{3}}(s,\zeta)
\bkone\partial_{\xi_{i}}\hat{\bar{g}}_{k_{4}}(s,\xi-\eta-\zeta)\hat{\bar{g}}_{k_{2}}(s,\eta)d\eta d\zeta ds\\
&+\int_{0}^{t}\int_{\R^6} e^{is\Phitwo(\xi,\eta,\zeta)}\hat{g}_{k_{3}}(s,\zeta)
\partial_{\xi_{j}}\bigg(\partial_{\xi_{i}}\big(\bkone\big)\hat{\bar{g}}_{k_{4}}(s,\xi-\eta-\zeta)\bigg)\hat{\bar{g}}_{k_{2}}(s,\eta)d\eta d\zeta ds\\
&+\int_{0}^{t}\int_{\R^6} e^{is\Phitwo(\xi,\eta,\zeta)}\hat{g}_{k_{3}}(s,\zeta)
\partial_{\xi_{j}}\big(\bkone\big)\partial_{\xi_{i}}\hat{\bar{g}}_{k_{4}}(s,\xi-\eta-\zeta)\hat{\bar{g}}_{k_{2}}(s,\eta)d\eta d\zeta ds\\
&+\int_{0}^{t}\int_{\R^6} e^{is\Phitwo(\xi,\eta,\zeta)}\partial_{\zeta_{j}}\big(\hat{g}_{k_{3}}(s,\zeta)
\bkone\big)\partial_{\xi_{i}}\hat{\bar{g}}_{k_{4}}(s,\xi-\eta-\zeta)\hat{\bar{g}}_{k_{2}}(s,\eta)d\eta d\zeta ds\bigg]\\
&=:\sum_{\ell=1}^{6}\psi_{k}(\xi)\MJ_{3,2}^{(\ell)}(t,\xi),
\end{align*}
where $\phi_{1},\Phi_{2}$ are defined as in \eqref{e4.42}, \eqref{e4.43}, and $\Gammathree$ is an operator defined by
\EQN{
\Gammathree:=\partial_{\xi_{i}}\Phitwo(\xi,\eta,\zeta)\partial_{\xi_{j}}+\partial_{\xi_{j}}\Phitwo(\xi,\eta,\zeta)\partial_{\xi_{i}}+\partial_{\xi_{i},\xi_{j}}^{2}\Phitwo(\xi,\eta,\zeta).
}
By integration by parts in $s$, we can express $\psi_{k}(\xi)\MJ_{3,2}^{(1)}(t,\xi)$ as
\begin{align*}
&\psi_{k}(\xi)\sum_{\substack{(k_{1},k_{2})\in\chi_{k}^{2}\\
(k_{3},k_{4})\in\chi_{k_{1}}^{3}}}\bigg[i\int_{\R^6} e^{it\Phitwo(\xi,\eta,\zeta)} t^{2}\frac{\Akone(\xi,\eta,\zeta)}{\Phitwo(\xi,\eta,\zeta)}\hat{g}_{k_{3}}(t,\zeta)\hat{\bar{g}}_{k_{4}}(t,\xi-\eta-\zeta)\hat{\bar{g}}_{k_{2}}(t,\eta)d\eta d\zeta\\
&-i\int_{0}^{t}\int_{\R^6} e^{is\Phitwo(\xi,\eta,\zeta)}\frac{\Akone(\xi,\eta,\zeta)}{\Phitwo(\xi,\eta,\zeta)}\partial_{s}\bigg(s^{2}\hat{g}_{k_{3}}(s,\zeta)\hat{\bar{g}}_{k_{4}}(s,\xi-\eta-\zeta)\hat{\bar{g}}_{k_{2}}(s,\eta)\bigg)d\eta d\zeta ds\bigg],
\end{align*}
where
\EQN{
\Akone(\xi,\eta,\zeta):=\frac{\partial_{\xi_{i}}\Phitwo(\xi,\eta,\zeta)\partial_{\xi_{j}}\Phitwo(\xi,\eta,\zeta)}{\phione(\xi,\eta,\zeta)}\psi_{k_{1}}(\xi-\eta).
}
By \eqref{e5.27A}, Lemmas \ref{lem3.9} and \ref{lem4.3}, we have
\begin{align*}
&\sum_{|\beta|\leq 1}2^{|\beta|k}\|\psi_{k}\partial_{\xi}^{\beta}\hat{Q}_{3,2}(t)\|_{2}\\
&\les\sup_{|k_{1}-k|\leq 2}\int_{0}^{t}\big(\|B(u,\bar{u})_{k_{1}}(s)\|_{2}+2^{k_{1}}\|\partial_{\xi_{j}}\big(e^{is|\xi|^{\alpha}}B(u,\bar u)^\wedge_{k_{1}}(s,\xi)\big)\|_{2}\big)\sum_{k_{2}\in\Z}\|u_{k_{2}}(s)\|_{\infty}ds\\
&\quad+\sup_{|k_{1}-k|\leq 2}\int_{0}^{t}s2^{\alpha k_{1}}\|B(u,\bar{u})_{k_{1}}(s)\|_{2}\sum_{k_{2}\in\Z}\|u_{k_{2}}(s)\|_{\infty}ds\\
&\les 2^{\lambda k-2k_{+}}.
\end{align*}
By Lemmas \ref{lem4.5}, \ref{lem4.2} and $(\infty,2,\infty)$ estimates, we have
\begin{align*}
\|\psi_{k}\partial^{2}_{\xi_{i},\xi_{j}}\hat{Q}_{3,2}(t)\|_{2}\leq\sum_{\ell=1}^{6}\|\psi_{k}\MJ_{3,2}^{(\ell)}(t)\|_{2}\les 2^{(\lambda-2)k-2k_{+}},
\end{align*}
{where we use \eqref{e4.7a} to $\{\psi_{k}\MJ_{3,2}^{(\ell)}(t)\}_{\ell=4}^{6}$.}
\end{proof}

\begin{lemma}\label{lem4.18}
Suppose that \eqref{e1.1} and \eqref{e2.4} hold on $[0,\infty)$. Let $Q_{4}$ be defined as in \eqref{e4.29} and $\{(W_{i},U_{i})\}_{i=1}^{3}$ be defined as in \eqref{e4.2s}, \eqref{e4.3s}. Then we have
\begin{align*}
\sup_{t>0}\|Q_{4}(t)\|_{\PW}\les\sum_{i=2}^{3}\big(\|(w,u)\|_{W_{i}\times U_{i}}^{3}+\|(w,u)\|_{W_{i}\times U_{i}}^{4}\big).
\end{align*}
\end{lemma}

\begin{proof}
Without loss of generality, assume that $\sum_{i=2}^{3}\big(\|(w,u)\|_{W_{i}\times U_{i}}^{3}+\|(w,u)\|_{W_{i}\times U_{i}}^{4}\big)=1$.
By \eqref{e2.2}, we have
\begin{align*}
\psi_{k}(\xi)\hat{Q}_{4}(t,\xi)=\psi_{k}(\xi)\sum_{(k_{1},k_{2})\in \chi_{k}^{3}}
\int_{0}^{t}\int_{\R^3}e^{is\phi(\xi,\eta)}\phi(\xi,\eta)^{-1}e^{is|\xi-\eta|^{\alpha}}(|u|^2)_{k_{1}}^\wedge(s,\xi-\eta)\hat{\bar{g}}_{k_{2}}(s,\eta)d\eta ds,
\end{align*}
where $\phi$ is defined as in \eqref{e2.3}.
By integration by parts, we can write $\psi_{k}(\xi)\partial_{\xi_{j}}\hat{Q}_{4}(t,\xi)$ as
\begin{align*}
&\psi_{k}(\xi)\sum_{(k_{1},k_{2})\in \chi_{k}^{3}}\bigg(\int_{0}^{t}\int_{\R^3} e^{is\phi(\xi,\eta)}s\big(\frac{\partial_{\xi_{j}}\phi+\partial_{\eta_{j}}\phi}{\phi}\big)(\xi,\eta)e^{is|\xi-\eta|^{\alpha}}(|u|^2)_{k_{1}}^\wedge(s,\xi-\eta)\hat{\bar{g}}_{k_{2}}(s,\eta)d\eta ds\\
&+\int_{0}^{t}\int_{\R^3}e^{is\phi(\xi,\eta)}\big(\partial_{\xi_{j}}+\partial_{\eta_{j}}\big)\big(\phi^{-1}\big)(\xi,\eta)e^{is|\xi-\eta|^{\alpha}}(|u|^2)_{k_{1}}^\wedge(s,\xi-\eta)\hat{\bar{g}}_{k_{2}}(s,\eta)d\eta ds\\
&+\int_{0}^{t}\int_{\R^3} e^{is\phi(\xi,\eta)}\phi(\xi,\eta)^{-1}e^{is|\xi-\eta|^{\alpha}}(|u|^2)_{k_{1}}^\wedge(s,\xi-\eta)\partial_{\eta_{j}}\hat{\bar{g}}_{k_{2}}(s,\eta)d\eta ds\bigg)\\
&=:\sum_{\ell=1}^{3}\psi_{k}(\xi)\MI_{4}^{(\ell)}(t,\xi).
\end{align*}
By integration by parts, we can write $\psi_{k}(\xi)\partial_{\xi_{i},\xi_{j}}^{2}\hat{Q}_{4}(t,\xi)$ as the sum of following terms,
\begin{align*}
&\psi_{k}(\xi)\sum_{(k_{1},k_{2})\in \chi_{k}^{3}}\bigg(i\int_{0}^{t}\int_{\R^3} e^{is\phi(\xi,\eta)}sb_{ij}(\xi,\eta)e^{is|\xi-\eta|^{\alpha}}(|u|^2)_{k_{1}}^\wedge(s,\xi-\eta)\hat{\bar{g}}_{k_{2}}(s,\eta)d\eta ds\\
&\quad+i\int_{0}^{t}\int_{\R^3} e^{is\phi(\xi,\eta)}se^{is|\xi-\eta|^{\alpha}}(|u|^2)_{k_{1}}^\wedge(s,\xi-\eta)(\Gamma_{ij}^{(4)}+\Gamma_{ji}^{(4)})\hat{\bar{g}}_{k_{2}}(s,\eta)d\eta ds\\
&\quad+\int_{0}^{t}\int_{\R^3}e^{is\phi(\xi,\eta)}e^{is|\xi-\eta|^{\alpha}}(|u|^2)_{k_{1}}^\wedge(s,\xi-\eta)(\Gamma_{ij}^{(5)}+\Gamma_{ji}^{(5)})\hat{\bar{g}}_{k_{2}}(s,\eta)d\eta ds\\
&\quad+\int_{0}^{t}\int_{\R^3}e^{is\phi(\xi,\eta)}(\partial_{\xi_{i}}+\partial_{\eta_{i}})\big(\partial_{\xi_{j}}+\partial_{\eta_{j}}\big)\big(\phi^{-1}\big)(\xi,\eta)e^{is|\xi-\eta|^{\alpha}}(|u|^2)_{k_{1}}^\wedge(s,\xi-\eta)\hat{\bar{g}}_{k_{2}}(s,\eta)d\eta ds\\
&\quad+\int_{0}^{t}\int_{\R^3} e^{is\phi(\xi,\eta)}\phi(\xi,\eta)^{-1}e^{is|\xi-\eta|^{\alpha}}(|u|^2)_{k_{1}}^\wedge(s,\xi-\eta)\partial_{\eta_{i},\eta_{j}}^{2}\hat{\bar{f}}_{k_{2}}(s,\eta)d\eta ds\\
&\quad+i\int_{0}^{t}\int_{\R^3} e^{is\phi(\xi,\eta)}\phi(\xi,\eta)^{-1}e^{is|\xi-\eta|^{\alpha}}(|u|^2)_{k_{1}}^\wedge(s,\xi-\eta)\partial_{\eta_{i},\eta_{j}}^{2}\big(e^{-is|\eta|^{\alpha}}\overline{B(u,\bar{u})}^\wedge_{k_{2}}(s,\eta)\big)d\eta ds\\
&\quad-\int_{0}^{t}\int_{\R^3} e^{is\phi(\xi,\eta)}s^{2}\frac{c_{ij}(\xi,\eta)}{\phi(\xi,\eta)}e^{is|\xi-\eta|^{\alpha}}(|u|^2)_{k_{1}}^\wedge(s,\xi-\eta)\hat{\bar{g}}_{k_{2}}(s,\eta)d\eta ds\bigg)\\
&=:\sum_{\ell=1}^{7}\psi_{k}(\xi)\MJ_{4}^{(\ell)}(t,\xi),
\end{align*}
where
\EQN{
b_{ij}(\xi,\eta):=(\partial_{\xi_{i}}+\partial_{\eta_{i}})\big(\frac{\partial_{\xi_{j}}\phi+\partial_{\eta_{j}}\phi}{\phi}\big)(\xi,\eta)+(\partial_{\xi_{i}}\phi+\partial_{\eta_{i}}\phi)(\xi,\eta)\big(\partial_{\xi_{j}}+\partial_{\eta_{j}}\big)\big(\phi^{-1}\big)(\xi,\eta),
}
\EQN{
c_{ij}(\xi,\eta):=(\partial_{\xi_{i}}\phi+\partial_{\eta_{i}}\phi)(\xi,\eta)(\partial_{\xi_{j}}\phi+\partial_{\eta_{j}}\phi)(\xi,\eta),
}
and $\{\Gamma_{ij}^{(\ell)}\}_{\ell=4}^{5}$ are operators defined by
\EQN{
\Gamma_{ij}^{(4)}:=\big(\frac{\partial_{\xi_{i}}\phi+\partial_{\eta_{i}}\phi}{\phi}\big)(\xi,\eta)\partial_{\eta_{j}},\, \, \Gamma_{ij}^{(5)}:=\big(\partial_{\xi_{i}}+\partial_{\eta_{i}}\big)\big(\phi^{-1}\big)(\xi,\eta)\partial_{\eta_{j}}.
}
We further decompose $\psi_{k}(\xi)\MJ_{4}^{(7)}(t,\xi)$ into the following four terms,
\begin{align*}
&-\psi_{k}(\xi)\sum_{(k_{1},k_{2})\in \chi_{k}^{3}}\bigg[
\int_{0}^{t}\int_{\R^3} e^{is\phi(\xi,\eta)}s^{2}\frac{c_{ij}(\xi,\eta)}{\phi(\xi,\eta)} e^{is|\xi-\eta|^{\alpha}}((u\bar{u})_{HL+LH})_{k_{1}}^\wedge(s,\xi-\eta)\hat{\bar{g}}_{k_{2}}(s,\eta)d\eta ds\\
&+\sum_{\substack{(k_{3},k_{4})\in\chi_{k_{1}}^{1}\\
k_{3}\geq k+20}}
\int_{0}^{t}\int_{\R^3} e^{is\phi(\xi,\eta)}s^{2}\frac{c_{ij}(\xi,\eta)}{\phi(\xi,\eta)}  e^{is|\xi-\eta|^{\alpha}}(u_{k_{3}}\bar{u}_{k_{4}})_{k_{1}}^\wedge(s,\xi-\eta)\hat{\bar{g}}_{k_{2}}(s,\eta)d\eta ds\\
&+\sum_{\substack{(k_{3},k_{4})\in\chi_{k_{1}}^{1}\\
k_{3}\leq k-20}}
\int_{0}^{t}\int_{\R^3} e^{is\phi(\xi,\eta)}s^{2}\frac{c_{ij}(\xi,\eta)}{\phi(\xi,\eta)} e^{is|\xi-\eta|^{\alpha}}(u_{k_{3}}\bar{u}_{k_{4}})_{k_{1}}^\wedge(s,\xi-\eta)\hat{\bar{g}}_{k_{2}}(s,\eta)d\eta ds\\
&+\sum_{\substack{(k_{3},k_{4})\in\chi_{k_{1}}^{1}\\
|k_{3}-k|<20}}
\int_{0}^{t}\int_{\R^3} e^{is\phi(\xi,\eta)}s^{2}\frac{c_{ij}(\xi,\eta)}{\phi(\xi,\eta)} e^{is|\xi-\eta|^{\alpha}}(u_{k_{3}}\bar{u}_{k_{4}})_{k_{1}}^\wedge(s,\xi-\eta)\hat{\bar{g}}_{k_{2}}(s,\eta)d\eta ds\bigg]\\
&=:\sum_{\ell=1}^{4}\psi_{k}(\xi)\MJ_{4}^{(7,\ell)}(t,\xi).
\end{align*}
To deal with $\phi^{-1}$,
we employ Lemma \ref{lem5.1} following an approach similar to that used in Lemma \ref{lem4.3}. Consequently,
by Lemmas \ref{lem3.3}, \ref{lem3.4}, \ref{lem4.3}, \ref{lem4.5}, and \eqref{e4.1s}, there exists $\varepsilon>0$ such that
\begin{align*}
\|\psi_{k}\hat{Q}_{4}(t)\|_{2}&\les 2^{\lambda k-2k_{+}}
\int_{0}^{t}2^{-\alpha k}\min\{\sum_{k_{1}\leq k+10}s^{-1}2^{(\lambda+\frac{3}{2})k_{1}}2^{-2k_{1,+}},(1+s)^{-2-2\delta}\}ds\\
&\les
2^{\lambda k-2k_{+}}\int_{0}^{t}\min\{s^{-1}2^{(\lambda-\alpha+\frac{3}{2})k},2^{-\alpha k}(1+s)^{-2-2\delta}\}ds\\
&\les
2^{\lambda k-2k_{+}}\int_{0}^{t}\min\{s^{-1+\varepsilon},s^{-1-\varepsilon}\}ds\les2^{\lambda k-2k_{+}}.
\end{align*}
By applying $(\infty,2)$ estimates to $\psi_{k}\MI_{4}^{(1)}(t)$ and the same approach used for $\psi_{k}\hat{Q}_{4}(t)$ to $\{\psi_{k}\MI_{4}^{(\ell)}(t)\}_{\ell=2}^{3}$, we obtain the desired estimate for $\psi_{k}\partial_{\xi_{j}}\hat{Q}_{4}(t)$.
By $(\infty,2)$ estimates and the same approach as for $\psi_{k}\hat{Q}_{4}(t)$, we have
\begin{align}\label{e4.44}
\sum_{\ell=1}^{6}\|\psi_{k}\MJ_{4}^{(\ell)}(t)\|_{2}\les 2^{(\lambda-2)k-2k_{+}}.
\end{align}
By integration by parts in $s$,  we can express $\psi_{k}(\xi)\MJ_{4}^{(7,1)}(t,\xi)$ as
\begin{align*}
&-\sum_{\substack{|k_{2}-k|\leq 2\\
|k_{3}-k_{4}|\geq 10}}\psi_{k}(\xi)\int_{0}^{t}\int_{\R^6} e^{is\Phithr(\xi,\eta,\zeta)}s^{2}\Akthr(\xi,\eta)\hat{g}_{k_{3}}(s,\xi-\eta-\zeta)\hat{\bar{g}}_{k_{4}}(s,\zeta)\hat{\bar{g}}_{k_{2}}(s,\eta)d\eta d\zeta ds\\
&=\psi_{k}(\xi)\sum_{\substack{|k_{2}-k|\leq 2\\
|k_{3}-k_{4}|\geq 10}}\bigg[i\int_{\R^6} e^{it\Phithr(\xi,\eta,\zeta)}t^{2}\frac{\Akthr(\xi,\eta)}{\Phithr(\xi,\eta,\zeta)}\hat{g}_{k_{3}}(t,\xi-\eta-\zeta)\hat{\bar{g}}_{k_{4}}(t,\zeta)\hat{\bar{g}}_{k_{2}}(t,\eta)d\eta d\zeta\\
&\quad-i\int_{0}^{t}\int_{\R^6} e^{is\Phithr(\xi,\eta,\zeta)}\partial_{s}\bigg(s^{2}\frac{\Akthr(\xi,\eta)}{\Phithr(\xi,\eta,\zeta)}\hat{g}_{k_{3}}(s,\xi-\eta-\zeta)\hat{\bar{g}}_{k_{4}}(s,\zeta)\hat{\bar{g}}_{k_{2}}(s,\eta)\bigg)d\eta d\zeta ds\bigg],
\end{align*}
where
$\Phithr(\xi,\eta,\zeta):=|\xi|^{\alpha}+|\eta|^{\alpha}-|\xi-\eta-\zeta|^{\alpha}+|\zeta|^{\alpha}\sim 2^{\alpha k}$ on $\Omega^{(3)}_{k,k_{2},k_{3},k_{4}}$ and
\EQ{\label{e4.45}
\Akthr(\xi,\eta):=\frac{(\partial_{\xi_{i}}\phi+\partial_{\eta_{i}}\phi)(\xi,\eta)(\partial_{\xi_{j}}\phi+\partial_{\eta_{j}}\phi)(\xi,\eta)}{\phi(\xi,\eta)}\psi_{\leq k_{2}-10}(\xi-\eta).
}
Here, $\Omega^{(3)}_{k,k_{2},k_{3},k_{4}}$ consists of all 
$(\xi,\eta,\zeta)\in\R^{9}$ satisfying
\EQ{\label{e4.46}
2^{-k}|\xi|, 2^{-k_{2}}|\eta|, 2^{-k_{3}}|\xi-\eta-\zeta|, 2^{-k_{4}}|\zeta|\in(1/2,2),2^{-k_{2}+10}|\xi-\eta|\leq2.
}
In the above sum, we must have $\max\{k_{3},k_{4}\}\leq k$ and $\psi_{\leq k_{2}-10}(\xi-\eta)$ satisfies the condition stated in Lemma \ref{lem3.3}. {In view of \eqref{e5.27A}, \eqref{e5.27AA} and \eqref{e5.13A}, we can use $(\infty,\infty,2)$ estimates to bound} $\|\psi_{k}\MJ_{4}^{(7,1)}(t)\|_{2}$ by
\begin{align*}
&2^{-2k}\sup_{|k_{2}-k|\leq2}\sum_{\substack{j_{1}\leq 
j_{2}\\
j_{2}\leq 
k}}
\bigg[\int_{0}^{t}s(\|u_{j_{1}}(s)\|_{\infty}+s\|e^{-isD^{\alpha}}\partial_{s}g_{j_{1}}(s)\|_{\infty})(\|u_{j_{2}}(s)\|_{\infty}+s\|e^{-isD^{\alpha}}\partial_{s}g_{j_{2}}(s)\|_{\infty})\\
&\qquad \times(\|g_{k_{2}}(s)\|_{2}+s\|e^{-isD^{\alpha}}\partial_{s}g_{k_{2}}(s)\|_{2})ds
+
t^{2}\|u_{j_{1}}(t)\|_{\infty}\|u_{j_{2}}(t)\|_{\infty}\|g_{k_{2}}(t)\|_{2}\bigg]\\
&\les2^{-2k}\sup_{|k_{2}-k|\leq2}\bigg(
\int_{0}^{t}(1+s)^{-1-\frac{\delta}{2}}(\|g_{k_{2}}(s)\|_{2}+s\|e^{-isD^{\alpha}}\partial_{s}g_{k_{2}}(s)\|_{2})ds
+\|g_{k_{2}}(t)\|_{2}\bigg)\\
&\les 2^{(\lambda-2)k-2k_{+}}.
\end{align*}
Let $\{\rho_{l}^{(3)}\}_{l=1}^{3}$ be given by Lemma \ref{lem3.12aa}.
By integration by parts in $\eta$, we can write $\psi_{k}(\xi)\MJ_{4}^{(7,2)}(t,\xi)$ as
\begin{align*}
&\psi_{k}(\xi)\sum_{l=1}^{3}\sum_{\substack{|k_{2}-k|\leq 2\\
|k_{3}-k_{4}|< 10\\k_{3}\geq k+20}}\bigg[-i\int_{0}^{t}\int_{\R^6} e^{is\Phithr(\xi,\eta,\zeta)}s\frac{\Akthr(\xi,\eta)}{\partial_{\eta_{l}}\Phithr(\xi,\eta,\zeta)}\partial_{\eta_{l}}\big(
\rho_{l}^{(3)}\hat{g}_{k_{3}}\big)(s,\xi-\eta-\zeta)\\
&\hskip 7cm \times \hat{\bar{g}}_{k_{4}}(s,\zeta)\hat{\bar{g}}_{k_{2}}(s,\eta)d\eta d\zeta ds\\
&-i\int_{0}^{t}\int_{\R^6} e^{is\Phithr(\xi,\eta,\zeta)}s\big(\rho_{l}^{(3)}\hat{g}_{k_{3}}\big)(s,\xi-\eta-\zeta)\hat{\bar{g}}_{k_{4}}(s,\zeta)\partial_{\eta_{l}}\bigg(\frac{\Akthr(\xi,\eta)}{\partial_{\eta_{l}}\Phithr(\xi,\eta,\zeta)}\hat{\bar{g}}_{k_{2}}(s,\eta)\bigg)d\eta d\zeta ds\bigg],
\end{align*}
where $\partial_{\eta_{l}}\Phithr\in 2^{(\alpha-1)k_{3}}S^{0}(\Omega^{(3)}_{k,k_{2},k_{3},k_{4}}
)$ and $|\partial_{\eta_{l}}\Phithr|\sim2^{(\alpha-1)k_{3}} $ if the integrands are not  zero. Here $\Akthr$ and $\Omega^{(3)}_{k,k_{2},k_{3},k_{4}}$ are defined by \eqref{e4.45} and \eqref{e4.46}.
Then by $(2,\infty,\infty)$ and $(\infty,\infty,2)$ estimates, we have
\begin{align*}
\|\psi_{k}\MJ_{4}^{(7,2)}(t)\|_{2}\les 2^{(\lambda-2)k-2k_{+}}\sum_{k_{3}\geq k+10}2^{-(\alpha-1)(k_{3}-k)}
\les 2^{(\lambda-2)k-2k_{+}}.
\end{align*}
By integration by parts in $s$, we can write $\psi_{k}(\xi)\MJ_{4}^{(7,3)}(t,\xi)$ as
\begin{align*}
&\psi_{k}(\xi)\sum_{\substack{|k_{2}-k|\leq 2\\
|k_{3}-k_{4}|< 10\\
k_{3}\leq k-20}}\bigg[i\int_{\R^6} e^{it\Phithr(\xi,\eta,\zeta)}t^{2}\frac{\Akthr(\xi,\eta)}{\Phithr(\xi,\eta,\zeta)}\hat{g}_{k_{3}}(t,\xi-\eta-\zeta)\hat{\bar{g}}_{k_{4}}(t,\zeta)\hat{\bar{g}}_{k_{2}}(t,\eta)d\eta d\zeta\\
&-i\int_{0}^{t}\int_{\R^6} e^{is\Phithr(\xi,\eta,\zeta)}\partial_{s}\bigg(s^{2}\frac{\Akthr(\xi,\eta)}{\Phithr(\xi,\eta,\zeta)}\hat{g}_{k_{3}}(s,\xi-\eta-\zeta)\hat{\bar{g}}_{k_{4}}(s,\zeta)\hat{\bar{g}}_{k_{2}}(s,\eta)\bigg)d\eta d\zeta ds\bigg],
\end{align*}
where $\Akthr$ is given by \eqref{e4.45} and $|\Phithr|\sim2^{\alpha k}$ on $\Omega^{(3)}_{k,k_{2},k_{3},k_{4}}$.
Then $\psi_{k}\MJ_{4}^{(7,3)}(t)$ can be estimated in a similar manner to $\psi_{k}\MJ_{4}^{(7,1)}(t)$.
We can express $\psi_{k}(\xi)\MJ_{4}^{(7,4)}(t,\xi)$ as
\begin{align*}
&\psi_{k}(\xi)\sum_{\substack{|k_{2}-k|\leq 2\\
|k_{3}-k_{4}|<10\\|k_{3}- k|< 20}}\bigg[-\int_{0}^{t}\int_{\R^6} e^{is\Phithr(\xi,\eta,\zeta)}s^{2}\Akkthr(\xi,\eta)\hat{g}_{k_{3}}(s,\xi-\eta-\zeta)\hat{\bar{g}}_{k_{4}}(s,\zeta)\hat{\bar{g}}_{k_{2}}(s,\eta)d\eta d\zeta ds\\
&\hskip 2cm-\int_{0}^{t}\int_{\R^6} e^{is\Phithr(\xi,\eta,\zeta)}s^{2}\Bkkthr(\xi,\eta)\hat{g}_{k_{3}}(s,\xi-\eta-\zeta)\hat{\bar{g}}_{k_{4}}(s,\zeta)\hat{\bar{g}}_{k_{2}}(s,\eta)d\eta d\zeta ds\bigg],
\end{align*}
where
$$
\Akkthr(\xi,\eta):=\frac{c_{ij}(\xi,\eta)}{\phi(\xi,\eta)}\psi_{\leq k_{2}-10}(\xi-\eta)\psi_{\leq k-100}(\xi-\eta),$$
$$
\Bkkthr(\xi,\eta):=\frac{c_{ij}(\xi,\eta)}{\phi(\xi,\eta)}\psi_{\leq k_{2}-10}(\xi-\eta)(1-\psi_{\leq k-100}(\xi-\eta)).
$$
For the first term, we have $|\Phithr|\sim 2^{\alpha k}$ on $\Omega^{(3)}_{k,k_{2},k_{3},k_{4}}$ and we can apply integration by parts in $s$ combined with $(\infty,\infty,2)$ estimates. For the second term, we have $|\xi-\eta|\sim 2^{k}$,  then $\nabla_{\zeta}\Phithr\in 2^{(\alpha-1)k}S^{0}(\Omega^{(3)}_{k,k_{2},k_{3},k_{4}})$ and $|\nabla_{\zeta}\Phithr|\sim 2^{(\alpha-1)k}$ if the integrand does not vanish. Definition \ref{def3.7} provides the precise formulation of the symbol class.
We can use integration by parts in $\zeta$. Then we can obtain the desired estimate for $\psi_{k}\MJ_{4}^{(7,4)}(t)$.
These complete the proof of $\psi_{k}\MJ_{4}^{(7)}(t)$ and then by \eqref{e4.44}, we have
\begin{align*}
\|\psi_{k}\partial^{2}_{\xi_{i},\xi_{j}}\hat{Q}_{4}(t)\|_{2}\leq\sum_{\ell=1}^{7}\|\psi_{k}\MJ_{4}^{(\ell)}(t)\|_{2}
\les 2^{(\lambda-2) k-2k_{+}}.
\end{align*}
\end{proof}

\begin{lemma}\label{lem4.19}
Suppose that \eqref{e1.1} and \eqref{e2.4} hold on $[0,\infty)$. Let $Q_{5}$ be defined as in \eqref{e4.30}, and let $\{(W_{i},U_{i})\}_{i=1}^{3}$ be defined as in \eqref{e4.2s} and \eqref{e4.3s}. Then we have
\begin{align*}
\sup_{t>0}\|Q_{5}(t)\|_{\PW}\les\sum_{i=1}^{3}\left(\|(w,u)\|_{W_{i}\times U_{i}}^{3}+\|(w,u)\|_{W_{i}\times U_{i}}^{4}\right).
\end{align*}
\end{lemma}

\begin{proof}
Without loss of generality,
assume that $\sum_{i=1}^{3}\left(\|(w,u)\|_{W_{i}\times U_{i}}^{3}+\|(w,u)\|_{W_{i}\times U_{i}}^{4}\right)=1$.
By \eqref{e2.2}, we have
\begin{align*}
\psi_{k}(\xi)\hat{Q}_{5}(t,\xi)=\psi_{k}(\xi)\sum_{(k_{1},k_{2})\in\chi_{k}^{3}}\int_{0}^{t}\int_{\R^3} e^{is\phi(\xi,\eta)}\phi(\xi,\eta)^{-1}\hat{g}_{k_{1}}(s,\xi-\eta)e^{-is|\eta|^{\alpha}}(|u|^{2})_{k_{2}}^\wedge(s,\eta)d\eta ds,
\end{align*}
and we can write $\psi_{k}(\xi)\partial_{\xi_{j}}\hat{Q}_{5}(t,\xi)$ as
\begin{align*}
&\psi_{k}(\xi)\sum_{(k_{1},k_{2})\in\chi_{k}^{3}}\bigg(i\int_{0}^{t}\int_{\R^3} e^{is\phi(\xi,\eta)}s\frac{\partial_{\xi_{j}}\phi(\xi,\eta)}{\phi(\xi,\eta)}\hat{g}_{k_{1}}(s,\xi-\eta)e^{-is|\eta|^{\alpha}}(|u|^{2})_{k_{2}}^\wedge(s,\eta)d\eta ds\\
&\hskip 3cm+\int_{0}^{t}\int_{\R^3} e^{is\phi(\xi,\eta)}\partial_{\xi_{j}}\big(\phi(\xi,\eta)^{-1}\hat{g}_{k_{1}}(s,\xi-\eta)\big)(|u|^{2})_{k_{2}}^\wedge(s,\eta)\big)d\eta ds\bigg).
\end{align*}
Define $k_{-}:=\min\{k,0\}$. Then by Lemma \ref{lem4.2}, \eqref{e5.13A} and \eqref{e4.1s}, we have
\begin{align*}
\sum_{|\beta|\leq1}2^{|\beta|k}\|\psi_{k}\partial_{\xi}^{\beta}\hat{Q}_{5}(t)\|_{2}&\les \int_{0}^{t}2^{(\lambda-\alpha+\frac{3}{2})k_{-}}\min\{s^{-1}2^{\lambda k-2k_{+}},(1+s)^{-1-\delta}2^{-2k_{+}}\}ds\\
&{\quad+\int_{0}^{t}s(1+s)^{-1-\delta}s^{-1}2^{\lambda k-2k_{+}}ds}\\
&\les2^{\lambda k-2k_{+}}.
\end{align*}

Next, we consider $\psi_{k}\partial_{\xi_{i},\xi_{j}}^{2}\hat{Q}_{5}(t)$. By \eqref{e4.30}, it suffices to estimate $\{\psi_{k}\partial_{\xi_{i},\xi_{j}}^{2}\hat{Q}_{5,\ell}(t)\}_{\ell=1}^{3}$.

{\bf Case 1:} Estimates for $\psi_{k}\partial_{\xi_{i},\xi_{j}}^{2}\hat{Q}_{5,1}(t)$.  {Let $\{\rho_{l}^{(1)}\}_{l=1}^{3}$ be given by Lemma \ref{lem3.12aa}.}
By integration by parts and \eqref{e2.2}, we have
\begin{align*}
&\psi_{k}(\xi)\partial_{\xi_{i},\xi_{j}}^{2}\hat{Q}_{5,1}(t,\xi)\\
&=\psi_{k}(\xi)\sum_{\substack{(k_{1},k_{2})\in\chi_{k}^{3}\\
(k_{3},k_{4})\in\chi_{k_{2}}^{1}}}\int_{0}^{t}\int_{\R^6} e^{is\Phifo(\xi,\eta,\zeta)}\bigg[\partial_{\xi_{i},\xi_{j}}^{2}\big(\frac{\psi_{k_{2}}(\xi-\eta)}{\phi(\xi,\xi-\eta)}\big)\hat{g}_{k_{1}}(s,\eta)\hat{\bar{g}}_{k_{3}}(s,\zeta)\hat{g}_{k_{4}}(s,\xi-\eta-\zeta)\\
&\quad+\sum_{1\leq l\leq 3}\hat{g}_{k_{1}}(s,\eta)\partial_{\zeta_{l}}\bigg(\big(\Gamma_{ij}^{(6,l)}+\Gamma_{ji}^{(6,l)}\big)\big(\frac{\psi_{k_{2}}(\xi-\eta)}{\phi(\xi,\xi-\eta)}\big)\hat{\bar{f}}_{k_{3}}(s,\zeta)\hat{f}_{k_{4}}(s,\xi-\eta-\zeta)\bigg)\\
&\quad+is\partial_{\xi_{i},\xi_{j}}^{2}\Phifo(\xi,\eta,\zeta)\frac{\psi_{k_{2}}(\xi-\eta)}{\phi(\xi,\xi-\eta)}\hat{g}_{k_{1}}(s,\eta)\hat{\bar{f}}_{k_{3}}(s,\zeta)\hat{g}_{k_{4}}(s,\xi-\eta-\zeta)\\
&\quad+is\hat{g}_{k_{1}}(s,\eta)\hat{\bar{f}}_{k_{3}}(s,\zeta)\Gamma_{ij}^{(7)}\big(\frac{\psi_{k_{2}}(\xi-\eta)}{\phi(\xi,\xi-\eta)}\big)e^{is|\xi-\eta-\zeta|^{\alpha}}B(u,\bar{u})^\wedge_{k_{4}}(s,\xi-\eta-\zeta)\\
&\quad+is\hat{g}_{k_{1}}(s,\eta)e^{-is|\zeta|^{\alpha}}\overline{B(u,\bar{u})}^\wedge_{k_{3}}(s,\zeta)\Gamma_{ij}^{(8)}
\big(\frac{\psi_{k_{2}}(\xi-\eta)}{\phi(\xi,\xi-\eta)}\hat{g}_{k_{4}}(s,\xi-\eta-\zeta)\big)\\
&\quad+is\partial_{\zeta_{i}}\Phifo(\xi,\eta,\zeta)\frac{\psi_{k_{2}}(\xi-\eta)}{\phi(\xi,\xi-\eta)}\hat{g}_{k_{1}}(s,\eta)e^{-is|\zeta|^{\alpha}}\overline{B(u,\bar{u})}^\wedge_{k_{3}}(s,\zeta)\partial_{\xi_{j}}\hat{g}_{k_{4}}(s,\xi-\eta-\zeta)\\
&\quad+\frac{\psi_{k_{2}}(\xi-\eta)}{\phi(\xi,\xi-\eta)}\hat{g}_{k_{1}}(s,\eta)\partial_{\zeta_{i}}\hat{\bar{g}}_{k_{3}}(s,\zeta)\partial_{\xi_{j}}\hat{g}_{k_{4}}(s,\xi-\eta-\zeta)\\
&\quad+\sum_{(i_{1},j_{1})=(i,j),(j,i)}\hat{g}_{k_{1}}(s,\eta)\hat{\bar{g}}_{k_{3}}(s,\zeta)\partial_{\xi_{i_{1}}}\big(\frac{\psi_{k_{2}}(\xi-\eta)}{\phi(\xi,\xi-\eta)}\big)\partial_{\xi_{j_{1}}}\hat{g}_{k_{4}}(s,\xi-\eta-\zeta)\\
&\quad+is\frac{\psi_{k_{2}}(\xi-\eta)}{\phi(\xi,\xi-\eta)}\hat{g}_{k_{1}}(s,\eta)\hat{\bar{f}}_{k_{3}}(s,\zeta)\Gamma_{ij}^{(9)}\big(\hat{g}_{k_{4}}(s,\xi-\eta-\zeta)\big)\\
&\quad-s^{2}\big(\partial_{\xi_{i}}\Phifo\partial_{\xi_{j}}\Phifo\big)(\xi,\eta,\zeta)\frac{\psi_{k_{2}}(\xi-\eta)}{\phi(\xi,\xi-\eta)}\hat{g}_{k_{1}}(s,\eta)\hat{\bar{g}}_{k_{3}}(s,\zeta)\hat{g}_{k_{4}}(s,\xi-\eta-\zeta)\bigg]d\eta d\zeta ds\\
&=:\sum_{\ell=1}^{10}\psi_{k}(\xi)\MJ_{5,1}^{(\ell)}(t,\xi),
\end{align*}
where $\Phifo(\xi,\eta,\zeta):=|\xi|^{\alpha}-|\eta|^{\alpha}+|\zeta|^{\alpha}-|\xi-\eta-\zeta|^{\alpha}$, and 
\EQN{
\Gamma_{ij}^{(6,l)}:=\rho_{l}^{(1)}(\xi-\eta,\xi-\eta-\zeta)\frac{\partial_{\xi_{i}}\Phifo(\xi,\eta,\zeta)}{\partial_{\zeta_{l}}\Phifo(\xi,\eta,\zeta)}\partial_{\xi_{j}},\quad {1\leq l\leq 3,}
}
\EQN{
\Gamma_{ij}^{(7)}:=\partial_{\xi_{i}}\Phifo(\xi,\eta,\zeta)\partial_{\xi_{j}}+\partial_{\xi_{j}}\Phifo(\xi,\eta,\zeta)\partial_{\xi_{i}},
}
\EQN{
\Gamma_{ij}^{(8)}:=\partial_{\xi_{i}}\Phifo(\xi,\eta,\zeta)\partial_{\xi_{j}}+\partial_{\xi_{j}}\Phifo(\xi,\eta,\zeta)\partial_{\xi_{i}}+\partial_{\xi_{i},\xi_{j}}^{2}\Phifo(\xi,\eta,\zeta),
}
\EQN{
\Gamma_{ij}^{(9)}:=\partial_{\xi_{i}}\Phifo(\xi,\eta,\zeta)\partial_{\xi_{j}}+\partial_{\xi_{j}}\Phifo(\xi,\eta,\zeta)\partial_{\xi_{i}}+\partial_{\zeta_{i}}\Phifo(\xi,\eta,\zeta)\partial_{\xi_{j}}.
}
We note that $\partial_{\xi_{i}}\Phifo,\partial_{\eta_{i}}\Phifo\in 2^{(\alpha-1)k_{3}}S^{0}(\Omega^{(4)}_{k,k_{1},k_{2},k_{3},k_{4}})$, $\partial_{\zeta_{i}}\Phifo\in 2^{(\alpha-2)k_{3}+k}S^{0}(\Omega^{(4)}_{k,k_{1},k_{2},k_{3},k_{4}})$, $\partial_{\xi_{i},\xi_{j}}^{2}\Phifo\in 2^{(\alpha-2)k}S^{0}(\Omega^{(4)}_{k,k_{1},k_{2},k_{3},k_{4}})$ and $|\partial_{\zeta_{l}}\Phifo(\xi,\eta,\zeta)|\sim 2^{(\alpha-2)k_{3}+k}$ if $\rho_{l}^{(1)}(\xi-\eta,\xi-\eta-\zeta)\neq0$, where $1\leq i,j,l\leq3$, the symbol class under consideration is defined in Definition \ref{def3.7} and $\Omega^{(4)}_{k,k_{1},k_{2},k_{3},k_{4}}$ is defined by
\EQ{\label{omega4}
\{(\xi,\eta,\zeta)\in \R^{9}: 2^{-k}|\xi|,2^{-k_{1}}|\eta|,2^{-k_{2}}|\xi-\eta|,2^{-k_{3}}|\zeta|,2^{-k_{4}}|\xi-\eta-\zeta|\in(1/2,2)\}.
}
By \eqref{e2.4}, we can write $\psi_{k}(\xi)\MJ_{5,1}^{(10)}(t,\xi)$ as
\begin{align*}
&-\psi_{k}(\xi)\sum_{\substack{(k_{1},k_{2})\in\chi_{k}^{3}\\
(k_{3},k_{4})\in\chi_{k_{2}}^{1}\\k_{4}> k_{2}+12}}\int_{0}^{t}\int_{\R^6} e^{is\Phifo(\xi,\eta,\zeta)}\bigg[s^{2}A_{k_{2}}^{(5,1)}(\xi,\eta,\zeta)\hat{g}_{k_{1}}(s,\eta)\hat{\bar{f}}_{k_{3}}(s,\zeta)\hat{f}_{k_{4}}(s,\xi-\eta-\zeta)\\
&+s^{2}A_{k_{2}}^{(5,1)}(\xi,\eta,\zeta)\hat{g}_{k_{1}}(s,\eta)e^{-is|\zeta|^{\alpha}}(\overline{B(u,\bar{u})})^\wedge_{k_{3}}(s,\zeta)e^{is|\xi-\eta-\zeta|^{\alpha}}B(u,\bar{u})^\wedge_{k_{4}}(s,\xi-\eta-\zeta)\\
&+s^{2}A_{k_{2}}^{(5,1)}(\xi,\eta,\zeta)\hat{g}_{k_{1}}(s,\eta)e^{-is|\zeta|^{\alpha}}\overline{B(u,\bar{u})}^\wedge_{k_{3}}(s,\zeta)\hat{f}_{k_{4}}(s,\xi-\eta-\zeta)\\
&+s^{2}A_{k_{2}}^{(5,1)}(\xi,\eta,\zeta)\hat{g}_{k_{1}}(s,\eta)\hat{\bar{f}}_{k_{3}}(s,\zeta) e^{is|\xi-\eta-\zeta|^{\alpha}}B(u,\bar{u})^\wedge_{k_{4}}(s,\xi-\eta-\zeta)\bigg]d\eta d\zeta ds\\
&-\psi_{k}(\xi)\sum_{\substack{(k_{1},k_{2})\in\chi_{k}^{3}\\
(k_{3},k_{4})\in\chi_{k_{2}}^{1}\\|k_{4}- k_{2}|\leq12}}\int_{0}^{t}\int_{\R^6} e^{is\Phifo(\xi,\eta,\zeta)}s^{2}A_{k_{2}}^{(5,1)}(\xi,\eta,\zeta)\hat{g}_{k_{1}}(s,\eta)\hat{\bar{g}}_{k_{3}}(s,\zeta)\hat{g}_{k_{4}}(s,\xi-\eta-\zeta)d\eta d\zeta ds\\
&=:\sum_{\ell=1}^{5}\psi_{k}(\xi)\MJ_{5,1}^{(10,\ell)}(t,\xi),
\end{align*}
where
\EQ{\label{e4.48}
A_{k_{2}}^{(5,1)}(\xi,\eta,\zeta):=\frac{\partial_{\xi_{i}}\Phifo(\xi,\eta,\zeta)\partial_{\xi_{j}}\Phifo(\xi,\eta,\zeta)}{\phi(\xi,\xi-\eta)}\psi_{k_{2}}(\xi-\eta).
}
The proof of $\{\psi_{k}\MJ_{5,1}^{(\ell)}(t)\}_{\ell=1}^{3}$ follows a similar argument to  that of $\psi_{k}\hat{Q}_{5}(t)$, using the method from the proof of \eqref{e5.13A}. By Bernstein's inequality, $(\infty,2,2)$ estimates, \eqref{e5.27A}, \eqref{e4.8s} and \eqref{e4.1s}, we have
\begin{align*}
\sum_{\ell=4}^{8}\|\psi_{k}\MJ_{5,1}^{(\ell)}(t)\|_{2}\les \sum_{k_{3}\geq k-20}\int_{0}^{t}2^{(\frac{1}{2}-\alpha) k}(1+s)^{-1-\frac{\delta}{2}}2^{(2\lambda-1)k_{3}-4k_{3,+}}ds\les 2^{(\lambda-2)k-2k_{+}}.
\end{align*}
By Bernstein's inequality, $(\infty,\infty,2)$ and $(\infty,2,2)$ estimates, we have
\begin{align*}
&\|\psi_{k}\MJ_{5,1}^{(9)}(t)\|_{2}\les \int_{0}^{t}s2^{-2k}(1+s)^{-1-\frac{\delta}{2}}2^{\lambda k-2k_{+}}C_{k,1}(s)ds\les 2^{(\lambda-2)k-2k_{+}},
\end{align*}
where $C_{k,1}(s)$ is defined as in \eqref{e4.5}. 
By applying integration by parts in $\zeta$ twice, we can write $\psi_{k}(\xi)\MJ_{5,1}^{(10,1)}(t,\xi)$ as
\begin{align*}
&\psi_{k}(\xi)\sum_{l=1}^{3}\sum_{\substack{(k_{1},k_{2})\in\chi_{k}^{3}\\
(k_{3},k_{4})\in\chi_{k}^{1}\\k_{4}> k_{2}+12}}\int_{0}^{t}\int_{\R^6} e^{is\Phifo(\xi,\eta,\zeta)}\hat{g}_{k_{1}}(s,\eta)
\partial_{\zeta_{l}}\bigg[\frac{1}{\partial_{\zeta_{l}}\Phifo(\xi,\eta,\zeta)}\\
&\quad\times\partial_{\zeta_{l}}\bigg(\rho_{l}^{(1)}(\xi-\eta,\xi-\eta-\zeta)\frac{A_{k_{2}}^{(5,1)}(\xi,\eta,\zeta)}{\partial_{\zeta_{l}}\Phifo(\xi,\eta,\zeta)}\hat{\bar{f}}_{k_{3}}(s,\zeta)\hat{f}_{k_{4}}(s,\xi-\eta-\zeta)\bigg)\bigg]d\eta d\zeta ds,
\end{align*}
where $A_{k_{2}}^{(5,1)}$ is given by \eqref{e4.48}. 
By Lemmas \ref{lem3.2}, \ref{lem4.6} and \eqref{e4.6s}, we have
\begin{align*}
\|\psi_{k}\MJ_{5,1}^{(10,1)}(t)\|_{2}&\les 2^{(\lambda-2) k-2k_{+}}\sup_{|k_{2}-k|\leq 2}\int_{0}^{t}2^{-\alpha k}\sum_{k_{1}\leq k}\|u_{k_{1}}(s)\|_{\infty}C_{k,1}(s)ds\les 2^{(\lambda-2) k-2k_{+}},
\end{align*}
where $C_{k,1}(s)$ is defined as in \eqref{e4.5}. Let $\{\rho_{l}^{(3)}\}_{l=1}^{3}$ be given by Lemma \ref{lem3.12aa}. For $\psi_{k}\MJ_{5,1}^{(10,3)}(t)$ and $\psi_{k}\MJ_{5,1}^{(10,4)}(t)$, we can apply integration by parts in $\eta$ to write $\psi_{k}(\xi)\MJ_{5,1}^{(10,3)}(t,\xi)$ as
\begin{align*}
&\psi_{k}(\xi)\sum_{l=1}^{3}\sum_{\substack{(k_{1},k_{2})\in\chi_{k}^{3}\\
(k_{3},k_{4})\in\chi_{k_{2}}^{1}\\k_{4}> k_{2}+12}}-i\int_{0}^{t}\int_{\R^6} e^{is\Phifo(\xi,\eta,\zeta)}\bigg[s\partial_{\eta_{l}}\bigg(\frac{A_{k_{2}}^{(5,1)}(\xi,\eta,\zeta)}{\partial_{\eta_{l}}\Phifo(\xi,\eta,\zeta)}\hat{g}_{k_{1}}(s,\eta)\bigg)\\
&\hskip 5cm\times
e^{-is|\zeta|^{\alpha}}\overline{B(u,\bar{u})}^\wedge_{k_{3}}(s,\zeta)
(\rho_{l}^{(3)}\hat{f}_{k_{4}})(s,\xi-\eta-\zeta)\\
&-is\frac{A_{k_{2}}^{(5,1)}(\xi,\eta,\zeta)}{\partial_{\eta_{l}}\Phifo(\xi,\eta,\zeta)}\hat{g}_{k_{1}}(s,\eta)e^{-is|\zeta|^{\alpha}}\overline{B(u,\bar{u})}^\wedge_{k_{3}}(s,\zeta)\partial_{\xi_{l}}(\rho_{l}^{(3)}\hat{f}_{k_{4}})(s,\xi-\eta-\zeta)\bigg]d\eta d\zeta ds,
\end{align*}
and to write $\psi_{k}(\xi)\MJ_{5,1}^{(10,4)}(t,\xi)$ as
\begin{align*}
&\psi_{k}(\xi)\sum_{l=1}^{3}\sum_{\substack{(k_{1},k_{2})\in\chi_{k}^{3}\\
(k_{3},k_{4})\in\chi_{k_{2}}^{1}\\k_{4}> k_{2}+12}}i\int_{0}^{t}\int_{\R^6} e^{is\Phifo(\xi,\eta,\zeta)}\bigg[
s\frac{A_{k_{2}}^{(5,1)}(\xi,\eta,\zeta)}{\partial_{\eta_{l}}\Phifo(\xi,\eta,\zeta)}\hat{g}_{k_{1}}(s,\eta)\hat{\bar{f}}_{k_{3}}(s,\zeta)\\
&\qquad\qquad\qquad\qquad\quad\times\partial_{\xi_{l}}(\rho_{l}^{(3)}e^{is|\cdot|^{\alpha}}B(u,\bar{u})^\wedge_{k_{4}})(s,\xi-\eta-\zeta)\\
&-s\partial_{\eta_{l}}\bigg(\frac{A_{k_{2}}^{(5,1)}(\xi,\eta,\zeta)}{\partial_{\eta_{l}}\Phifo(\xi,\eta,\zeta)}\hat{g}_{k_{1}}(s,\eta)\bigg)\hat{\bar{f}}_{k_{3}}(s,\zeta) (\rho_{l}^{(3)}e^{is|\cdot|^{\alpha}}B(u,\bar{u})^\wedge_{k_{4}})(s,\xi-\eta-\zeta)\bigg]d\eta d\zeta ds,
\end{align*}
where $A_{k_{2}}^{(5,1)}$ is given by \eqref{e4.48} and $|\partial_{\eta_{l}}\Phifo|\sim 2^{(\alpha-1)k_{3}}$ if the integrands are not  zero.
By $(\infty,\infty,2)$ estimates, Proposition \ref{prop4.9}, Lemma \ref{lem4.2}, as well as \eqref{e4.8s} and \eqref{e4.9s}, we can get the desired estimates for $\{\psi_{k}\MJ_{5,1}^{(10,\ell)}(t)\}_{\ell=2}^{4}$. For $\psi_{k}\MJ_{5,1}^{(10,5)}(t)$, $k,k_{2},k_{3},k_{4}$ are roughly the same and we can apply integration by parts in $\zeta$.
These give us the estimate for $\psi_{k}\MJ_{5,1}^{(10)}(t)$ and then we have
\begin{align*}
\|\psi_{k}\partial^{2}_{\xi_{i},\xi_{j}}\hat{Q}_{5,1}(t)\|_{2}\leq\sum_{\ell=1}^{10}\|\psi_{k}\MJ_{5,1}^{(\ell)}(t)\|_{2}\les 2^{(\lambda-2) k-2k_{+}}.
\end{align*}

{\bf Case 2:} Estimates for $\psi_{k}\partial_{\xi_{i},\xi_{j}}^{2}\hat{Q}_{5,2}(t)$. We can write $\psi_{k}(\xi)\partial_{\xi_{i},\xi_{j}}^{2}\hat{Q}_{5,2}(t,\xi)$ as
\begin{align*}
&\psi_{k}(\xi)\sum_{\substack{(k_{1},k_{2})\in\chi_{k}^{3}\\
(k_{3},k_{4})\in\chi_{k_{2}}^{2}}}\int_{0}^{t}\int_{\R^6} e^{is\Phifi(\xi,\eta,\zeta)}
\bigg[\partial_{\xi_{i},\xi_{j}}^{2}\big(\frac{\psi_{k_{2}}(\xi-\eta)}{\phi(\xi,\xi-\eta)}\big)\hat{g}_{k_{1}}(s,\eta)\hat{\bar{g}}_{k_{3}}(s,\xi-\eta-\zeta)\hat{g}_{k_{4}}(s,\zeta)\\
&+\sum_{(i_{1},j_{1})=(i,j),(j,i)}\partial_{\xi_{i_{1}}}\big(\frac{\psi_{k_{2}}(\xi-\eta)}{\phi(\xi,\xi-\eta)}\big)\hat{g}_{k_{1}}(s,\eta)\partial_{\xi_{j_{1}}}\hat{\bar{g}}_{k_{3}}(s,\xi-\eta-\zeta)\hat{g}_{k_{4}}(s,\zeta)\\
&+\partial_{\eta_{i}}\big(\frac{\psi_{k_{2}}(\xi-\eta)}{\phi(\xi,\xi-\eta)}\hat{g}_{k_{1}}(s,\eta)\big)\partial_{\xi_{j}}\hat{\bar{g}}_{k_{3}}(s,\xi-\eta-\zeta)\hat{g}_{k_{4}}(s,\zeta)\\
&+is\partial_{\eta_{i}}\Phifi(\xi,\eta,\zeta)
\frac{\psi_{k_{2}}(\xi-\eta)}{\phi(\xi,\xi-\eta)}\hat{g}_{k_{1}}(s,\eta)\partial_{\xi_{j}}\hat{\bar{g}}_{k_{3}}(s,\xi-\eta-\zeta)\hat{g}_{k_{4}}(s,\zeta)\\
&+is\hat{g}_{k_{1}}(s,\eta)(\partial_{\xi_{i}}\Phifi(\xi,\eta,\zeta)
\partial_{\xi_{j}}+\partial_{\xi_{j}}\Phifi(\xi,\eta,\zeta)
\partial_{\xi_{i}})\big(\frac{\psi_{k_{2}}(\xi-\eta)}{\phi(\xi,\xi-\eta)}\hat{\bar{g}}_{k_{3}}(s,\xi-\eta-\zeta)\big)\hat{g}_{k_{4}}(s,\zeta)\\
&+is\hat{g}_{k_{1}}(s,\eta)\partial_{\xi_{i},\xi_{j}}^{2}\Phifi(\xi,\eta,\zeta)
\frac{\psi_{k_{2}}(\xi-\eta)}{\phi(\xi,\xi-\eta)}\hat{\bar{g}}_{k_{3}}(s,\xi-\eta-\zeta)\hat{g}_{k_{4}}(s,\zeta)\\
&-s^{2}\partial_{\xi_{i}}\Phifi(\xi,\eta,\zeta)\partial_{\xi_{j}}\Phifi(\xi,\eta,\zeta)\frac{\psi_{k_{2}}(\xi-\eta)}{\phi(\xi,\xi-\eta)}\hat{g}_{k_{1}}(s,\eta)\hat{\bar{g}}_{k_{3}}(s,\xi-\eta-\zeta)\hat{g}_{k_{4}}(s,\zeta)\bigg]d\eta d\zeta ds\\
&=:\sum_{\ell=1}^{7}\psi_{k}(\xi)\MJ_{5,2}^{(\ell)}(t,\xi),
\end{align*}
where $\Phifi(\xi,\eta,\zeta):=|\xi|^{\alpha}-|\eta|^{\alpha}+|\xi-\eta-\zeta|^{\alpha}-|\zeta|^{\alpha}$, $\partial_{\xi_{i}}\Phifi,\partial_{\eta_{i}}\Phifi\in2^{(\alpha-1)k}S^{0}(\Omega^{(5)}_{k,k_{1},k_{2},k_{3},k_{4}})$ for all $i$ and $\Omega^{(5)}_{k,k_{1},k_{2},k_{3},k_{4}}$ is defined by
\EQN{
\{(\xi,\eta,\zeta)\in \R^{9}: 2^{-k}|\xi|,2^{-k_{1}}|\eta|,2^{-k_{2}}|\xi-\eta|,2^{-k_{3}}|\xi-\eta-\zeta|,2^{-k_{4}}|\zeta|\in(1/2,2)\}.
}
The symbol class under consideration is defined in Definition \ref{def3.7}.
By integration by parts in $s$, we can write $\psi_{k}(\xi)\MJ_{5,2}^{(7)}(t,\xi)$ as
\begin{align*}
&\psi_{k}(\xi)\sum_{\substack{(k_{1},k_{2})\in\chi_{k}^{3}\\
(k_{3},k_{4})\in\chi_{k_{2}}^{2}}}\bigg[i\int_{\R^6} e^{it\Phifi(\xi,\eta,\zeta)}t^{2}\frac{A_{k_{2}}^{(5,2)}(\xi,\eta,\zeta)}{\Phifi(\xi,\eta,\zeta)}\hat{g}_{k_{1}}(t,\eta)\hat{\bar{g}}_{k_{3}}(t,\xi-\eta-\zeta)\hat{g}_{k_{4}}(t,\zeta)d\eta d\zeta\\
&-i\int_{0}^{t}\int_{\R^6} e^{is\Phifi(\xi,\eta,\zeta)}\partial_{s}\bigg(s^{2}\frac{A_{k_{2}}^{(5,2)}(\xi,\eta,\zeta)}{\Phifi(\xi,\eta,\zeta)}\hat{g}_{k_{1}}(s,\eta)\hat{\bar{g}}_{k_{3}}(s,\xi-\eta-\zeta)\hat{g}_{k_{4}}(s,\zeta)\bigg)d\eta d\zeta ds\bigg],
\end{align*}
where $|\Phifi|\sim 2^{\alpha k}$ on $\Omega^{(5)}_{k,k_{1},k_{2},k_{3},k_{4}}$ and
\EQ{\label{e4.49}
A_{k_{2}}^{(5,2)}(\xi,\eta,\zeta):=\frac{\partial_{\xi_{i}}\Phifi(\xi,\eta,\zeta)\partial_{\xi_{j}}\Phifi(\xi,\eta,\zeta)}{\phi(\xi,\xi-\eta)}\psi_{k_{2}}(\xi-\eta).
}
An application of Lemmas \ref{lem4.2} and \ref{lem4.5}, together with \eqref{e5.27A} and $(\infty,2,\infty)$ estimates, yields
\begin{align*}
\|\psi_{k}\partial^{2}_{\xi_{i},\xi_{j}}\hat{Q}_{5,2}(t)\|_{2}\leq\sum_{\ell=1}^{7}\|\psi_{k}\MJ_{5,2}^{(\ell)}(t)\|_{2}\les 2^{(\lambda-2)k-2k_{+}}.
\end{align*}

{\bf Case 3:} Estimates for $\psi_{k}\partial_{\xi_{i},\xi_{j}}^{2}\hat{Q}_{5,3}(t)$. We can express $\psi_{k}(\xi)\partial_{\xi_{i},\xi_{j}}^{2}\hat{Q}_{5,3}(t,\xi)$ as
\begin{align*}
&\psi_{k}(\xi)\sum_{\substack{(k_{1},k_{2})\in\chi_{k}^{3}\\
(k_{3},k_{4})\in\chi_{k_{2}}^{3}}}\int_{0}^{t}\int_{\R^6} e^{is\Phisi(\xi,\eta,\zeta)}\bigg[\partial_{\xi_{i},\xi_{j}}^{2}\big(\frac{\psi_{k_{2}}(\xi-\eta)}{\phi(\xi,\xi-\eta)}\big)\hat{g}_{k_{1}}(s,\eta)\hat{\bar{g}}_{k_{3}}(s,\zeta)\hat{g}_{k_{4}}(s,\xi-\eta-\zeta)\\
&+\sum_{(i_{1},j_{1})=(i,j),(j,i)}\partial_{\xi_{i_{1}}}(\frac{\psi_{k_{2}}(\xi-\eta)}{\phi(\xi,\xi-\eta)})\hat{g}_{k_{1}}(s,\eta)\hat{\bar{g}}_{k_{3}}(s,\zeta)\partial_{\xi_{j_{1}}}\hat{g}_{k_{4}}(s,\xi-\eta-\zeta)\\
&+\partial_{\eta_{i}}\big(\frac{\psi_{k_{2}}(\xi-\eta)}{\phi(\xi,\xi-\eta)}\hat{g}_{k_{1}}(s,\eta)\big)\hat{\bar{g}}_{k_{3}}(s,\zeta)\partial_{\xi_{j}}\hat{g}_{k_{4}}(s,\xi-\eta-\zeta)\\
&+is\partial_{\eta_{i}}\Phisi(\xi,\eta,\zeta)\frac{\psi_{k_{2}}(\xi-\eta)}{\phi(\xi,\xi-\eta)}\hat{g}_{k_{1}}(s,\eta)\hat{\bar{g}}_{k_{3}}(s,\zeta)\partial_{\xi_{j}}\hat{g}_{k_{4}}(s,\xi-\eta-\zeta)\\
&+is\hat{g}_{k_{1}}(s,\eta)\hat{\bar{g}}_{k_{3}}(s,\zeta)\Gamma_{ij}^{(10)}\big(\frac{\psi_{k_{2}}(\xi-\eta)}{\phi(\xi,\xi-\eta)}\hat{g}_{k_{4}}(s,\xi-\eta-\zeta)\big)\\
&+is^{2}A_{k_{2}}^{(5,3)}(\xi,\eta,\zeta)\hat{g}_{k_{1}}(s,\eta)\hat{\bar{g}}_{k_{3}}(s,\zeta)B(u,\bar{u})^\wedge_{k_{4}}(s,\xi-\eta-\zeta)\\
&-s^{2}A_{k_{2}}^{(5,3)}(\xi,\eta,\zeta)\hat{g}_{k_{1}}(s,\eta)\hat{\bar{g}}_{k_{3}}(s,\zeta)\hat{f}_{k_{4}}(s,\xi-\eta-\zeta)\bigg]d\eta d\zeta ds\\
&=:\sum_{\ell=1}^{7}\psi_{k}(\xi)\MJ_{5,3}^{(\ell)}(t,\xi),
\end{align*}
where $\Phisi(\xi,\eta,\zeta):=|\xi|^{\alpha}-|\eta|^{\alpha}+|\zeta|^{\alpha}-|\xi-\eta-\zeta|^{\alpha}$ and
\EQ{\label{e4.50}
A_{k_{2}}^{(5,3)}(\xi,\eta,\zeta):=\frac{\partial_{\xi_{i}}\Phisi(\xi,\eta,\zeta)\partial_{\xi_{j}}\Phisi(\xi,\eta,\zeta)}{\phi(\xi,\xi-\eta)}\psi_{k_{2}}(\xi-\eta),
}
\EQN{
\Gamma_{ij}^{(10)}:=\partial_{\xi_{i}}\Phisi(\xi,\eta,\zeta)\partial_{\xi_{j}}+\partial_{\xi_{j}}\Phisi(\xi,\eta,\zeta)\partial_{\xi_{i}}+\partial_{\xi_{i},\xi_{j}}^{2}\Phisi(\xi,\eta,\zeta).
}
We have that $\partial_{\xi_{i}}\Phisi\in 2^{(\alpha-2)k}\max\{2^{k_{1}},2^{k_{3}}\}S^{0}(\Omega^{(4)}_{k,k_{1},k_{2},k_{3},k_{4}})$, $\partial_{\eta_{i}}\Phisi\in 2^{(\alpha-1)k}S^{0}(\Omega^{(4)}_{k,k_{1},k_{2},k_{3},k_{4}})$,  and $\partial_{\zeta_{i}}\Phisi\in 2^{(\alpha-1)k}S^{0}(\Omega^{(4)}_{k,k_{1},k_{2},k_{3},k_{4}})$ for all $i$. Here $\Omega^{(4)}_{k,k_{1},k_{2},k_{3},k_{4}}$ is defined as in \eqref{omega4} and { $\{\rho_{l}^{(3)}\}_{l=1}^{3}$ is given by Lemma \ref{lem3.12aa}.}  {The symbol class under consideration is defined in Definition \ref{def3.7}.}

{By integration by parts, we have
\EQN{
\psi_{k}(\xi)\MJ_{5,3}^{(7)}(t,\xi)=\sum_{\ell=1}^{6}\psi_{k}(\xi)\MJ_{5,3}^{(7,\ell)}(t,\xi),
}
where  $\{\MJ_{5,3}^{(7,\ell)}\}_{\ell=1}^{6}$ has the same form as  $\MJ_{5,3}^{(7)}$ except for symbols and their symbols are given by 
\EQN{
\sum_{l=1}^{3}is\frac{A_{k_{2}}^{(5,3)}(\xi,\eta,\zeta)}{\partial_{\eta_{l}}\Phisi(\xi,\eta,\zeta)}\hat{g}_{k_{1}}(s,\eta)\hat{\bar{g}}_{k_{3}}(s,\zeta)\partial_{\xi_{l}}\big(\rho_{l}^{(3)}\hat{f}_{k_{4}}\big)(s,\xi-\eta-\zeta),
}
\EQN{
\sum_{l=1}^{3}i\frac{1}{\partial_{\zeta_{l}}\Phisi}\partial_{\eta_{l}}\bigg(\frac{A_{k_{2}}^{(5,3)}(\xi,\eta,\zeta)}{\partial_{\eta_{l}}\Phisi(\xi,\eta,\zeta)}\hat{g}_{k_{1}}(s,\eta)\bigg)\hat{\bar{g}}_{k_{3}}(s,\zeta)\partial_{\xi_{l}}\big(\rho_{l}^{(3)}\hat{f}_{k_{4}}\big)(s,\xi-\eta-\zeta),
}
\EQN{
\sum_{l=1}^{3}\chi_{[k_{1},\infty)}(k_{3})\frac{A_{k_{2}}^{(5,3)}(\xi,\eta,\zeta)}{\big(\partial_{\eta_{l}}\Phisi\partial_{\zeta_{l}}\Phisi\big)(\xi,\eta,\zeta)}\partial_{\eta_{l}}\hat{g}_{k_{1}}(s,\eta)\partial_{\zeta_{l}}\hat{\bar{g}}_{k_{3}}(s,\zeta)\big(\rho_{l}^{(3)}\hat{f}_{k_{4}}\big)(s,\xi-\eta-\zeta),
}
\EQN{
\sum_{l=1}^{3}\chi_{(k_{3},\infty)}(k_{1})\frac{A_{k_{2}}^{(5,3)}(\xi,\eta,\zeta)}{\big(\partial_{\eta_{l}}\Phisi\partial_{\zeta_{l}}\Phisi\big)(\xi,\eta,\zeta)}\partial_{\eta_{l}}\hat{g}_{k_{1}}(s,\eta)\partial_{\zeta_{l}}\hat{\bar{g}}_{k_{3}}(s,\zeta)\big(\rho_{l}^{(3)}\hat{f}_{k_{4}}\big)(s,\xi-\eta-\zeta),
}
\EQN{
\sum_{l=1}^{3}\hat{g}_{k_{1}}(s,\eta)\partial_{\zeta_{l}}\bigg(\frac{1}{\partial_{\zeta_{l}}\Phisi(\xi,\eta,\zeta)}\partial_{\eta_{l}}\big(\frac{A_{k_{2}}^{(5,3)}(\xi,\eta,\zeta)}{\partial_{\eta_{l}}\Phisi(\xi,\eta,\zeta)}\big)\hat{\bar{g}}_{k_{3}}(s,\zeta)\bigg)\big(\rho_{l}^{(3)}\hat{f}_{k_{4}}\big)(s,\xi-\eta-\zeta),
}
\begin{align*}
\sum_{l=1}^{3}\partial_{\zeta_{l}}\bigg(\frac{A_{k_{2}}^{(5,3)}(\xi,\eta,\zeta)}{\big(\partial_{\eta_{l}}\Phisi\partial_{\zeta_{l}}\Phisi\big)(\xi,\eta,\zeta)}\bigg)\partial_{\eta_{l}}\hat{g}_{k_{1}}(s,\eta)\hat{\bar{g}}_{k_{3}}(s,\zeta)\big(\rho_{l}^{(3)}\hat{f}_{k_{4}}\big)(s,\xi-\eta-\zeta).
\end{align*}
We remark that}  $A_{k_{2}}^{(5,3)}$ is given by \eqref{e4.50} and $|\partial_{\eta_{l}}\Phisi|$, $|\partial_{\zeta_{l}}\Phisi|\sim 2^{(\alpha-1)k}$ if the symbols are not zero. 
By combining Lemmas \ref{lem4.2} and \ref{lem4.3} with $(\infty,\infty,2)$ estimates, we can derive bounds for $\{\psi_{k}\MJ_{5,3}^{(\ell)}(t)\}_{\ell=1}^{6}$ and $\{\psi_{k}\MJ_{5,3}^{(7,\ell)}(t)\}_{\ell=1}^{2}$.
By Lemma \ref{lem4.2} and \eqref{e4.1s},  for $k\leq0$, we have
\begin{align*}
\|\psi_{k}\MJ_{5,3}^{(7,3)}(t)\|_{2}&\les 2^{-\alpha k-2k}\sum_{k_{3}\leq k}\sum_{k_{1}\leq k_{3}}2^{2k_{3}}2^{\frac{3}{2}k_{1}}\int_{0}^{t}\|\nabla\hat{g}_{k_{1}}(s)\|_{2}\|\nabla\hat{g}_{k_{3}}(s)\|_{2}\|w(s)\|_{\infty}ds\\
&\les 2^{-3k}\sum_{k_{3}\leq k}\sum_{k_{1}\leq k_{3}}2^{(\lambda+\frac{1}{2})k_{1}-2k_{1,+}}2^{(\lambda-\alpha+2)k_{3}-2k_{3,+}}\int_{0}^{t}(1+s)^{-1-\delta}ds\\
&\les2^{-3k}\sum_{k_{3}\leq k}2^{(\lambda+1)k_{3}-2k_{3,+}}
\les2^{(\lambda-2)k},
\end{align*}
while for $k\geq0$, by  Lemmas \ref{lem3.2} 
 and \ref{lem4.2}, we can bound $\|\psi_{k}\MJ_{5,3}^{(7,3)}(t)\|_{2}$ by 
\begin{align*}
&2^{-\alpha k-2k}\sup_{|k_{4}-k|\leq4}\sum_{\substack{k_{1}\leq k_{3}\\
k_{3}\leq k}}
2^{2k_{3}}2^{\frac{3}{2}k_{1}}\int_{0}^{t}\|\nabla\hat{g}_{k_{1}}(s)\|_{2}\|\nabla\hat{g}_{k_{3}}(s)\|_{2}\min\{2^{\frac{3}{2}k_{3}}\|w_{k_{4}}(s)\|_{2},\|w_{k_{4}}(s)\|_{\infty}\}ds\\
&\les2^{-3k}\sum_{k_{3}\leq k}2^{(\lambda+1)k_{3}-2k_{3,+}}\int_{0}^{\infty}\min\{2^{\frac{3}{2}k_{3}}2^{-2k_{+}},s^{-\frac{3}{2}}2^{(\lambda-\frac{3}{2}\alpha+\frac{3}{2})k-2k_{+}}\}ds\\
&\les2^{-3k}\sum_{k_{3}\leq k}2^{(\lambda+1)k_{3}-2k_{3,+}}2^{{\frac{1}{4}}k_{3}}2^{(\frac{2}{3}\lambda-{\frac{1}{2}}\alpha+1)k-2k_{+}}2^{\frac{\alpha}{4}k}\int_{0}^{\infty}\min\{s^{-\frac{3}{4}},s^{-\frac{5}{4}}\}ds\\
&\les2^{\frac{\alpha}{4}k}2^{(\lambda-{\frac{1}{2}}\alpha-2)k-2k_{+}}\leq 2^{(\lambda-2)k-2k_{+}}.
\end{align*}
We can estimate $\psi_{k}\MJ_{5,3}^{(7,4)}(t)$ in the same manner as $\psi_{k}\MJ_{5,3}^{(7,3)}(t)$. {The same approach used for $\{\psi_{k}\MJ_{5,3}^{(7,\ell)}(t)\}_{\ell=3}^{4}$ applies to estimating  $\{\psi_{k}\MJ_{5,3}^{(7,\ell)}(t)\}_{\ell=5}^{6}$, by separately considering the cases $k_{3}\geq k_{1}$ and $k_{1}>k_{3}$.}
These imply that
\begin{align*}
\|\psi_{k}\partial^{2}_{\xi_{i},\xi_{j}}\hat{Q}_{5,3}(t)\|_{2}\leq\sum_{\ell=1}^{7}\|\psi_{k}\MJ_{5,3}^{(\ell)}(t)\|_{2}\les 2^{(\lambda-2)k}2^{-2k_{+}}.
\end{align*}
Based on the results for $\{\psi_{k}\partial_{\xi_{i},\xi_{j}}^{2}\hat{Q}_{5,\ell}(t)\}_{\ell=1}^{3}$, we conclude that
\begin{align*}
\|\psi_{k}\partial^{2}_{\xi_{i},\xi_{j}}\hat{Q}_{5}(t)\|_{2}\leq\sum_{\ell=1}^{3}\|\psi_{k}\partial^{2}_{\xi_{i},\xi_{j}}\hat{Q}_{5,\ell}(t)\|_{2}\les 2^{(\lambda-2)k}2^{-2k_{+}}.
\end{align*}
\end{proof}

\begin{proof}[Proof of Proposition \ref{prop4.10s}]
Recall $f(t)=e^{itD^{\alpha}}w(t)$. 
By Lemmas \ref{lem4.15}, \ref{lem4.16}, \ref{lem4.17}, \ref{lem4.18} and \ref{lem4.19}, we have
\begin{align*}
\sup_{t>0}\sum_{\ell=1}^{5}\|Q_{\ell}(t)\|_{\PW}\les \sum_{l=2}^{4}\|(w,u)\|_{W\times U}^{l}.
\end{align*}
By \eqref{e2.5} and Duhamel's formula, for all $t\geq0$, we have
\begin{align*}
\hat{f}(t,\xi)&=\hat{w}_{0}(\xi)
+\sum_{\substack{k_{2}-k_{1}\leq 10\\
k_{1},k_{2}\in \Z}}\bigg(\int_{0}^{t}\int_{\R^3}e^{is|\xi|^{\alpha}}\hat{w}_{k_{1}}(s,\xi-\eta)\hat{\bar{w}}_{k_{2}}(s,\eta)d\eta ds\\
&\quad+\int_{0}^{t}\int_{\R^3}e^{is|\xi|^{\alpha}}\hat{w}_{k_{1}}(s,\xi-\eta)(\overline{B(u,\bar u)})^\wedge_{k_{2}}(s,\eta)d\eta ds\\
&\quad+\int_{0}^{t}\int_{\R^3}e^{is|\xi|^{\alpha}}B(u,\bar u)^\wedge_{k_{1}}(s,\xi-\eta)\hat{\bar{u}}_{k_{2}}(s,\eta)d\eta ds+\int_{0}^{t}e^{is|\xi|^{\alpha}}B(|u|^{2},\bar u)^\wedge(s,\xi)ds\\
&\quad+\int_{0}^{t}e^{is|\xi|^{\alpha}}B(u,|u|^{2})^\wedge(s,\xi)ds\bigg)\\
&=\hat{w}_{0}(\xi)+\sum_{\ell=1}^{5}\hat{Q}_{\ell}(t,\xi),
\end{align*}
which means
\begin{align*}
\|w\|_{W_{3}}=\sup_{t\geq0}\|f(t)\|_{\PW}\les \|w_{0}\|_{\PW}+\sum_{l=2}^{4}\|(w,u)\|_{W\times U}^{l}.
\end{align*}
\end{proof}

\begin{proof}[Proof of Proposition \ref{prop4.1}]
Proposition \ref{prop4.1} follows by combining Propositions \ref{prop4.7}, \ref{prop4.9}, and \ref{prop4.10s}.
\end{proof}

\section{Proof of the main theorem}
\setcounter{equation}{0}

Define $(w^{(0)}(t),u^{(0)}(t)):=(0,0)$ for all $t\in\R$. For $n\geq0$, consider
\EQ{\label{e6.1}
\CAS{
(\partial_{t}+iD^\alpha)u^{(n+1)}=u^{(n)}\overline{u^{(n)}}, \\
u^{(n+1)}= w^{(n+1)}-iB(u^{(n)},\overline{ u^{(n)}}),\\
u^{(n+1)}(0)=u_{0}\in \Hs\cap \PW.
}
}
Similarly to the proof of Lemma \ref{lem2.1}, for all $n\geq1$, we have
\EQ{\label{e6.2}
(\partial_{t}+iD^\alpha)w^{(n+1)}&= (w^{(n)}\overline{w^{(n)}})_{HH+HL}+i[w^{(n)} \overline{B(u^{(n-1)},\overline{u^{(n-1)}})}]_{HH+HL}\\
&-i[B(u^{(n-1)},\overline{u^{(n-1)}})\overline{u^{(n)}}]_{HH+HL}+iB(|u^{(n-1)}|^2,\overline{u^{(n)}})+iB(u^{(n)},|u^{(n-1)}|^2).
}
We denote the nonlinear term of \eqref{e6.2} by $Q(w^{(n)},u^{(n)},u^{(n-1)})$. 

\begin{lemma}\label{lem6.1}
Let $(W,U)$ be defined as in \eqref{e4.4A}. Under the assumption of Theorem \ref{thm:main}, $\{(w^{(n)},u^{(n)})\}_{n\geq1}\subseteq C(\R;\Hs)^2$ and satisfy
\EQ{\label{e6.3}
\sup_{n\geq1}\|(w^{(n)},u^{(n)})\|_{W\times U}\les \varepsilon_{0}.
}
Moreover, for all $n\geq4$, we have
\EQ{\label{e6.4}
&\|(w^{(n+1)}-w^{(n)},u^{(n+1)}-u^{(n)})\|_{W\times U}\\
&\leq \frac{1}{2}\sup_{n-2\leq j\leq n}\|(w^{(j)}-w^{(j-1)},u^{(j)}-u^{(j-1)})\|_{W\times U}.
}
\end{lemma}

\begin{proof}
By the theory of linear equations, we have $(w^{(n)},u^{(n)})\in C(\R;\Hs)^2$ for all $n\geq1$.
Since $(e^{itD^{\alpha}}w^{(1)}(t),e^{itD^{\alpha}}u^{(1)}(t))=(u_{0},u_{0})$, an application of Proposition \ref{prop4.9} yields 
\EQN{
\|(w^{(1)},u^{(1)})\|_{W\times U}\les\|u_{0}\|_{\Hs}+\|u_{0}\|_{\PW}\leq\varepsilon_{0}.
}
Let $n\geq2$. By \eqref{e6.1} and Lemma \ref{lem4.7a}, we have $w^{(n)}(0)=w_{0}$ and 
\EQN{
\|w^{(n)}(0)\|_{\Hs}+\|w^{(n)}(0)\|_{F}\les \varepsilon_{0}+\sup_{1\leq j\leq n-1}\|(w^{(j)},u^{(j)})\|_{W\times U}^{2}.
}
By an approach analogous to the proof of Propositions \ref{prop5.1} and \ref{prop4.1}, we have
\begin{align*}
\|(w^{(n)},u^{(n)})\|_{W\times U}\les \varepsilon_{0}+\sup_{1\leq j\leq n-1}\|(w^{(j)},u^{(j)})\|_{W\times U}^{2}
\end{align*}
if $\sup_{1\leq j\leq n-1}\|(w^{(j)},u^{(j)})\|_{W\times U}\leq 1$. Then by induction, \eqref{e6.3} follows if $\varepsilon_{0}$ is small enough.

For $n\geq4$, by \eqref{e6.1} and \eqref{e6.2}, we have
\EQN{
\CAS{(\partial_{t}+iD^\alpha)(w^{(n+1)}-w^{(n)})=Q(w^{(n)},u^{(n)},u^{(n-1)})-Q(w^{(n-1)},u^{(n-1)},u^{(n-2)}), \\
 u^{(n+1)}-u^{(n)}= w^{(n+1)}-w^{(n)}-iB(u^{(n)},\overline{ u^{(n)}})+iB(u^{(n-1)},\overline{ u^{(n-1)}}),
}
}
and $(w^{(n+1)}-w^{(n)})(0)=(u^{(n+1)}-u^{(n)})(0)=0$. By \eqref{e6.3} and an argument similar to that in Propositions \ref{prop5.1} and \ref{prop4.1}, we have
\EQN{
&\|(w^{(n+1)}-w^{(n)},u^{(n+1)}-u^{(n)})\|_{W\times U}\\
&\leq C
\varepsilon_{0}\sup_{n-2\leq j\leq n}\|(w^{(j)}-w^{(j-1)},u^{(j)}-u^{(j-1)})\|_{W\times U}.
}
Then \eqref{e6.4} follows if $\varepsilon_{0}$ is small enough.
\end{proof}

\begin{lemma}\label{lem6.2}
Suppose that there exist two global solutions $u,\tilde{u}$ to \eqref{e1.1} in $C(\R;\Hs)$. Then $u=\tilde{u}$.
\end{lemma}

\begin{proof}
Let $N\geq1$. Since $\tilde{u},u\in C(\R;\Hs)$, we have
\EQN{
\sup_{t\in[0,N]}\big(\|\tilde{u}(t)\|_{\Hs}+\|u(t)\|_{\Hs}\big)\leq C_{\tilde{u},u,N}<\infty.
}
Because $u,\tilde{u}$ are solutions to \eqref{e1.1}, we have
\EQN{
\CAS{(\partial_{t}+iD^{\alpha})(\tilde{u}-u) = \tilde{u}\bar{\tilde{u}}-u\bar{u}, \\
(\tilde{u}-u)(0) =0.
}
}
By Duhamel's formula, for all $t\geq0$, we have
\EQN{
e^{itD^{\alpha}}(\tilde{u}-u)(t)
=\int_{0}^{t}e^{isD^{\alpha}}(\tilde{u}\bar{\tilde{u}}-u\bar{u})(s)ds.
}
Then for all $t\in[0,N]$, we obtain
\begin{align*}
\|\tilde{u}(t)-u(t)\|_{\Hs}\les C_{\tilde{u},u,N}\int_{0}^{t}\|\tilde{u}(s)-u(s)\|_{\Hs}ds,
\end{align*}
where we use $\Hs\subseteq L^{\infty}$. Then by Gronwall's inequality, we have $\|\tilde{u}(t)-u(t)\|_{\Hs}=0$ and $\tilde{u}(t)=u(t)$ for all $t\in[0,N]$. Since $N$ is arbitrary, we have $\tilde{u}(t)=u(t)$ for all $t\geq0$. By the same way we can show that $\tilde{u}(t)=u(t)$ for all $t\leq0$.
\end{proof}

\begin{proof}[Proof of Theorem \ref{thm:main}]
Let $(W,U)$ be defined as in \eqref{e4.4A}.
Without loss of generality, we only consider the existence of solutions on $[0,\infty)$.
Let $\{(w^{(n)},u^{(n)})\}_{n\geq1}$ be given by \eqref{e6.1}. By \eqref{e6.3},  we have
\EQN{
\sup_{n\geq3}\|(w^{(n)},u^{(n)})-(w^{(n-1)},u^{(n-1)})\|_{W\times U}\leq C\varepsilon_{0}.
}
Suppose for some $m\geq1$, we have
\EQ{\label{e6.5}
\sup_{n\geq3m}\|(w^{(n)},u^{(n)})-(w^{(n-1)},u^{(n-1)})\|_{W\times U}\leq C\varepsilon_{0}(\frac{1}{2})^{m-1}.
}
Then  by \eqref{e6.4}, for all $n\geq 3m+3$, we have
\EQN{
\|(w^{(n)},u^{(n)})-(w^{(n-1)},u^{(n-1)})\|_{W\times U}\leq C\varepsilon_{0}(\frac{1}{2})^{m-1}\frac{1}{2}= C\varepsilon_{0}(\frac{1}{2})^{m},
}
which closes the induction. This tells us that \eqref{e6.5} holds for all $m\geq1$ and 
 $\{(w^{(n)},u^{(n)})\}_{n\geq1}$ is a Cauchy sequence in $W\times U$. Then by Lemma \ref{lem6.1},  
 \EQN{
(w,u):=\lim_{n\rightarrow\infty}(w^{(n)},u^{(n)})\in C([0,\infty);\Hs)^{2}, 
 }
 where the limit is in $W\times U$ sense.
Since $\{(w^{(n)},u^{(n)})\}_{n\geq1}$ satisfy \eqref{e6.1}, \eqref{e6.2} and \eqref{e6.3}, $(w,u)$ satisfies \eqref{e2.10}, \eqref{e1.1} and \eqref{e2.5} by a limit argument.  

Next, we will prove the uniqueness. Suppose $(\tilde{w},\tilde{u})$ satisfies \eqref{e2.5}, \eqref{e2.10} and $\tilde{u}(0)=u_{0}$. Then we have
\EQN{
\CAS{(\partial_{t}+iD^\alpha)(\tilde{w}-w)=Q(\tilde{w},\tilde{u},\tilde{u})-Q(w,u,u), \\
\tilde{u}-u= \tilde{w}-w-iB(\tilde{u},\bar{ \tilde{u}})+iB(u,\bar{u}),
}
}
and $(\tilde{w}-w)(0)=(\tilde{u}-u)(0)=0$, where $Q$ is defined by the nonlinear term of \eqref{e6.2}. By \eqref{e2.10} and an argument analogous to that in Propositions \ref{prop4.1} and \ref{prop5.1}, we have
\EQN{
\|(\tilde{w}-w,\tilde{u}-u)\|_{W\times U}\leq \frac{1}{2}\|(\tilde{w}-w,\tilde{u}-u)\|_{W\times U},
}
which means $\|(\tilde{w}-w,\tilde{u}-u)\|_{W\times U}=0$ and thus $(\tilde{w},\tilde{u})=(w,u)$. This proves the uniqueness of the Cauchy problem associated with \eqref{e2.5}. The uniqueness of \eqref{e1.1} follows from Lemma \ref{lem6.2}.

Now we consider the scattering property. By \eqref{e1.1} and Duhamel's formula, for all $t\geq0$, we have
\EQN{
e^{itD^{\alpha}}u(t)
=u_{0}+\int_{0}^{t}e^{isD^{\alpha}}(u\bar{u})(s)ds.
}
Since
\EQN{
\|e^{isD^{\alpha}}(u\bar{u})(s)\|_{\Hs}\les \|u(s)\|_{\Hs}\|u(s)\|_{\infty}\les (1+s)^{-1-\delta}\varepsilon_{0}^{2},
}
we can take the limit in $\Hs$ sense to write
\EQN{
\lim_{t\rightarrow+\infty}e^{itD^{\alpha}}u(t)
=u_{0}+\int_{0}^{\infty}e^{isD^{\alpha}}(u\bar{u})(s)ds\in \Hs.
}
By the proof of Lemma \ref{lem4.3} and \eqref{e2.10}, there exists a $\delta'>0$ such that for all $t\geq1$,
\EQN{
\|e^{itD^{\alpha}}B(u,\bar{u})(t)\|_{\Hs}\les t^{-\delta'}\varepsilon^{2}_{0},
}
which, together with \eqref{e2.4}, implies
\EQN{
\lim_{t\rightarrow+\infty}e^{itD^{\alpha}}w(t)=\lim_{t\rightarrow+\infty}e^{itD^{\alpha}}u(t)\in \Hs.
}
  Here the limits are taken in the sense of $\Hs$.
\end{proof}

\section{Final data problem}\label{sec:final data}
\setcounter{equation}{0}

Consider the final data problem
\EQ{\label{e7.1}
\CAS{(\partial_{t}+iD^{\alpha})u = \rho u\bar{u}, \\
\lim_{t\rightarrow +\infty}e^{itD^{\alpha}}u(t) =f_{\infty}\, \, in\, \, {\Hs},
}
}
where $u(t,x):\R\times \R^3 \to \C$ is the unknown function, $f_{\infty}$ is a given data and $\alpha\in (1,2)$. Here $\rho\in \C$ and plays no role in this paper.  We may assume $\rho=1$.

{
\begin{theorem}\label{thm7.1}
Suppose that $\alpha\in (1,2)$ and $\lambda\in(\frac{\alpha-1}{2},\frac{1}{2})$. Assume that the final data satisfies the following assumption
$$
\|f_{\infty}\|_{\Hs}+\|f_{\infty}\|_{\PW}\leq\varepsilon_{0},
$$
{where $\varepsilon_{0}>0$ is a sufficiently small constant that depends only on $\alpha,\lambda$.}
Then there exists a unique global solution $u$ to \eqref{e7.1} in $C(\R;\Hs)$ satisfying $\norm{u(t)}_{L^\infty}\les (1+t)^{-1-}$.
\end{theorem}
}

Theorem \ref{thm7.1} is a consequence of Theorem \ref{thm7.2}.
\begin{theorem}\label{thm7.2}
Under the assumption of Theorem \ref{thm7.1},
{there exists a unique global solution $(w,u)$ to \eqref{e2.5} in $C(\R;\Hs)^2$ satisfying
\EQ{\label{e7.2}
\lim_{t\rightarrow +\infty}e^{itD^{\alpha}}w(t) =\lim_{t\rightarrow +\infty}e^{itD^{\alpha}}u(t) =f_{\infty}\, \, in\, \, \Hs
}
and }
\EQ{\label{e7.3}
\sup_{t>0}\left(\|w(t)\|_{\Hs}+\|e^{itD^{\alpha}}w(t)\|_{\PW}+(1+t)^{1+\delta}\norm{w(t)}_{L^\infty}\right)&\\
+\sup_{t>0}\left(\|u(t)\|_{\Hs}+\|e^{itD^{\alpha}}u(t)\|_{\PU}+(1+t)^{1+\delta}\norm{u(t)}_{L^\infty}\right)&\leq C_{\alpha,\lambda} \, \varepsilon_{0},
}
where $0<\delta<\min\{\frac{\lambda+\frac{3}{2}}{\alpha},\frac{3}{2}\}-1$. In particular, $u$ is a unique global solution to \eqref{e7.1} in $C(\R;\Hs)$ satisfying $\norm{u(t)}_{L^\infty}\les (1+t)^{-1-}$.
\end{theorem}

\begin{remark}
Analogous observations to those in Remarks \ref{re2.4} and \ref{re2.5} apply here.
\end{remark}

Define $(w^{(0)}(t),u^{(0)}(t)):=(0,0)$ for all $t\in\R$. For $n\geq0$, consider
\EQ{\label{e7.4}
\CAS{
(\partial_{t}+iD^\alpha)u^{(n+1)}=u^{(n)}\overline{u^{(n)}}, \\
u^{(n+1)}= w^{(n+1)}-iB(u^{(n)},\overline{ u^{(n)}}),\\
\lim_{t\rightarrow+\infty}e^{itD^{\alpha}}u^{(n+1)}(t)=f_{\infty}\, \, in\, \, \Hs.
}
}
Similarly to the proof of Lemma \ref{lem2.1}, for all $n\geq1$, 
\eqref{e6.2} holds.

\begin{lemma}\label{lem7.4}
Let $(W,U)$ be defined as in \eqref{e4.4A}. Under the assumption of Theorem \ref{thm7.1}, $\{(w^{(n)},u^{(n)})\}_{n\geq1}\subseteq C(\R;\Hs)^2$ and 
\EQN{
\sup_{n\geq1}\|(w^{(n)},u^{(n)})\|_{W\times U}\les \varepsilon_{0}.
}
Moreover, for all $n\geq4$, we have
\EQ{\label{e7.5}
&\|(w^{(n+1)}-w^{(n)},u^{(n+1)}-u^{(n)})\|_{W\times U}\\
&\leq \frac{1}{2}\sup_{n-2\leq j\leq n}\|(w^{(j)}-w^{(j-1)},u^{(j)}-u^{(j-1)})\|_{W\times U}.
}
\end{lemma}

\begin{proof}
By the theory of linear equations, we have $(w^{(n)},u^{(n)})\in C(\R;\Hs)^2$ for all $n\geq1$.
Since $(e^{itD^{\alpha}}w^{(1)}(t),e^{itD^{\alpha}}u^{(1)}(t))=(f_{\infty},f_{\infty})$,
an application of Proposition \ref{prop4.9} yields 
\EQN{
\|(w^{(1)},u^{(1)})\|_{W\times U}\leq C_{1}(\|f_{\infty}\|_{\Hs}+\|f_{\infty}\|_{\PW})\leq C_{1}\varepsilon_{0}.
}
Suppose for some $n\geq1$, we have
\EQN{
\|(w^{(n)},u^{(n)})\|_{W\times U}\leq C_{1}\varepsilon_{0}.
}
By Duhamel's formula, we have
\EQN{
e^{itD^{\alpha}}u^{(n+1)}(t)=f_{\infty}
-\int_{t}^{\infty}e^{isD^{\alpha}}u^{(n)}(s)\overline{u^{(n)}}(s)ds,
}
\EQN{
e^{itD^{\alpha}}w^{(n+1)}(t)=f_{\infty}
-\int_{t}^{\infty}e^{isD^{\alpha}}Q(w^{(n)},u^{(n)},u^{(n-1)})(s)ds,
}
where $Q(w^{(n)},u^{(n)},u^{(n-1)})$ is the nonlinear term of \eqref{e6.2}.
By the proof of Lemma \ref{lem4.3}, there exists a $\delta'>0$ such that for all $t\geq1$,
\EQ{\label{e7.6}
\|e^{itD^{\alpha}}B(u^{(n)},\overline{u^{(n)}})(t)\|_{\Hs}\les t^{-\delta'}C_{1}^{2}\varepsilon^{2}_{0},
}
which means $\lim_{t\rightarrow+\infty}e^{itD^{\alpha}}w^{(n+1)}(t)=\lim_{t\rightarrow+\infty}e^{itD^{\alpha}}u^{(n+1)}(t)=f_{\infty}$ in $\Hs$.
By a similar way to the proof of Propositions \ref{prop5.1} and \ref{prop4.1}, we have
\EQN{
\|(w^{(n+1)},u^{(n+1)})\|_{W\times U}\leq C_{1}\varepsilon_{0},
}
if $C_{1}$ is large enough and $\varepsilon_{0}>0$ is small enough. This closes the induction. For trilinear estimates involving integration by parts in time, the boundary terms at infinity vanish due to the sufficiently rapid decay in time.  The proof of \eqref{e7.5} is similar to that of \eqref{e6.4} and we omit it here.
\end{proof}

\begin{lemma}\label{lem7.5}
Suppose that there exist two global solutions $u,\tilde{u}$ to \eqref{e7.1} in $C(\R;\Hs)$. Assume
\EQ{\label{e7.7}
 \sup_{t>0}(1+t)^{1+\delta'}\big(\|u(t)\|_{\infty}+\|\tilde{u}(t)\|_{\infty}\big)<\infty
}
for some $\delta'>0$. Then $u=\tilde{u}$.
\end{lemma}

\begin{proof}
Define $\tilde{v}(t):=\tilde{u}(-t)$ and $v(t):=u(-t)$.
Let $N\geq1$. Since $\tilde{v},v\in C(\R;\Hs)$, we can use \eqref{e7.7} and $\Hs\subseteq L^{\infty}$ to obtain
\EQN{
 \sup_{t\leq N}(1+|t|)^{1+\delta'}\big(\|\tilde{v}(t)\|_{\infty}+\|v(t)\|_{\infty}\big)\leq C_{\tilde{u},u,N}<\infty.
}
Because $u,\tilde{u}$ are solutions to \eqref{e7.1}, we have
\EQN{
\CAS{(\partial_{t}-iD^{\alpha})(\tilde{v}-v) = v\bar{v}-\tilde{v}\bar{\tilde{v}},\\
\lim_{t\rightarrow -\infty}e^{-itD^{\alpha}}(\tilde{v}-v)(t)=0\, \, in\, \, \Hs.
}
}
Following the same method as in the proof of Lemma \ref{lem6.2}, for all $t\leq N$ we have
\begin{align*}
\|\tilde{v}(t)-v(t)\|_{\Hs}\les C_{\tilde{u},u,N}\int_{-\infty}^{t}\|\tilde{v}(s)-v(s)\|_{\Hs}(1+|s|)^{-1-\delta'}ds.
\end{align*}
By Gronwall's inequality, we have 
$\tilde{v}(t)=v(t)$ for all $t\leq N$.
 Since $N$ is arbitrary, we have $\tilde{v}(t)=v(t)$ for all $t\in\R$.
\end{proof}

\begin{proof}[Proof of Theorem \ref{thm7.2}]
Let $(W,U)$ be given by \eqref{e4.4A}. Without loss of generality, we only consider the existence of solutions on $[0,\infty)$.
For $n\geq1$, by \eqref{e7.4}, \eqref{e6.2} and \eqref{e7.6}, we have
\EQN{
\CAS{(\partial_{t}+iD^\alpha)(w^{(n+1)}-w^{(n)})=Q(w^{(n)},u^{(n)},u^{(n-1)})-Q(w^{(n-1)},u^{(n-1)},u^{(n-2)}), \\
 u^{(n+1)}-u^{(n)}= w^{(n+1)}-w^{(n)}-iB(u^{(n)},\overline{ u^{(n)}})+iB(u^{(n-1)},\overline{ u^{(n-1)}}),
}
}
and $\lim_{t\rightarrow+\infty}e^{itD^{\alpha}}(w^{(n+1)}-w^{(n)})(t)=\lim_{t\rightarrow+\infty}e^{itD^{\alpha}}(u^{(n+1)}-u^{(n)})(t)=0$ in $\Hs$. By \eqref{e7.5} and an argument similar to that in the proof of Theorem \ref{thm:main}, we obtain that $\{(w^{(n)},u^{(n)})\}_{n\geq1}$ is a Cauchy sequence in $W\times U$. Then by Lemma \ref{lem7.4}, 
\EQN{
(w,u):=\lim_{n\rightarrow\infty}(w^{(n)},u^{(n)})\in C([0,\infty);\Hs)^{2},
}
where the limit is in $W\times U$ sense. By Lemma \ref{lem7.4} and a limit argument, \eqref{e7.3} follows. Since $\{(w^{(n)},u^{(n)})\}_{n\geq1}$ satisfy \eqref{e7.4} and \eqref{e6.2}, $(w,u)$ satisfies \eqref{e7.1} and \eqref{e2.5} by a limit argument. 
{By an argument analogous to the one used in the proof of the scattering property in Theorem \ref{thm:main}, \eqref{e7.2} is valid. }

Next we prove uniqueness. Suppose $(\tilde{w},\tilde{u})$ satisfies \eqref{e2.5}, \eqref{e7.2} and \eqref{e7.3}. Then we have
\EQN{
\CAS{(\partial_{t}+iD^\alpha)(\tilde{w}-w)=Q(\tilde{w},\tilde{u},\tilde{u})-Q(w,u,u), \\
\tilde{u}-u= \tilde{w}-w-iB(\tilde{u},\bar{ \tilde{u}})+iB(u,\bar{u}),
}
}
and $\lim_{t\rightarrow+\infty}e^{itD^{\alpha}}(\tilde{w}-w)(t)=\lim_{t\rightarrow+\infty}e^{itD^{\alpha}}(\tilde{u}-u)(t)=0$ in $\Hs$. By \eqref{e7.3} and an argument similar to that in Lemma \ref{lem7.4}, we have
\EQN{
\|(\tilde{w}-w,\tilde{u}-u)\|_{W\times U}\leq \frac{1}{2}\|(\tilde{w}-w,\tilde{u}-u)\|_{W\times U},
}
which means $\|(\tilde{w}-w,\tilde{u}-u)\|_{W\times U}=0$ and thus $(\tilde{w},\tilde{u})=(w,u)$. This proves the uniqueness of the final data problem associated with \eqref{e2.5}. The uniqueness of \eqref{e7.1} follows from Lemma \ref{lem7.5}.
\end{proof}

\section*{Declarations}

\noindent Conflicts of interest: none. 


\section*{Financial Support}
Z. Guo is the recipient of an Australian Research Council Future Fellowship (project number FT230100588) funded by the Australian Government. N. Liu and L. Song are supported by National Key R\&D Program of China 2022YFA1005700. N. Liu is supported by China Postdoctoral Science Foundation (No. 2024M763732). L. Song is supported by NNSF of China (No. 12471097).

\bibliographystyle{plain}

\end{document}